\documentclass[10pt,reqno]{amsart}
\usepackage[ams]{optional}

\usepackage{dsfont,amsmath,amsfonts,amscd,amssymb,graphicx,mathrsfs,eufrak,enumitem,euscript,upgreek,colonequals}
\usepackage[all,arc,curve,color,frame]{xy}
\usepackage[usenames]{color}

\usepackage{pifont}
\newcommand{\cmark}{\ding{51}}%
\newcommand{\xmark}{\ding{55}}

\usepackage{oubraces}

\usepackage{setspace}
\usepackage{array,multirow,booktabs,longtable}

\setlist[enumerate]{format=\normalfont}

\usepackage[colorlinks,hyperfootnotes=false]{hyperref}

\newcommand{\marginparstretch}{0.6}
\let\oldmarginpar\marginpar
\renewcommand\marginpar[1]{\-\oldmarginpar[\framebox{\setstretch{\marginparstretch}\begin{minipage}{\marginparwidth}{\raggedleft\tiny #1}\end{minipage}}]{\framebox{\setstretch{\marginparstretch}\begin{minipage}{\marginparwidth}{\raggedright\tiny #1}\end{minipage}}}}

\usepackage[colorlinks]{hyperref}
\usepackage{tikz,mathrsfs}
\usepackage{tikz-3dplot}
\usetikzlibrary{shapes,calc,positioning}
\usetikzlibrary{calc,intersections}
\usetikzlibrary{arrows,decorations.pathmorphing,decorations.pathreplacing,positioning,shapes.geometric,shapes.misc,decorations.markings,decorations.fractals,calc,patterns}

\tikzset{
        DB/.style={circle,draw=black,circle,fill=white,inner sep=0pt, minimum size=5pt},
        DW/.style={circle,draw=black,fill=black,inner sep=0pt, minimum size=5pt},
        cvertex/.style={circle,draw=black,fill=white,inner sep=1pt,outer sep=3pt},
        vertex/.style={circle,fill=black,inner sep=1pt,outer sep=3pt},
        star/.style={circle,fill=yellow,inner sep=0.75pt,outer sep=0.75pt},
        tvertex/.style={inner sep=1pt,font=\scriptsize},
	pvertex/.style={circle,inner sep=1pt,outer sep=2pt,font=\scriptsize},
        gap/.style={inner sep=0.5pt,fill=white}}

\tikzstyle{mybox} = [draw=black, fill=blue!10, very thick,
    rectangle, rounded corners, inner sep=10pt, inner ysep=20pt]
\tikzstyle{boxtitle} =[fill=blue!50, text=white,rectangle,rounded corners]

\newcommand{\rotateRPY}[3]
{   \pgfmathsetmacro{\rollangle}{#1}
    \pgfmathsetmacro{\pitchangle}{#2}
    \pgfmathsetmacro{\yawangle}{#3}

    \pgfmathsetmacro{\newxx}{cos(\yawangle)*cos(\pitchangle)}
    \pgfmathsetmacro{\newxy}{sin(\yawangle)*cos(\pitchangle)}
    \pgfmathsetmacro{\newxz}{-sin(\pitchangle)}
    \path (\newxx,\newxy,\newxz);
    \pgfgetlastxy{\nxx}{\nxy};

    \pgfmathsetmacro{\newyx}{cos(\yawangle)*sin(\pitchangle)*sin(\rollangle)-sin(\yawangle)*cos(\rollangle)}
    \pgfmathsetmacro{\newyy}{sin(\yawangle)*sin(\pitchangle)*sin(\rollangle)+ cos(\yawangle)*cos(\rollangle)}
    \pgfmathsetmacro{\newyz}{cos(\pitchangle)*sin(\rollangle)}
    \path (\newyx,\newyy,\newyz);
    \pgfgetlastxy{\nyx}{\nyy};

    \pgfmathsetmacro{\newzx}{cos(\yawangle)*sin(\pitchangle)*cos(\rollangle)+ sin(\yawangle)*sin(\rollangle)}
    \pgfmathsetmacro{\newzy}{sin(\yawangle)*sin(\pitchangle)*cos(\rollangle)-cos(\yawangle)*sin(\rollangle)}
    \pgfmathsetmacro{\newzz}{cos(\pitchangle)*cos(\rollangle)}
    \path (\newzx,\newzy,\newzz);
    \pgfgetlastxy{\nzx}{\nzy};
}

\tikzset{RPY/.style={x={(\nxx,\nxy)},y={(\nyx,\nyy)},z={(\nzx,\nzy)}}}

  \tikzset{current plane/.prefix style={scale=\R}}

\newcommand{\arrow}[2][20]
 {
  \hspace{-5pt}
  \begin{tikzpicture}
   \node (A) at (0,0) {};
   \node (B) at (#1pt,0) {};
   \draw [#2] (A) -- (B);
  \end{tikzpicture}
  \hspace{-5pt}
 }

\newcommand{\equalsSep}{0.8pt}
\newcommand{\pad}[1]{\;{#1}\;}

\opt{ams}{\addtolength{\hoffset}{-0.5cm} \addtolength{\textwidth}{1cm}
\addtolength{\voffset}{-1.5cm} \addtolength{\textheight}{2cm}}

\newtheorem{thm}{Theorem}[section]
\newtheorem{prop}[thm]{Proposition}
\newtheorem{lemma}[thm]{Lemma}
\newtheorem{defin}[thm]{Definition}
\newtheorem{cor}[thm]{Corollary}

\theoremstyle{definition}

\newtheorem{remark}[thm]{Remark}

\newtheorem{notation}[thm]{Notation}

\newcommand{\step}[1]{\medskip\emph{Step #1}:}

\numberwithin{equation}{section}

\newcommand\citetype[1]{}

\newcommand{\m}{\mathfrak{m}}

\newcommand{\p}{\mathfrak{p}}

\renewcommand{\t}[1]{\textnormal{#1}}

\def\Auteq{\mathop{\rm Auteq}\nolimits}
\def\Cl{\mathop{\rm Cl}\nolimits}

\def\op{\mathop{\rm op}\nolimits}

\def\CM{\mathop{\rm CM}\nolimits}

\def\depth{\mathop{\rm depth}\nolimits}

\def\mod{\mathop{\rm mod}\nolimits}

\def\coh{\mathop{\rm coh}\nolimits}

\def\refl{\mathop{\rm ref}\nolimits}

\def\pd{\mathop{\rm pd}\nolimits}

\def\Hom{\mathop{\rm Hom}\nolimits}

\def\RHom{\mathop{\rm {\bf R}Hom}\nolimits}
\def\End{\mathop{\rm End}\nolimits}
\def\Ext{\mathop{\rm Ext}\nolimits}

\def\add{\mathop{\rm add}\nolimits}

\def\rank{\mathop{\rm rank}\nolimits}

\def\Supp{\mathop{\rm Supp}\nolimits}

\def\Spec{\mathop{\rm Spec}\nolimits}

\def\Perf{\mathop{\rm{Perf}}\nolimits}

\def\op{\rm{op}}

\def\Db{\mathop{\rm{D}^b}\nolimits}

\def\flop{{\sf{F}}}

\def\Id{\mathop{\rm{Id}}\nolimits}

\newcommand{\deform}{\mathrm{def}}

\newcommand{\con}{\mathrm{con}}

\newcommand{\CA}{\mathrm{A}_{\con}}

\def\ab{\mathop{\rm ab}\nolimits}

\def\Rf{{\rm\bf R}f}

\def\RHom{{\rm{\bf R}Hom}}

\def\sEnd{\scrE\mathrm{nd}}
\def\RsHom{{\bf R}\scrH\mathrm{om}}
\newcommand\RDerived[1]{{\rm\bf R}{#1}}

\newcommand\Rfi[1]{{\rm\bf R}^{#1}f}

\newcommand\art{\mathsf{Art}}

\newcommand\Sets{\mathsf{Sets}}

\newcommand\Def{\scrD\mathrm{ef}}

\newcommand{\cM}{\mathcal{M}}

\newcommand{\Per}{{}^{0}\scrP\mathrm{er}}
\newcommand{\PerOne}{{}^{-1}\scrP\mathrm{er}}

\newcommand\Curve{\mathrm{C}}

\newcommand{\cStab}[1]{\mathrm{Stab}_{#1}^{\kern -0.5pt \circ}\kern -0.25pt}
\newcommand{\Stab}[1]{\mathrm{Stab}_{#1}}
\newcommand{\cAut}{\mathrm{Aut}^{\kern 0pt \circ}\kern -0.25pt}

\newcommand{\GV}[1]{\mathrm{GV}_{#1}}

\newcommand\twistGen{{\sf Twist}}

\newcommand\placeholder{-}

\newlength\tempWidth

\newcommand{\scrA}{\EuScript{A}}
\newcommand{\scrB}{\EuScript{B}}
\newcommand{\scrC}{\EuScript{C}}
\newcommand{\scrD}{\EuScript{D}}
\newcommand{\scrE}{\EuScript{E}}
\newcommand{\scrF}{\EuScript{F}}
\newcommand{\scrG}{\EuScript{G}}
\newcommand{\scrH}{\EuScript{H}}
\newcommand{\scrI}{\EuScript{I}}

\newcommand{\scrK}{\EuScript{K}}
\newcommand{\scrL}{\EuScript{L}}
\newcommand{\scrM}{\EuScript{M}}
\newcommand{\scrN}{\EuScript{N}}
\newcommand{\scrO}{\EuScript{O}}
\newcommand{\scrP}{\EuScript{P}}

\newcommand{\scrR}{\EuScript{R}}
\newcommand{\scrS}{\EuScript{S}}
\newcommand{\scrT}{\EuScript{T}}

\newcommand{\scrV}{\EuScript{V}}

\newcommand{\scrZ}{\EuScript{Z}}

\newcommand{\righttilt}{\scrR}
\newcommand{\lefttilt}{\scrL}
\newcommand{\Tilt}{\scrT\mathrm{ilt}}
\newcommand{\finitelength}{\mathfrak{fl}\,}

\def\spiralheight{600}
\def\scaledots{1.2}

\newcommand\picHelix[2]{
\def\type{#1}
\def\scale{#2}
\def\typeOriginal{1}
\def\typeModified{2}
\ifx\type\typeOriginal
\def\rotation{0}
\def\spiralrotation{0}
\def\inner{1.5cm}
\def\mid{2cm}
\def\midlabels{1.75cm}
\def\outerlabels{2.6cm}
\def\outer{3.25cm}
\def\flip{1}
\else
\def\rotation{-50} 
\def\spiralrotation{68}
\def\inner{1.3cm}
\def\mid{2.2cm}
\def\midlabels{1.75cm}
\def\outerlabels{2.9cm}
\def\outer{3.5cm}
\def\flip{1} 
\fi
\tdplotsetmaincoords{70}{15}
\begin{array}{ccc}
\begin{array}{c}
\begin{tikzpicture}[tdplot_main_coords,>=stealth,scale=2.5]
\draw[black] ({sin(\spiralrotation)},{cos(\spiralrotation)},0) 
    \foreach \t in {2,3,...,360}
    {
        --({sin(\t+\spiralrotation)},{cos(\t+\spiralrotation)},{\t/\spiralheight})
    };
\draw[black,thin,line cap=round,dash pattern=on 0pt off 2.5\pgflinewidth] (-{sin(-\spiralrotation)},{cos(-\spiralrotation)},0) 
    \foreach \t in {2,3,...,540}
    {
        --(-{sin(\t-\spiralrotation)},{cos(\t-\spiralrotation)},-{\t/\spiralheight})
    };
\draw[black,thin,line cap=round,dash pattern=on 0pt off 2.5\pgflinewidth] ({sin(\spiralrotation)},{cos(\spiralrotation)},{360/\spiralheight}) 
    \foreach \t in {2,3,...,540}
    {
        --({sin(\t+\spiralrotation)},{cos(\t+\spiralrotation)},{\t/\spiralheight + 360/\spiralheight})
    };
solid segments
\foreach \t in {0,20,40,...,360}
  {
  \draw (0,0,0.2) -- ({sin(\t+\spiralrotation)},{cos(\t+\spiralrotation)},{\t/\spiralheight});
  }; 
\foreach \t in {0,20,40,...,360}
  {
  \draw[line cap=round,dash pattern=on 0pt off 2.5\pgflinewidth,thin] (0,0,0.2) -- (-{sin(\t-\spiralrotation)},{cos(\t-\spiralrotation)},{-\t/\spiralheight});
  }; 
\foreach \t in {0,20,40,...,360}
  {
  \draw[line cap=round,dash pattern=on 0pt off 2.5\pgflinewidth,thin] (0,0,0.2) -- ({sin(\t+\spiralrotation)},{cos(\t+\spiralrotation)},{\t/\spiralheight + 360/\spiralheight});
  }; 
\draw[red] (0,0,0.2) -- ({sin(\spiralrotation)},{cos(\spiralrotation)},{360/\spiralheight}); 
\draw[red] (0,0,0.2) -- ({sin(\spiralrotation)},{cos(\spiralrotation)},0); 
\draw[blue,->] ({1.1*sin(240)},{1.1*cos(240)},{240/\spiralheight}) -- node[left] {$\scriptstyle \otimes\scrO(1)$} ({1.1*sin(240)},{1.1*cos(240)},{240/\spiralheight+360/\spiralheight});
\end{tikzpicture}
\end{array}
\opt{ip}{\vspace{0.2cm} \\}\opt{ams}{&}
\opt{ip}{\hspace{1.2cm}}
\begin{array}{c}
\begin{tikzpicture}[scale=\scale,rotate=-40]
\draw (0,0) circle (\inner);
\draw (0,0) circle (\mid);
\draw (0,0) circle (\outer);

\node (A) at (90*\flip+\rotation:1cm) {$\scrA_0$};
\node (A1) at (54*\flip+\rotation:1cm) {$\scrA_1$};
\node (A2) at (18*\flip+\rotation:1cm) {$\scrA_2$};
\node (A3) at (-18*\flip+\rotation:1cm) {$\scrA_3$};
\node (A4) at (-54*\flip+\rotation:1cm) {$\scrA_4$};
\node (A5) at (-90*\flip+\rotation:1cm) {$\scrA_5$};
\node (A6) at (-126*\flip+\rotation:1cm) {$\scrA_6$};
\node (A7) at (-162*\flip+\rotation:1cm) {$\scrA_7$};
\node (A8) at (-198*\flip+\rotation:1cm) {$\scrA_8$};
\node (A9) at (-234*\flip+\rotation:1cm) {$\scrA_9$};
\node (P) at (99*\flip+\rotation:\midlabels) {$\scrV_{0}$};
\node[rotate=\ifx\type\typeOriginal 99 \else 0 \fi] at (99*\flip+\rotation:\outerlabels) {$\scriptstyle \omega_{5\Curve}[1]$};
\node (Q) at (81*\flip+\rotation:\midlabels) {$\scrV_{\kern -2pt -1}$};
\node[rotate=\ifx\type\typeOriginal 81 \else -10 \fi] at (81*\flip+\rotation:\outerlabels) {$\scriptstyle\scrO_{\Curve}(-1)$};
\node (P1) at (62*\flip+\rotation:\midlabels) {$\scrV_1$};
\node[rotate=\ifx\type\typeOriginal 62 \else -28 \fi] at (62*\flip+\rotation:\outerlabels+-1) {$\scriptstyle\scrO_{\Curve}(-1)[1]$};
\node (Q1) at (44*\flip+\rotation:\midlabels) {$\scrV_0$};
\node[rotate=\ifx\type\typeOriginal 0 \else 0 \fi] at (44*\flip+\rotation:\outerlabels) {$\scriptstyle \scrO_{5\Curve}$};
\node (P2) at (8*\flip+\rotation:\midlabels) {$\scrV_1$};
\node at (8*\flip+\rotation:\outerlabels) {$\scriptstyle \scrO_{4\Curve}$};
\node (Q2) at (26*\flip+\rotation:\midlabels) {$\scrV_2$};
\node at (27*\flip+\rotation:\outerlabels) {$\scriptstyle \scrO_{5\Curve}[1]$};
\node (P3) at (-10*\flip+\rotation:\midlabels) {$\scrV_3$};
\node at (-9*\flip+\rotation:\outerlabels) {$\scriptstyle \scrO_{4\Curve}[1]$};
\node (Q3) at (-28*\flip+\rotation:\midlabels) {$\scrV_2$};
\node at (-28*\flip+\rotation:\outerlabels) {$\scriptstyle \scrO_{3\Curve}$};
\node (P4) at (-46*\flip+\rotation:\midlabels) {$\scrV_4$};
\node at (-45*\flip+\rotation:\outerlabels) {$\scriptstyle \scrO_{3\Curve}[1]$};
\node (Q4) at (-62*\flip+\rotation:\midlabels) {$\scrV_3$};
\node at (-62*\flip+\rotation:\outerlabels) {$\scriptstyle \scrZ$};
\node (P5) at (-81*\flip+\rotation:\midlabels) {$\scrV_5$};
\node at (-81*\flip+\rotation:\outerlabels) {$\scriptstyle \scrZ[1]$};
\node (Q5) at (-99*\flip+\rotation:\midlabels) {$\scrV_4$};
\node at (-99*\flip+\rotation:\outerlabels) {$\scriptstyle \scrO_{2\Curve}$};
\node (P6) at (-116*\flip+\rotation:\midlabels) {$\scrV_6$};
\node at (-117*\flip+\rotation:\outerlabels) {$\scriptstyle \scrO_{2\Curve}[1]$};
\node (Q6) at (-134*\flip+\rotation:\midlabels) {$\scrV_5$};
\node at (-135*\flip+\rotation:\outerlabels) {$\scriptstyle \scrZ^\omega\!(1)$};
\node (P7) at (-152*\flip+\rotation:\midlabels) {$\scrV_7$};
\node at (-152.5*\flip+\rotation:\outerlabels) {$\scriptstyle \scrZ^\omega\!(1)\![1]$};
\node (Q7) at (-170*\flip+\rotation:\midlabels) {$\scrV_6$};
\node[rotate=\ifx\type\typeOriginal 0 \else 0 \fi] at (-171*\flip+\rotation:\outerlabels) {$\scriptstyle \omega_{3\Curve}(1)$};
\node (P8) at (-190*\flip+\rotation:\midlabels) {$\scrV_8$};
\node[rotate=\ifx\type\typeOriginal -10 \else 45 \fi] at (-190*\flip+\rotation:\outerlabels) {$\scriptstyle \omega_{3\Curve}(1)[1]$};
\node (Q8) at (-206*\flip+\rotation:\midlabels) {$\scrV_7$};
\node[rotate=\ifx\type\typeOriginal -26 \else 0 \fi] at (-208*\flip+\rotation:\outerlabels) {$\scriptstyle \omega_{4\Curve}(1)$};
\node (P9) at (-226*\flip+\rotation:\midlabels) {$\scrV_9$};
\node[rotate=\ifx\type\typeOriginal -46 \else 10 \fi] at (-226*\flip+\rotation:\outerlabels) {$\scriptstyle \omega_{4\Curve}(1)[1]$};
\node (P10) at (-242*\flip+\rotation:\midlabels) {$\scrV_8$};
\node[rotate=\ifx\type\typeOriginal -63 \else 0 \fi] at (-243*\flip+\rotation:\outerlabels) {$\scriptstyle \omega_{5\Curve}(1)$};
\draw (0*\flip+\rotation:0) -- (72*\flip+\rotation:\outer) ;
\draw (0*\flip+\rotation:0) -- (36*\flip+\rotation:\outer) ;
\draw (0*\flip+\rotation:0) -- (0*\flip+\rotation:\outer) ;
\draw (0*\flip+\rotation:0) -- (-36*\flip+\rotation:\outer) ;
\draw (0*\flip+\rotation:0) -- (-72*\flip+\rotation:\outer) ;
\draw (0*\flip+\rotation:0) -- (-108*\flip+\rotation:\outer) ;
\draw (0*\flip+\rotation:0) -- (-144*\flip+\rotation:\outer) ;
\draw (0*\flip+\rotation:0) -- (-180*\flip+\rotation:\outer) ;
\draw (0*\flip+\rotation:0) -- (-216*\flip+\rotation:\outer) ;
\draw[red] (0*\flip+\rotation:0) -- (-252*\flip+\rotation:\outer) ;
\draw (90*\flip+\rotation:\inner) -- (90*\flip+\rotation:\outer) ;
\draw (54*\flip+\rotation:\inner) -- (54*\flip+\rotation:\outer) ;
\draw (18*\flip+\rotation:\inner) -- (18*\flip+\rotation:\outer) ;
\draw (-18*\flip+\rotation:\inner) -- (-18*\flip+\rotation:\outer) ;
\draw (-54*\flip+\rotation:\inner) -- (-54*\flip+\rotation:\outer) ;
\draw (-90*\flip+\rotation:\inner) -- (-90*\flip+\rotation:\outer) ;
\draw (-126*\flip+\rotation:\inner) -- (-126*\flip+\rotation:\outer) ;
\draw (-162*\flip+\rotation:\inner) -- (-162*\flip+\rotation:\outer) ;
\draw (-198*\flip+\rotation:\inner) -- (-198*\flip+\rotation:\outer) ;
\draw (-234*\flip+\rotation:\inner) -- (-234*\flip+\rotation:\outer) ;
\node (A) at (90*\flip+\rotation:1cm) {$\scrA_0$};
\node (A1) at (54*\flip+\rotation:1cm) {$\scrA_1$};
\node (A2) at (18*\flip+\rotation:1cm) {$\scrA_2$};
\node (A3) at (-18*\flip+\rotation:1cm) {$\scrA_3$};
\node (A4) at (-54*\flip+\rotation:1cm) {$\scrA_4$};
\node (A5) at (-90*\flip+\rotation:1cm) {$\scrA_5$};
\node (A6) at (-126*\flip+\rotation:1cm) {$\scrA_6$};
\node (A7) at (-162*\flip+\rotation:1cm) {$\scrA_7$};
\node (A8) at (-198*\flip+\rotation:1cm) {$\scrA_8$};
\node (A9) at (-234*\flip+\rotation:1cm) {$\scrA_9$};
\end{tikzpicture}
\end{array}
\end{array}
}

\newcommand\segment[3]{to [out=#1,in=#3+180] (#2)}
\newcommand\circsegment[3]{arc (#1+90:#3+90:#2 and #2)}

\newcommand\sphTwoPoint[1]{
\def\type{#1}
\def\typeIntro{1}
\def\typeOther{2}
\def\typeProof{3}
\ifx\type\typeIntro
\def\spherescale{1.35 }
\def\Alabel{$\scriptstyle \scrS'_{N/2-1}$} %
\def\Blabel{$\scriptstyle \scrS'_1$} %
\def\Dlabel{$\scriptstyle \scrS_1$} %
\def\Elabel{$\scriptstyle \scrS_{N/2}$} %
\def\toplabel{$\scriptstyle \otimes\scrO_{\kern -1pt X}\!(1)$}
\def\toplabelx{0.3}
\def\toplabely{0.75}
\def\bottomlabel{$\scriptstyle \otimes\scrO_{\kern -1pt X^+}\!(1)$}
\def\bottomlabelx{0.25}
\def\bottomlabely{-0.8}
\def\leftdotspos{0.68}
\else 
\ifx\type\typeOther
\def\spherescale{1.3 }
\def\Alabel{$\scriptstyle \scrO_{3\Curve'}$}
\def\Blabel{$\scriptstyle \scrO_{\ell\Curve'}$}
\def\Dlabel{$\scriptstyle \scrO_{\ell\Curve}$}
\def\Elabel{$\scriptstyle \scrO_{2\Curve}$}
\def\toplabel{$\scriptstyle \otimes\scrO_{\kern -1pt X}\!(1)$}
\def\toplabelx{0.3}
\def\toplabely{0.75}
\def\bottomlabel{$\scriptstyle \otimes\scrO_{\kern -1pt X^+}\!(1)$}
\def\bottomlabelx{0.25}
\def\bottomlabely{-0.8}
\def\leftdotspos{0.68}
\else 
\def\spherescale{1.3 }  
\def\Alabel{$\scriptstyle c_{N/2-1}$}
\def\Blabel{$\scriptstyle c_1$}
\def\Dlabel{$\scriptstyle b_1$}
\def\Elabel{$\scriptstyle b_{N/2}$}
\def\toplabel{$\scriptstyle a$}
\def\toplabelx{0.2}
\def\toplabely{0.8}
\def\bottomlabel{$\scriptstyle d$}
\def\bottomlabelx{0.15}
\def\bottomlabely{-0.8}
\def\leftdotspos{0.68}
\fi
\fi
\def\vertspherescale{0.27}
\begin{tikzpicture}[>=stealth,scale=2]
\draw[thick] ([shift=(-84:\spherescale)]0,0) arc (-84:84:\spherescale)  
   [bend left] to (96:\spherescale)
   arc (96:264:\spherescale)
   [bend left] to cycle;
\draw[\colorfront,line cap=round,dash pattern=on 0pt off 3\pgflinewidth] (\spherescale,0) arc (0:180:\spherescale and \vertspherescale);
\draw[\colorfront] (\spherescale,0) arc (0:-180:\spherescale and \vertspherescale)
coordinate[pos=0.8] (A) coordinate[pos=\leftdotspos] (leftdots) coordinate[pos=0.63] (B) coordinate[pos=0.53] (C) coordinate[pos=0.43] (D) coordinate[pos=0.33] (rightdots) coordinate[pos=0.2] (E);
\filldraw[fill=white,draw=black] (A) circle (1.5pt);
\filldraw[fill=white,draw=black] (B) circle (1.5pt);
\filldraw[fill=white,draw=black] (C) circle (1.5pt);
\filldraw[fill=white,draw=black] (D) circle (1.5pt);
\filldraw[fill=white,draw=black] (E) circle (1.5pt);
\node [rotate=-8] (l) at (leftdots) [shift=(-70:0.6)] {\scalebox{0.8}{$\cdots$}};
\node [rotate=5] (r) at (rightdots) [shift=(100:0.7)] {\scalebox{\scaledots}{$\cdots$}};
\node (Alabel) at ($(A)+(0.02,0)$) [above=0.25] {\Alabel};
\node (Blabel) at ($(B)+(0.05,0)$) [above=0.25] {\Blabel};
\node (Dlabel) at (D) [below=0.25] {\Dlabel};
\node (Elabel) at (E) [below=0.25] {\Elabel};
\node (p) at (-0.1*\spherescale,0.6*\spherescale) {};
\node (q) at (-0.1*\spherescale,-0.7*\spherescale) {};
\node[color=\colorlinebundletop] at (\toplabelx*\spherescale,\toplabely*\spherescale) {\toplabel};
\draw[\monodromystyle,,color=\colorlinebundletop,bend left,looseness=0.5] 
(p) to ([shift=(100:\spherescale)]0,0);
\draw[\monodromystyle,color=\colorlinebundletop,bend right,line cap=round,dash pattern=on 0pt off 2.5\pgflinewidth] ([shift=(100:\spherescale)]0,0) to ([shift=(80:\spherescale)]0,0);
\draw[\monodromystyle,color=\colorlinebundletop, looseness=1.0,decoration={
    markings,
    mark=at position 0.5 with {\arrow{>}}},postaction=decorate] 
([shift=(78:\spherescale)]0,0)
   \segment{-120}{[shift=(85:0.85*\spherescale)]0,0}{-150}
   \segment{-150}{p}{-110};
\draw[\monodromystyle,color=\colortwist, looseness=0.8, decoration={markings,
    mark=at position 0.54 with {\arrow[rotate=15]{>}}},  postaction=decorate] ($(p)+(100:0.2pt)$)
  \segment{10}{$(E)+(45:3.5pt)$}{-45} 
  \circsegment{-45}{3.5pt}{-260}
  \segment{-260}{$(p)+(100:-0.2pt)$}{190};

\draw[\monodromystyle, color=\colortwist, looseness=0.8, decoration={markings,
    mark=at position 0.54 with {\arrow[rotate=15]{>}}}, postaction=decorate] ($(p)+(80:0.2pt)$)
  \segment{-10}{$(D)+(25:3.5pt)$}{-65} 
  \circsegment{-65}{3.5pt}{-285}
  \segment{-285}{$(p)+(80:-0.2pt)$}{170};

\draw[\monodromystyle,color=\colortwist, looseness=0.8, decoration={markings,
    mark=at position 0.54 with {\arrow[rotate=15]{>}}},  postaction=decorate] ($(q)+(100+180:0.2pt)$)
  \segment{10+180}{$(A)+(45+180:3.5pt)$}{-45+180} 
  \circsegment{-45+180}{3.5pt}{-260+180}
  \segment{-260+180}{$(q)+(100:-0.2pt)$}{190+180};

\draw[\monodromystyle, color=\colortwist, looseness=0.8, decoration={markings,
    mark=at position 0.54 with {\arrow[rotate=15]{>}}}, postaction=decorate] ($(q)+(80+180:0.2pt)$)
  \segment{-10+180}{$(B)+(25+180:3.5pt)$}{-65+180} 
  \circsegment{-65+180}{3.5pt}{-285+180}
  \segment{-285+180}{$(q)+(80:-0.2pt)$}{170+180};

 \draw[\monodromystyle,color=\colorflop,bend left,looseness=0.5,decoration={
    markings,
    mark=at position 0.8 with {\arrow{<}}},postaction=decorate] 
(p) to node[right,pos=0.8] {$\scriptstyle \flop$} (q);
 \draw[\monodromystyle,color=\colorflop,bend left,looseness=0.5,decoration={
    markings,
    mark=at position 0.7 with {\arrow{<}}},postaction=decorate] 
(q) to node[left,pos=0.7] {$\scriptstyle \flop$} (p);

\node[color=\colorlinebundlebottom] at (\bottomlabelx*\spherescale,\bottomlabely*\spherescale) {\bottomlabel};
 \draw[\monodromystyle,color=\colorlinebundlebottom, looseness=1.2] 
(q)
\segment{-60}{[shift=(-80:\spherescale)]0,0}{-60};
\draw[\monodromystyle,color=\colorlinebundlebottom,bend left,line cap=round,dash pattern=on 0pt off 2.5\pgflinewidth] ([shift=(-100:\spherescale)]0,0) to ([shift=(-80:\spherescale)]0,0);
 \draw[\monodromystyle,color=\colorlinebundlebottom,bend left,looseness=1.3,decoration={
    markings,
    mark=at position 0.5 with {\arrow[rotate=15]{>}}},postaction=decorate] 
([shift=(-100:\spherescale)]0,0) to (q);
\filldraw[color=white] (p) circle (3pt);
\filldraw (p) circle (1pt);
\filldraw[color=white] (q) circle (3pt);
\filldraw (q) circle (1pt);
\draw[thick] ([shift=(-84:\spherescale)]0,0) arc (-84:84:\spherescale)  
[bend left] to (96:\spherescale)
arc (96:264:\spherescale)
[bend left] to cycle;
\end{tikzpicture}
}

\newcommand\sphOnePoint[1]{
\def\type{#1}
\def\typeIntro{1}
\def\typeAlg{2}
\def\typeGeom{3}
\def\colorlinebundletop{green!60!black}
\def\colorlinebundlebottom{red}
\def\colorflop{blue}
\def\colortwist{black}
\def\colorfront{black!20}
\def\colorback{black!10}
\def\monodromystyle{semithick}
\ifx\type\typeIntro
\def\topArrowDirection{>}
\def\topArrowLabel{$\scriptstyle \otimes\scrO_{\kern -1pt X}\!(1)$}
\def\bottomArrow{\arrow[rotate=-5]{<}}
\def\bottomArrowLabel{$\scriptstyle \flop^{-1}\big(\otimes\scrO_{\kern -1pt X^+}\!(1)\big)\flop$}
\def\bottomArrowLabelx{0.2}
\def\bottomArrowLabely{-0.65}
\def\Clabel{$\scriptstyle \scrS_0$}
\def\Dlabel{$\scriptstyle \scrS_1$}
\def\Elabel{$\scriptstyle \scrS_{N-1}$}
\else 
\ifx\type\typeAlg
\def\topArrowDirection{<} 
\def\topArrowLabel{$\scriptstyle q_-$}
\def\bottomArrow{\arrow[rotate=+5]{>}}
\def\bottomArrowLabel{$\scriptstyle q_+$}
\def\bottomArrowLabelx{0.15}
\def\bottomArrowLabely{-0.75}
\def\Clabel{$\scriptstyle q_0$}
\def\Dlabel{$\scriptstyle q_1$}
\def\Elabel{$\scriptstyle q_{N-1}$}
\else 
\def\topArrowDirection{<} 
\def\topArrowLabel{$\scriptstyle \otimes\scrO_{\kern -1pt  X}\!(-1)$}
\def\bottomArrow{\arrow[rotate=+5]{>}}
\def\bottomArrowLabel{$\scriptstyle \flop^{-1}\big(\otimes\scrO_{\kern -1pt X^+}\!(-1)\big)\flop$}
\def\bottomArrowLabelx{0.2}
\def\bottomArrowLabely{-0.65}
\def\Clabel{$\scriptstyle \scrS_0$}
\def\Dlabel{$\scriptstyle \scrS_1$}
\def\Elabel{$\scriptstyle \scrS_{N-1}$}
\fi
\fi
\def\colorlinebundlebottom{red}
\def\colorflop{blue}
\def\colortwist{black}
\def\colorfront{black!20}
\def\colorback{black!10}
\def\monodromystyle{semithick}
\def\vertspherescale{0.27}
\def\spherescale{1.15 }
\begin{tikzpicture}[>=stealth,scale=2]
\draw[thick] ([shift=(-84:\spherescale)]0,0) arc (-84:84:\spherescale)  
   [bend left] to (96:\spherescale)
   arc (96:264:\spherescale)
   [bend left] to cycle;
\draw[\colorfront,line cap=round,dash pattern=on 0pt off 3\pgflinewidth] (\spherescale,0) arc (0:180:\spherescale and \vertspherescale);
\draw[\colorfront] (\spherescale,0) arc (0:-180:\spherescale and \vertspherescale)
coordinate[pos=0.64] (C) coordinate[pos=0.52] (D) coordinate[pos=0.4] (rightdots) coordinate[pos=0.25] (E);
\filldraw[fill=white,draw=black] (C) circle (1.5pt);
\filldraw[fill=white,draw=black] (D) circle (1.5pt);
\filldraw[fill=white,draw=black] (E) circle (1.5pt);
\node [rotate=13] (r) at (rightdots) [shift=(100:0.7)] {\scalebox{\scaledots}{$\cdots$}};
\node (Clabel) at (C) [below=0.25] {\Clabel};
\node (Dlabel) at (D) [below=0.25] {\Dlabel};
\node (Elabel) at (E) [below=0.25] {\Elabel};
\node (p) at (-0.6*\spherescale,0.5*\spherescale) {};
\node[color=\colorlinebundletop] at (0.1*\spherescale,0.75*\spherescale) {\topArrowLabel};
\draw[\monodromystyle,,color=\colorlinebundletop,bend left,looseness=0.5] 
(p) to ([shift=(100:\spherescale)]0,0);
\draw[\monodromystyle,color=\colorlinebundletop,bend right,line cap=round,dash pattern=on 0pt off 2.5\pgflinewidth] ([shift=(100:\spherescale)]0,0) to ([shift=(80:\spherescale)]0,0);
\draw[\monodromystyle,color=\colorlinebundletop, looseness=1.0,decoration={
    markings,
    mark=at position 0.5 with {\arrow{\topArrowDirection}}},postaction=decorate] 
([shift=(75:\spherescale)]0,0)
   \segment{-120}{[shift=(100:0.85*\spherescale)]0,0}{-170}
   \segment{-170}{p}{-125};
\draw[\monodromystyle,color=\colortwist, looseness=0.8, decoration={markings,
    mark=at position 0.54 with {\arrow[rotate=15]{>}}},  postaction=decorate] ($(p)+(100:0.2pt)$)
  \segment{10}{$(E)+(45:3.5pt)$}{-45} 
  \circsegment{-45}{3.5pt}{-260}
  \segment{-260}{$(p)+(100:-0.2pt)$}{190};

\draw[\monodromystyle, color=\colortwist, looseness=0.8, decoration={markings,
    mark=at position 0.54 with {\arrow[rotate=15]{>}}}, postaction=decorate] ($(p)+(80:0.2pt)$)
  \segment{-10}{$(D)+(25:3.5pt)$}{-65} 
  \circsegment{-65}{3.5pt}{-285}
  \segment{-285}{$(p)+(80:-0.2pt)$}{170}; 

\draw[\monodromystyle, color=\colortwist, looseness=0.8, decoration={markings,
    mark=at position 0.54 with {\arrow[rotate=15]{>}}}, postaction=decorate] ($(p)+(0:0.2pt)$) 
  \segment{-90}{$(C)+(5:3.5pt)$}{-85} 
  \circsegment{-85}{3.5pt}{-305}
  \segment{-305}{$(p)+(0:-0.2pt)$}{90};

\def\circRadius{25pt}
\def\circRadiusShift{25pt}
\coordinate (circCentre) at ($(p)+(-44:\circRadius)$);
\coordinate (circCentreShift) at ($(p)+(-50:\circRadius)$);
\node[color=\colorlinebundlebottom] at (\bottomArrowLabelx*\spherescale,\bottomArrowLabely*\spherescale) {\bottomArrowLabel};
 \draw[\monodromystyle,color=\colorlinebundlebottom, looseness=1.2] 
($(p)+(-55:0.2pt)$)
\segment{-145}{$(circCentre)+(150:\circRadius)$}{-120}
\circsegment{60}{\circRadius}{150}
\segment{-30}{[shift=(-80:\spherescale)]0,0}{-50};
\draw[\monodromystyle,color=\colorlinebundlebottom,bend left,line cap=round,dash pattern=on 0pt off 2.5\pgflinewidth] ([shift=(-100:\spherescale)]0,0) to ([shift=(-80:\spherescale)]0,0);
 \draw[\monodromystyle,color=\colorlinebundlebottom,looseness=1.2,decoration={
    markings,
    mark=at position 0.95 with {\bottomArrow}},postaction=decorate] 
($(p)+(-55:-0.2pt)$)
\segment{-145}{$(circCentreShift)+(150:\circRadiusShift)$}{-120}
\circsegment{60}{\circRadiusShift}{150}
\segment{-30}{[shift=(-100:\spherescale)]0,0}{-50};
\filldraw[color=white] (p) circle (3pt);
\filldraw (p) circle (1pt);
\draw[thick] ([shift=(-84:\spherescale)]0,0) arc (-84:84:\spherescale)  
[bend left] to (96:\spherescale)
arc (96:264:\spherescale)
[bend left] to cycle;
\end{tikzpicture}
}

\newcommand\sphLengthOne[1]{
\def\colorlinebundletop{green!60!black}
\def\colorlinebundlebottom{red}
\def\colorflop{blue}
\def\colortwist{black}
\def\colorfront{black!20}
\def\colorback{black!10}
\def\monodromystyle{semithick}
\def\vertspherescale{0.20}
\def\spherescale{0.85 } 
\begin{tikzpicture}[>=stealth,scale=2]
\draw[thick] ([shift=(-84:\spherescale)]0,0) arc (-84:84:\spherescale)  
   [bend left] to (96:\spherescale)
   arc (96:264:\spherescale)
   [bend left] to cycle;
\draw[\colorfront,line cap=round,dash pattern=on 0pt off 3\pgflinewidth] (\spherescale,0) arc (0:180:\spherescale and \vertspherescale);
\draw[\colorfront] (\spherescale,0) arc (0:-180:\spherescale and \vertspherescale)
coordinate[pos=0.8] (A) coordinate[pos=0.7] (leftdots) coordinate[pos=0.63] (B) coordinate[pos=0.53] (C) coordinate[pos=0.43] (D) coordinate[pos=0.31] (rightdots) coordinate[pos=0.2] (E);
\filldraw[fill=white,draw=black] (C) circle (1.5pt);
\node (p) at (-0.1*\spherescale,0.4*\spherescale) {};
\node (q) at (-0.1*\spherescale,-0.7*\spherescale) {};
\node[color=\colorlinebundletop] at (0.6*\spherescale,1.05*\spherescale) {$\scriptstyle \otimes\scrO_{\kern -1pt X}\!(1)$};
\draw[\monodromystyle,,color=\colorlinebundletop,bend left,looseness=0.5] 
(p) to ([shift=(100:\spherescale)]0,0);
\draw[\monodromystyle,color=\colorlinebundletop,bend right,line cap=round,dash pattern=on 0pt off 2.5\pgflinewidth] ([shift=(100:\spherescale)]0,0) to ([shift=(80:\spherescale)]0,0);
\draw[\monodromystyle,color=\colorlinebundletop,bend right,looseness=0.5,decoration={
    markings,
    mark=at position 0.5 with {\arrow{>}}},postaction=decorate] 
([shift=(75:\spherescale)]0,0) to (p);
 
 \draw[\monodromystyle,color=\colorflop,bend left,looseness=1.0,decoration={
    markings,
    mark=at position 0.8 with {\arrow{<}}},postaction=decorate] 
(p) to node[right,pos=0.8] {$\scriptstyle \flop$} (q);
 \draw[\monodromystyle,color=\colorflop,bend left,looseness=1.0,decoration={
    markings,
    mark=at position 0.7 with {\arrow{<}}},postaction=decorate] 
(q) to node[left,pos=0.7] {$\scriptstyle \flop$} (p);

\node[color=\colorlinebundlebottom] at (0.5*\spherescale,-1.1*\spherescale) {$\scriptstyle \otimes\scrO_{\kern -1pt X^+}\!(1)$};
 \draw[\monodromystyle,color=\colorlinebundlebottom, looseness=1.2] 
(q)
\segment{-60}{[shift=(-80:\spherescale)]0,0}{-60};
\draw[\monodromystyle,color=\colorlinebundlebottom,bend left,line cap=round,dash pattern=on 0pt off 2.5\pgflinewidth] ([shift=(-100:\spherescale)]0,0) to ([shift=(-80:\spherescale)]0,0);
 \draw[\monodromystyle,color=\colorlinebundlebottom,bend left,looseness=1.3,decoration={
    markings,
    mark=at position 0.5 with {\arrow[rotate=15]{>}}},postaction=decorate] 
([shift=(-100:\spherescale)]0,0) to (q);
\filldraw[color=white] (p) circle (3pt);
\filldraw (p) circle (1pt);
\filldraw[color=white] (q) circle (3pt);
\filldraw (q) circle (1pt);
\draw[thick] ([shift=(-84:\spherescale)]0,0) arc (-84:84:\spherescale)  
[bend left] to (96:\spherescale)
arc (96:264:\spherescale)
[bend left] to cycle;
\end{tikzpicture}
}

\newcolumntype{C}[1]{>{\centering\let\newline\\\arraybackslash\hspace{0pt}}m{#1}}

\begin{document}
\title{\textsc{Stringy K\"ahler moduli, mutation and monodromy}}
\author{Will Donovan}
\address{Will Donovan: Yau Mathematical Sciences Center, Tsinghua University, Haidian District, Beijing 100084, China; Beijing Institute of Mathematical Sciences and Applications, Yanqi Lake, Huairou District, Beijing 101408, China.
}
\email{donovan@mail.tsinghua.edu.cn}
\author{Michael Wemyss}
\address{Michael Wemyss: The Mathematics and Statistics Building, University of Glasgow, University Place, Glasgow, G12 8QQ, UK.}
\email{michael.wemyss@glasgow.ac.uk}

\begin{abstract}
This paper gives the first description of derived monodromy on the stringy K\"ahler moduli space (SKMS) for a general irreducible flopping curve~$\Curve$ in a $3$-fold~$X$ with mild singularities.  We do this by constructing two new infinite helices: the first consists of sheaves supported on~$\Curve$, and the second comprises vector bundles in a tubular neighbourhood.  We prove that these helices determine the simples and projectives in iterated tilts of the category of perverse sheaves, and  that all objects in the first helix induce a twist autoequivalence for $X$.  We show that these new derived symmetries, along with established ones, induce the full monodromy on the SKMS.

The helices have many further applications.  We (1) prove representability of noncommutative deformations of the sheaves $\scrO_{\Curve},\hdots,\scrO_{\ell\Curve}$ associated to a length~$\ell$ flopping curve, via tilting \opt{ams}{theory}\opt{ip}{\mbox{theory}}, (2) control the representing objects, characterise when they are not commutative, and their central quotients, and (3)\opt{ams}{ }\opt{ip}{~}give new and sharp theoretical lower bounds on Gopakumar--Vafa invariants for a curve of length~$\ell$. When~$X$ is smooth and resolves an affine base, we furthermore (4) prove that the second helix classifies all tilting reflexive sheaves on $X$, and thus that (5)\opt{ams}{ }\opt{ip}{~}all noncommutative crepant resolutions arise from tilting bundles on\opt{ams}{ }\opt{ip}{~}$X$.


\end{abstract}
\opt{ams}{\subjclass[2010]{Primary 14F08; Secondary 14D15, 14E30, 14J33, 16S38, 18G80}}
\opt{ip}{\subjclass{Primary 14F08; Secondary 14D15, 14E30, 14J33, 16S38, 18G80}}
\thanks{\opt{ams}{The first author }\opt{ip}{W.D.\ }was supported by Yau MSC, Tsinghua University, Yanqi Lake BIMSA, the Thousand Talents Plan, and JSPS KAKENHI Grant Number~JP16K17561\opt{ams}{. The second author }\opt{ip}{, and M.W.\ }was supported by EPSRC grants~EP/R009325/1 and EP/R034826/1.}
\maketitle
\parindent 20pt
\parskip 0pt

\thispagestyle{empty}

\section{Introduction}

Describing the simples and projectives in categories of perverse sheaves, and their tilts, is a fundamental problem.  In this paper, in the setting of \opt{ip}{\mbox{$3$-fold}}\opt{ams}{$3$-fold} flopping contractions,  we describe both, for iterated tilts of zero perverse sheaves $\Per$ at simple objects.

Our main breakthrough is the construction of two new invariants of the flopping curve in the form of infinite helices of sheaves $\{ \scrS_i\}$ and\opt{ams}{ }\opt{ip}{~}$\{ \scrV_i\}$, which we show control the simples and projectives in iterated tilts of\opt{ams}{ }\opt{ip}{~}$\Per$, respectively.  We believe that the new helices are intrinsic and of wider importance: we show that they uncover new unexpected phenomena in the autoequivalence groups of $3$-folds, where each $ \scrS_i$ gives rise to a certain generalised spherical (or Dehn) twist, and we show that the helices have further applications to deformation theory, noncommutative resolutions, and Gopakumar--Vafa invariants.  

\subsection{SKMS and Monodromy}\label{Section 1.1} The stringy K\"ahler moduli space (\opt{ip}{or }SKMS) associated to a variety $X$ is central to the study of mirror symmetry: it is isomorphic  to the complex structure moduli of the mirror manifold, and furthermore its fundamental group is conjectured to recover the derived symmetries of $X$. However, even in the crucial setting of Calabi--Yau $3$-folds, this conjecture has only been verified rigorously in certain very restricted classes of examples. 

In the simplest case, namely the Atiyah flop between resolutions of $uv=x^2+y^2$, Aspinwall \cite{Asp} explains how  the  associated derived symmetries may be recovered from an SKMS given by the sphere minus three points, as illustrated in Figure~\ref{fig.length 1}. The variety~$X$ and its flop~$X^+$ correspond to `large radius limits' near two of the removed points, shown at the poles of the sphere. Derived equivalences correspond to homotopy classes of paths between the two basepoints; monodromy around the central hole corresponds to a spherical, or Dehn, twist.

Toda generalised this result to the flop between resolutions of $uv=x^2+y^{2n}$, known as flops of length one~\cite{TodaRes,TodaFib}.  In this case the SKMS, denoted $\cM_{\scrS\scrK}$, can be considered as a certain factor of a Bridgeland stability manifold associated to $X$.  Toda proves that in this class of examples, $\cM_{\scrS\scrK}$ is still a sphere minus three points, and furthermore  its fundamental group acts on the derived category of $X$ via compositions of the functors shown in Figure~\ref{fig.length 1}, where $\flop$ is the flop functor.

\begin{figure}[h]
\begin{center}
\sphLengthOne{0}
\end{center}
\caption{Monodromy on $\cM_{\scrS\scrK}$ for length one flops.} 
\label{fig.length 1}
\end{figure}

Irreducible flops are fundamental building blocks of Calabi--Yau geometry, and are the most elementary of higher-dimensional birational surgeries.  In this paper we give a geometric description of monodromy on $\cM_{\scrS\scrK}$ when $X\to\Spec R$ is an arbitrary  irreducible $3$-fold flop, of any type, and of any length, where $X$ can even have mild singularities.  In this setting, the defining equations of $R$ are not even precisely known.

To do this, we generalise Figure~\ref{fig.length 1}, revealing a surprising and beautiful structure.  

\subsection{The Simples Helix}\label{section new helices}
Much of the richness of a general flopping contraction arises because the exceptional locus need not be reduced.  For an irreducible 3-fold flopping contraction $X\to X_{\con}$ where $X$ has only Gorenstein terminal singularities, write $\Curve$ for the exceptional locus with reduced scheme structure. Suppose that the contraction has length\opt{ams}{ }\opt{ip}{~}$\ell$, as recalled in \S\ref{section thickenings}, which is necessarily a number between one and six.  Then there exists a sequence of successive thickenings
\[ 
\scrO_{\Curve}, \scrO_{2\Curve} ,\hdots, \scrO_{\ell\Curve}.
\] 
Full definitions are given in \S\ref{section thickenings}, but for calibration, $\scrO_{\ell\Curve}$ is the structure sheaf of the scheme-theoretic fibre of the contraction. 

Our simples helix $\{ \scrS_i\}_{i\in\mathbb{Z}}$ is defined in \S\ref{section simples helix}, first by specifying a finite region of size~$N$, then by translating this region by tensoring by line bundles.  We refer to~$N$ as the {\em helix period}, and it is defined as follows:
\[
\begin{tabular}{C{0.4cm}C{0.4cm}C{0.4cm}C{0.4cm}C{0.4cm}C{0.4cm}C{0.4cm}C{0.4cm}C{0.4cm}}
\toprule
$\ell$\quad&$1$&$2$&$3$&$4$&$5$&$6$ \\
\midrule
$N$\quad&$1$ & $2$ & $4$ & $6$ & $10$ & $12$ \\
\bottomrule
\end{tabular}
\]
The period $N$ depends only on the length~$\ell$ of the curve, but in contrast to the length, $N$~does not yet have an obvious algebro-geometric interpretation.

When $\ell=1$, the period~$N=1$ and the helix is simply $\scrS_i\colonequals \scrO_{\Curve}(i-1).$ For length $\ell>1$, the period is always even, and the helix is determined by the following properties.
\begin{enumerate}
\item We put
\[
\scrS_0,\hdots,\scrS_{N/2} =
\left\{
\begin{array}{ll}
\scrO_{\Curve}(-1),\scrO_{\ell\Curve},\hdots,\scrO_{2\Curve}&\mbox{if }\ell\leq 4\\
\scrO_{\Curve}(-1),\scrO_{\ell\Curve},\hdots,\scrO_{3\Curve},\scrZ,\scrO_{2\Curve} & \mbox{if }\ell=5,6
\end{array}
\right.
\] 
where $\scrZ$ is explained below.
\item $\scrS_{-n}=\mathbb{D}\scrS_n[-1]$ for all $n$, where $\mathbb{D}$ is the dualizing functor. 
\item $\scrS_{i+N}\cong\scrS_i\otimes\scrO(1)$ for all $i$.
\end{enumerate}
It follows that the simples helix contains, albeit not in order, the sheaves $\scrO_{\Curve}, \scrO_{2\Curve} ,\hdots, \scrO_{\ell\Curve}$, the canonical bundles $\omega_{\Curve}, \omega_{2\Curve} ,\hdots, \omega_{\ell\Curve}$, and also all line bundle twists of both.  However, there is a surprise: when $\ell=5,6$  we require a further sheaf $\scrZ$ above, which we construct in~\ref{Ext23} as the unique non-split extension  $0\to\scrO_{3\Curve}\to \scrZ\to\scrO_{2\Curve}\to 0$.  

\begin{remark}\label{why simples helix}
We call $\{\scrS_i\}_{i\in\mathbb{Z}}$ a helix by analogy with the helices $\{ E_i\}_{i\in\mathbb{Z}}$ for the total space of the canonical bundle $\omega_Z$ of a del Pezzo surface~$Z$, see for instance \cite{BridgelandLocalCY,BS,Rudakov}.  Here, the relation $E_{i+n}\cong E_i\otimes\omega_Z$ is replaced by $\scrS_{i+N} = \scrS_i \otimes \scrO(1)$, and the fact that the $E_i$ induce spherical objects is replaced by the fact that whilst the $\scrS_i$ do not determine spherical objects exactly, they do after noncommutative deformation.
\end{remark}

\subsection{Monodromy on the SKMS}\label{section main results}
The simples helix gives new derived autoequivalences.  Indeed, our first result shows that all members $\scrS_i$ induce an autoequivalence in the global quasi-projective setting. This answers a question of Kawamata \cite[6.8]{Kawamata} for the sheaves $\scrO_{\Curve}, \scrO_{2\Curve} ,\hdots, \scrO_{\ell\Curve}$, but furthermore includes other sheaves like $\scrZ$.   We prove in \ref{global rep} that the noncommutative deformation functor for each $\scrS_i$ is representable, and as a consequence obtain universal sheaves $\scrE_i$.

\begin{thm}[\ref{simples global auto}]\label{intro simples global auto}
Let $Y\to Y_{\con}$ be a flopping contraction of quasi-projective \mbox{$3$-folds}, where $Y$ has at worst Gorenstein terminal singularities. For any contracted curve, consider the simples helix~$\{\scrS_i\}_{i\in\mathbb{Z}}$. Then~$\scrE_i$ is perfect on~$Y$, and there is an autoequivalence $\twistGen_{\scrS_i}$ of $\Db(\coh Y)$ which fits into a functorial triangle
\[
\RHom_Y(\scrE_{i},-)\otimes_{\End_Y(\scrE_i)}^{\bf L}\scrE_{i}\to \Id\to\twistGen_{\scrS_i}\to.
\]
\end{thm}

Motivated by mirror symmetry (\S\ref{Section 1.1}), the hard work in this paper goes into showing how the above autoequivalences knit together, and describing monodromy on the stringy K\"ahler moduli space.   In our \opt{ip}{\mbox{$3$-fold}}\opt{ams}{$3$-fold} flop setting, the SKMS $\cM_{\scrS\scrK}$ is defined to be a certain quotient of normalised Bridgeland stability conditions \cite{TodaRes, HW}.  It follows easily using the techniques in \cite{HW}, summarised in \ref{stab cor}, that if $X\to\Spec \mathfrak{R}$ is a local length~$\ell$ Gorenstein terminal flop, then $\cM_{\scrS\scrK}$ is the $2$-sphere minus $N+2$ points, where $N$ is the helix period of \S\ref{section new helices}.  Henceforth $\cM_{\scrS\scrK}$ is the $2$-sphere, as above, but now with points removed from its poles, and $N$ points removed from the equator.

Our main result is that the simples helix describes the derived monodromy on $\cM_{\scrS\scrK}$.

\begin{thm}[\ref{prop geom monod}]\label{intro geom monod} 
There is a group homomorphism
\begin{align*}
\uppi_1(\cM_{\scrS\scrK})&\pad{\to}\Auteq\Db(\coh X)\\
q_i &\pad{\mapsto}\twistGen_{\scrS_i}\\
q_-&\pad{\mapsto}-\otimes\scrO_X(-1) \\
q_+ &\pad{\mapsto}\flop^{-1}\circ\big(-\otimes\scrO_{X^+}(- 1)\big)\circ\flop
\end{align*}
 illustrated in the following diagram
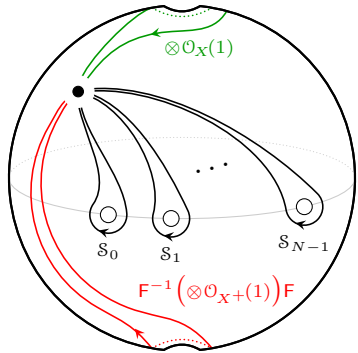
\[
\sphOnePoint{1}
\]
where $\flop$ is the flop functor, and black monodromies indicate twist functors $\twistGen_{\scrS_i}$ on~$X$, defined in a similar way as above.
\end{thm}

Given $X\to\Spec \mathfrak{R}$, we obtain a symmetric version of the above result if we systematically account for the flop $X^+\to\Spec \mathfrak{R}$.  Flopping is an involution, and this new version shows that this symmetry is reflected in the derived monodromy action. 

\begin{thm}[\ref{monodromy main}]\label{mainthm} For a $3$-fold flop $X \dashrightarrow X^+$, with length invariant $\ell$, there is an action of a fundamental groupoid of $\cM_{\scrS\scrK}$ with two basepoints on $\Db(\coh X)$ and $\Db(\coh X^+)$ given, in the cases $\ell=1$ and $\ell>1$ respectively, as follows:
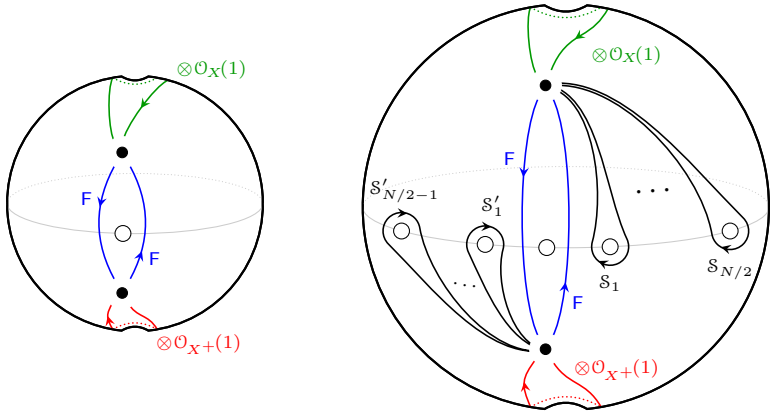
\def\colorlinebundletop{green!60!black}
\def\colorlinebundlebottom{red}
\def\colorflop{blue}
\def\colortwist{black}
\def\colorfront{black!20}
\def\colorback{black!10}
\def\monodromystyle{semithick}
\[
\def\vertspherescale{0.20}
\def\spherescale{0.85 }
\begin{array}{c}
\begin{tikzpicture}[>=stealth,scale=2]
\draw[thick] ([shift=(-84:\spherescale)]0,0) arc (-84:84:\spherescale)  
   [bend left] to (96:\spherescale)
   arc (96:264:\spherescale)
   [bend left] to cycle;
\draw[\colorfront,line cap=round,dash pattern=on 0pt off 3\pgflinewidth] (\spherescale,0) arc (0:180:\spherescale and \vertspherescale);
\draw[\colorfront] (\spherescale,0) arc (0:-180:\spherescale and \vertspherescale)
coordinate[pos=0.8] (A) coordinate[pos=0.7] (leftdots) coordinate[pos=0.63] (B) coordinate[pos=0.53] (C) coordinate[pos=0.43] (D) coordinate[pos=0.31] (rightdots) coordinate[pos=0.2] (E);
\filldraw[fill=white,draw=black] (C) circle (1.5pt);
\node (p) at (-0.1*\spherescale,0.4*\spherescale) {};
\node (q) at (-0.1*\spherescale,-0.7*\spherescale) {};
\node[color=\colorlinebundletop] at (0.6*\spherescale,1.05*\spherescale) {$\scriptstyle \otimes\scrO_{\kern -1pt X}\!(1)$};
\draw[\monodromystyle,,color=\colorlinebundletop,bend left,looseness=0.5] 
(p) to ([shift=(100:\spherescale)]0,0);
\draw[\monodromystyle,color=\colorlinebundletop,bend right,line cap=round,dash pattern=on 0pt off 2.5\pgflinewidth] ([shift=(100:\spherescale)]0,0) to ([shift=(80:\spherescale)]0,0);
\draw[\monodromystyle,color=\colorlinebundletop,bend right,looseness=0.5,decoration={
    markings,
    mark=at position 0.5 with {\arrow{>}}},postaction=decorate] 
([shift=(75:\spherescale)]0,0) to (p);
 
 \draw[\monodromystyle,color=\colorflop,bend left,looseness=1.0,decoration={
    markings,
    mark=at position 0.8 with {\arrow{<}}},postaction=decorate] 
(p) to node[right,pos=0.8] {$\scriptstyle \flop$} (q);
 \draw[\monodromystyle,color=\colorflop,bend left,looseness=1.0,decoration={
    markings,
    mark=at position 0.7 with {\arrow{<}}},postaction=decorate] 
(q) to node[left,pos=0.7] {$\scriptstyle \flop$} (p);

\node[color=\colorlinebundlebottom] at (0.5*\spherescale,-1.1*\spherescale) {$\scriptstyle \otimes\scrO_{\kern -1pt X^+}\!(1)$};
 \draw[\monodromystyle,color=\colorlinebundlebottom, looseness=1.2] 
(q)
\segment{-60}{[shift=(-80:\spherescale)]0,0}{-60};
\draw[\monodromystyle,color=\colorlinebundlebottom,bend left,line cap=round,dash pattern=on 0pt off 2.5\pgflinewidth] ([shift=(-100:\spherescale)]0,0) to ([shift=(-80:\spherescale)]0,0);
 \draw[\monodromystyle,color=\colorlinebundlebottom,bend left,looseness=1.3,decoration={
    markings,
    mark=at position 0.5 with {\arrow[rotate=15]{>}}},postaction=decorate] 
([shift=(-100:\spherescale)]0,0) to (q);
\filldraw[color=white] (p) circle (3pt);
\filldraw (p) circle (1pt);
\filldraw[color=white] (q) circle (3pt);
\filldraw (q) circle (1pt);
\draw[thick] ([shift=(-84:\spherescale)]0,0) arc (-84:84:\spherescale)  
[bend left] to (96:\spherescale)
arc (96:264:\spherescale)
[bend left] to cycle;
\end{tikzpicture}
\end{array}
\qquad
\begin{array}{c}
\sphTwoPoint{1}
\end{array}
\]
Black monodromies indicate twist functors $\twistGen_{\scrS_i}$ on~$X$ from \ref{intro simples global auto}, along with functors $\twistGen_{\scrS'_i}$ on~$X^+$ defined in the same way.
\end{thm}

The above result extends Figure~\ref{fig.length 1} to higher length flops.  It turns out that only the sheaves $\scrZ$ and $\scrO_{\ell\Curve}$ in the above theorem can be genuinely spherical (\ref{when NC needed}), and for all others we need  deformation theory to obtain the autoequivalences.

Our main new technique to construct the above actions is to geometrically describe the simples and projectives in iterated tilts of the category of perverse sheaves $\Per$.  The projectives control the noncommutative deformation theory,  and construct the twists. The actions then follow since monodromy corresponds to twisting around the simples.

\subsection{Tilts of Perverse Sheaves}\label{Ti intro section}
In the following, we construct a family of sheaves $\{ \scrV_i\}_{i\in\mathbb{Z}}$, which we then prove are vector bundles.

\begin{prop}[\ref{LemmaA}, \ref{V pos are bundles}, \ref{rank scrV eq rank V}]\label{Vi summary intro}
For all $i\in\mathbb{Z}$, $\scrV_i$ is a vector bundle on~$X$, and furthermore the following hold.
\begin{enumerate}
\item\label{Vi summary intro 2} There is a short exact sequence
\[
0\to\scrV_{i-1}\to\scrV_i^{\oplus n_i}\to\scrV_{i+1}\to 0,
\]
where  the numbers $n_i$, and the ranks of the $\scrV_i$, are given in~\ref{ns summary}.
\item\label{Vi summary intro 1} $\scrV_{i+N}\cong\scrV_i\otimes\scrO(1)$.
\end{enumerate}
\end{prop}
Motivated by \ref{Vi summary intro}\eqref{Vi summary intro 1}, we refer to $\{\scrV_i\}_{i\in\mathbb{Z}}$ as the vector bundle helix.
\begin{remark}
For length one flops, this helix is simply $\scrV_i= \scrO(i)$, and the short exact sequences above are just pullbacks to $X$ of the Euler sequence $0\to\scrO(-1)\to\scrO^{\oplus 2}\to \scrO(1)\to 0$ on~$\mathbb{P}^1$, twisted by~$\scrO(i)$.  For higher length flops, we thus view the new exact sequences in \ref{Vi summary intro}\eqref{Vi summary intro 2} as generalised Euler sequences.
\end{remark}

We construct, in \ref{def Ti} and \ref{def Ti Per}, a $\mathbb{Z}$-indexed family \opt{ams}{of hearts $\Tilt_i(\Per X)$}\opt{ip}{$\Tilt_i(\Per X)$ of hearts} obtained as successive tilts.  Our main technical result is the following, which shows that our helices $\{\scrS_i\}_{i\in\mathbb{Z}}$ and $\{\scrV_i\}_{i\in\mathbb{Z}}$ completely describe the simples and projectives in iterated tilts of perverse sheaves.

\begin{thm}\label{simples and projs}
For all $i\in\mathbb{Z}$, the abelian category $\Tilt_i(\Per X)$ has:
\begin{enumerate}
\item \textnormal{(\ref{progen main})} Progenerator $\scrP_i=\scrV_{i-1}\oplus\scrV_i$, which is a tilting bundle on $X$.
\item \textnormal{(\ref{simples main})} Simples $\scrS_{i-1}[1]$ and $\scrS_i$.
\end{enumerate}
\end{thm}

In the process of establishing the above, it turns out that there are two main actions on the family $\Tilt_i(\Per X)$, namely tensoring and duality, as follows.  We refer to the duality in \eqref{operations on Per 2} below as Perverse--Tilt duality. 
\begin{thm}\label{operations on Per}
With notation as above, for all $i,k\in\mathbb{Z}$, the following statements hold, where $\finitelength$ denotes the finite length subcategory.
\begin{enumerate}
\item\label{operations on Per 1} \textnormal{(\ref{T via tensor})}  $\Tilt_{i+kN}(\Per X)= \Tilt_i(\Per X)\otimes \scrO(k)$.
\item\label{operations on Per 2} \textnormal{(\ref{Per tilt duality})}   
$\mathbb{D}(\finitelength\Tilt_{i}(\Per X))=\finitelength\Tilt_{1-i}(\Per X)$.
\end{enumerate}
\end{thm}

Combining \ref{simples and projs} and \ref{operations on Per} we obtain a helix of abelian categories, with prescribed simples and projectives.  Figure~\ref{fig.length 5} below illustrates this in the case $\ell=5$, where $\scrA_i=\Tilt_i(\Per X)$. The inner circle illustrates the projectives, and the outer circle the simples. Here $\scrZ^\omega$ denotes the sheaf, defined in \S\ref{section simples helix}, which is dual to~$\scrZ$ in an appropriate sense.

\begin{figure}[h]
\hspace*{-3em}
\opt{ams}{$\picHelix{2}{1}$}\opt{ip}{$\picHelix{2}{1.05}$}
\caption{The simples and projectives helices for length five flops.}
\label{fig.length 5}
\end{figure}

A surprising consequence, in the local setting $X\to\Spec\mathfrak{R}$ where $X$ is smooth, is that the helix $\{\scrV_i\}_{i\in\mathbb{Z}}$ classifies \emph{all} tilting reflexive sheaves on $X$, as follows.  In particular, all noncommutative crepant resolutions of $\mathfrak{R}$ arise from tilting bundles on $X$.

\begin{cor}[\ref{all tilting bundles}]
Suppose that $\scrP$ is a basic reflexive tilting sheaf on $X$, and that $X$ is smooth.  Then $\scrP\cong\scrP_i$ for some $i\in\mathbb{Z}$.  In particular, the set of all basic tilting bundles on $X$ equals $\{\scrP_i=\scrV_{i-1}\oplus\scrV_i\}_{i\in\mathbb{Z}}$, and all reflexive tilting sheaves on $X$ are vector bundles.
\end{cor}

\subsection{Further Applications} 
As in our previous work~\cite{DW1,DW2} taking factors of local tilting bundles turns out to control the deformation theory of associated sheaves.  With our new $\mathbb{Z}$-indexed family $\scrP_i$ of such bundles from \ref{simples and projs}, consider $\End_X(\scrP_i)$ and write $[\scrV_i]$ for the two-sided ideal of morphisms that factor through a summand of a finite sum of copies of~$\scrV_i$. We thus obtain a $\mathbb{Z}$-indexed family of algebras
\[
\Lambda_i^{\deform}=\End_X(\scrP_i)/[\scrV_i].
\]
Each is finite dimensional as a vector space, since $\mathfrak{R}$ is isolated.  It turns out that $\Lambda_i^{\deform}$ represents the functor of noncommutative deformations of\opt{ams}{ }\opt{ip}{~}$\scrS_i$.

\begin{thm}[\ref{global rep}]
In the quasi-projective setting $Y\to Y_\con$ of\opt{ams}{ }\opt{ip}{~}\ref{intro simples global auto}, consider the formal fibre $X\to\Spec \mathfrak{R}$. Then $\Lambda_i^{\deform}$ represents noncommutative deformations of $\scrS_i\in\coh Y$.
\end{thm}
In particular, the abelianisation of $\Lambda_i^{\deform}$ represents commutative deformations of $\scrS_i$.  Since $\Lambda_i^{\deform}$ is a factor of a tilting algebra, the techniques of \cite{DW1,DW3} apply, and this homological control allows us to extract very fine  information about the dimension of the deformation spaces. In \ref{quivers and dims} we determine precisely when $\Lambda_i^{\deform}$ is not commutative, and we give lower bounds on its dimension, and on the dimension of its abelianisation.  

Recall that the \emph{contraction algebra} $\CA$ was defined in \cite{DW1} via noncommutative deformations of $\scrO_{\Curve}(-1)$. In the notation here, \opt{ip}{note that }$\CA=\Lambda_0^{\deform}$.

\begin{cor}[\ref{GV bounds}]
For smooth $X$, there exist lower bounds as follows, where GV are the Gopakumar--Vafa invariants.
\[
\begin{tabular}{C{0.6cm}lc}
\toprule
$\ell$&\textnormal{GV lower bounds}&$\dim\CA$ \textnormal{lower bound}\\
\midrule
$1$ & $(1)$& $1$\\
$2$ & $(4,1)$& $8$\\
$3$ & $(5,3,1)$& $26$\\
$4$ & $(6,4,2,1)$& $56$\\
$5$ & $(7,6,4,2,1)$& $124$\\
$6$ &$(6,6,4,3,2,1)$& $200$\\
\bottomrule
\end{tabular}
\]
\end{cor}
Our last corollary, which may be of independent interest, shows that noncommutative deformations detect higher multiples of \opt{ams}{the curve}\opt{ip}{$\Curve$}.  When the representing object of noncommutative deformation theory is not commutative, we say that strictly noncommutative deformations exist.
\begin{cor}[\ref{when NC needed}, \ref{higher mult def}]
For $1\leq a\leq \ell$, higher multiples $na\Curve$ of\opt{ams}{ }\opt{ip}{~}$a\Curve$ exist \mbox{(i.e.~$2a\leq \ell$)} if and only if there exist strictly noncommutative deformations of the sheaf $\scrO_{a\Curve}$.
\end{cor}

\subsection*{Acknowledgements} 
The authors thank Osamu Iyama, Yuki Hirano, Sheldon Katz and Yukinobu Toda for helpful conversations, and Kenny Brown for conversations regarding commutativity in \ref{quivers and dims}.

\section{Thickenings, Perversity and Dynkin Combinatorics}

The setup is a flopping contraction $f\colon X\to\Spec R$ of an irreducible rational curve, where $(R,\m)$ is a complete local $\mathbb{C}$-algebra, and $X$ has Gorenstein terminal singularities.

\subsection{Thickenings and Katz sequences}\label{section thickenings}
If $g$ is a generic element of\opt{ams}{ }\opt{ip}{~}$\m$, pulling back $X\to\Spec R$ along the map $\Spec R/g\to\Spec R$ gives a morphism $S\to\Spec R/g$, say.  By Reid's general elephant theorem\opt{ams}{ }\opt{ip}{~}\cite{Pagoda}, $R/g$ is an ADE surface singularity, and $S$ is a partial crepant resolution. As such, by the McKay correspondence the exceptional curve in $S$ corresponds to some vertex $v$ in some ADE Dynkin diagram.  Labelling the nodes of the Dynkin diagram by the rank of the highest root, the number attached to the node $v$ is called the \emph{length} $\ell$ of the curve.  

\begin{notation} Let $\Curve$ be the fibre $f^{-1}(\m)$ with \emph{reduced} scheme structure.\end{notation}

As usual $\Curve \cong \mathbb{P}^1$. Let $\scrI = \scrI_{\Curve,S}$ be the ideal sheaf of $\Curve$ in $S$.  For each $1\leq a\leq \ell$, let $\scrI_{a}$ be the saturation of $\scrI\cdot\hdots\cdot \scrI$ ($a$ times), namely the smallest ideal sheaf containing $\scrI\cdot\hdots\cdot \scrI$ that defines a subscheme of $S$ of pure dimension one.  This necessarily has support~$\Curve$.  Write $a\Curve\subset X$ for the subscheme of $X$ defined by $\scrI_a$.

\begin{prop}\label{Katz prop}
When $X$ has only Gorenstein terminal singularities, then for $1\leq a\leq \ell$, the following statements hold.
\begin{enumerate}
\item $a\Curve$ is a Cohen--Macaulay (CM) scheme of dimension one.  Furthermore, $\ell\Curve$ is the scheme-theoretic fibre over $\m$.
\item $\mathrm{H}^0(\scrO_{a\Curve})=\mathbb{C}$ and $\mathrm{H}^1(\scrO_{a\Curve})=0$.
\item If $a\geq 2$, there is a non-split short exact sequence
\begin{equation}
0\to\scrO_{\Curve}(-1)\to\scrO_{a\Curve}\to\scrO_{(a-1)\Curve}\to 0.\label{Katz ses}
\end{equation}
\end{enumerate}
\end{prop}
\begin{proof}
This is essentially \cite[Lemma 3.2]{Katz}, but since Katz works under the assumption that $X$ is smooth, and relies on a case-by-case argument that uses Katz--Morrison \cite{KM} (which is false in our singular setting), we give a more general argument here.\\
(1) The first statement is word-for-word \cite[Lemma 3.2(i)]{Katz}.  For the second statement, exactly as in  \cite[Lemma 3.2(ii)]{Katz} it suffices to show that the scheme fibre has no embedded points: this is just \cite[3.4.2]{VdB1d}.\\
(2) The surjection $\scrO\twoheadrightarrow\scrO_{a\Curve}$ together with the fact that $\Rf_*\scrO=\scrO$ gives $\mathrm{H}^1(\scrO_{a\Curve})=0$.  Since $\ell\Curve$ is the scheme fibre by (1), necessarily  $\mathrm{H}^0(\scrO_{\ell\Curve})=\mathbb{C}$ by \cite[3.4.2]{VdB1d}.  It is also clear that $\mathrm{H}^0(\scrO_{\Curve})=\mathbb{C}$, since $\Curve \cong \mathbb{P}^1$.   Thus when $\ell=1,2$, there is nothing more to prove.  When $\ell=5,6$, necessarily $R/g$ is an $E_8$ surface singularity, and further $S$ is the partial resolution already considered in \cite{Katz}. Thus when $\ell=5,6$ the result holds by \cite[Lemma 3.2(iii)]{Katz}.

It thus suffices to prove that  $\mathrm{H}^0(\scrO_{2\Curve})=\mathbb{C}$ when $\ell=3$, and to prove that $\mathrm{H}^0(\scrO_{2\Curve})=\mathbb{C}=\mathrm{H}^0(\scrO_{3\Curve})$ when $\ell=4$.  Since $\ell=3$ can appear in various places in the longest root of an ADE Dynkin diagram, and likewise for $\ell=4$, there are various cases, and each needs to be independently verified.  We illustrate the technique in the hardest case below (namely when $S$ has an $E_6$ singularity), with the technique being general, and all other cases being similar.

Consider the minimal resolution $Y\to\Spec R/g$ where $R/g$ is an $E_8$ singularity, and consider the length three curve obtained by contracting all curves in $Y$ except for the shaded vertex indicated below (which has value $3$ in the longest root), as follows.
\[
\begin{tikzpicture}
\node (A) at (-0.6,0) {$Y$};
\node (B) at (1,0) {$S$};
\node (C) at (3,0) {$\Spec R/g$};
\node (d) at (1.2,-0.5) {$
\begin{tikzpicture}[scale=0.21]
\node at (0,0) [DB] {};
\node at (1,0) [DB] {};
\node at (2,0) [DB] {};
\node at (2,1) [DB] {};
\node at (3,0) [DB] {};
\node at (4,0) [DB] {};
\node at (5,0) [DW] {};
\node at (6,0) [DB] {};
\end{tikzpicture}$};
\draw[->] (A)--node[above]{$\scriptstyle h$}(B);
\draw[->] (B)--(C);
\end{tikzpicture}
\]
The singular surface $S$ contains precisely one length three curve, and on this curve is one~$E_6$ singularity and one~$A_1$ singularity.   We next construct a sheaf on $Y$ which pushes down to give $\scrO_{2\Curve}$ on $S$.  Let $\Delta$ denote the set of vertices of the $E_8$ Dynkin diagram. Then we consider the root 
\[
\begin{array}{c}\begin{tikzpicture}[scale=0.21]
\node at (0,0) {$\scriptstyle 0$};
\node at (1,0) {$\scriptstyle 1$};
\node at (2,0) {$\scriptstyle 2$};
\node at (2,1) {$\scriptstyle 1$};
\node at (3,0) {$\scriptstyle 2$};
\node at (4,0)  {$\scriptstyle 2$};
\node at (5,0)  {$\scriptstyle 2$};
\node at (6,0) {$\scriptstyle 1$};
\end{tikzpicture}\end{array}
=(a_i)_{i\in\Delta}
\]
of  $E_8$, chosen since it has value $2$ on the shaded vertex.  These numbers $(a_i)_{i\in\Delta}$ determine a divisor $D=\sum_{i\in\Delta}a_iE_i$, where $E_i$ are the exceptional curves in~$Y$, and we consider the exact sequence
\[
0\to\scrO(-D)\to \scrO_Y\to\scrO_D\to 0.
\]
Since  $(a_i)_{i\in\Delta}$ is a root, it is well known (and proved by induction) that $\mathrm{H}^0(Y,\scrO_{D})=\mathbb{C}$.

The line bundle $\scrO(-D)$ can easily be computed using intersection theory: for each curve $E_i$, the following diagram gives the intersection product $-D\cdot E_i$.
\begin{equation}
\begin{array}{c}\begin{tikzpicture}[scale=0.21]
\node at (-0.5,-0.06) {$\scriptstyle -1$};
\node at (1,0) {$\scriptstyle 0$};
\node at (2,0) {$\scriptstyle 0$};
\node at (2,1) {$\scriptstyle 0$};
\node at (3,0) {$\scriptstyle 0$};
\node at (4,0)  {$\scriptstyle 0$};
\node at (5,0)  {$\scriptstyle 1$};
\node at (6,0) {$\scriptstyle 0$};
\end{tikzpicture}\end{array}\label{eqn:intproduct}
\end{equation}
We next claim that $\mathbf{R}^1h_*\scrO(-D)=0$.  By Grothendieck's Theorem on Formal Functions, this can be computed locally on $S$.  A neighbourhood of the $E_6$ singular point gives rise to a neighbourhood of $Y$ containing only the six curves of $E_6$ illustrated below.
\[
\begin{tikzpicture}[scale=0.21]
\node at (0,0) [DB] {};
\node at (1,0) [DB] {};
\node at (2,0) [DB] {};
\node at (2,1) [DB] {};
\node at (3,0) [DB] {};
\node at (4,0) [DB] {};
\node at (5,0) [DW] {};
\node at (6,0) [DB] {};
\draw (-0.5,1)--(1.25,1)--(1.25,2)--(2.75,2)--(2.75,1)--(4.5,1)--(4.5,-1)--(-0.5,-1)--cycle;
\end{tikzpicture}
\]
Against these six curves, \eqref{eqn:intproduct} verifies that $\scrO(-D)$ restricts to a line bundle with intersection products as follows.
\[
\begin{array}{c}\begin{tikzpicture}[scale=0.21]
\node at (-0.5,-0.06) {$\scriptstyle -1$};
\node at (1,0) {$\scriptstyle 0$};
\node at (2,0) {$\scriptstyle 0$};
\node at (2,1) {$\scriptstyle 0$};
\node at (3,0) {$\scriptstyle 0$};
\node at (4,0)  {$\scriptstyle 0$};
\end{tikzpicture}\end{array}
\]
Since this is a summand of the Artin--Verdier tilting bundle on the minimal resolution of~$E_6$, it follows that $\mathbf{R}^1h_*\scrO(-D)$ is zero near the $E_6$ singular point. (The construction of the summands of this tilting bundle is recalled in Section~\ref{section per sheaf}. Take $f$ there to be the minimal resolution, and $\scrL$ to instead be the line bundle which has degree one against a specific curve, and degree zero against the others. Then the corresponding summand of the Artin--Verdier bundle is the constructed $\scrN$.) Near the $A_1$ singular point,  \eqref{eqn:intproduct}  verifies that restricting $\scrO(-D)$ to the single curve
\[
\begin{tikzpicture}[scale=0.21]
\node at (0,0) [DB] {};
\node at (1,0) [DB] {};
\node at (2,0) [DB] {};
\node at (2,1) [DB] {};
\node at (3,0) [DB] {};
\node at (4,0) [DB] {};
\node at (5,0) [DW] {};
\node at (6,0) [DB] {};
\draw (5.5,1)--(6.5,1)--(6.5,-1)--(5.5,-1)--cycle;
\end{tikzpicture}
\]
gives the trivial bundle. Thus $\mathbf{R}^1h_*\scrO(-D)$ is also zero near the $A_1$ singular point.

The claim is verified, thus $\mathbf{R}^1h_*\scrO(-D)=0$ and so 
\[
0\to h_*\scrO(-D)\to\scrO_S\to h_*\scrO_D\to 0
\]
is exact.  This defines a scheme structure on the curve in $S$, which has multiplicity two at the smooth points, since this is true in $Y$, by construction of~$D$.  Since $\mathrm{H}^0(h_*\scrO_D)=\mathrm{H}^0(Y,\scrO_D)=\mathbb{C}$, it follows that this scheme structure gives a CM curve, since embedded points would give rise to additional sections.  Thus we have produced a CM curve of multiplicity two in $S$, and so as observed in \cite[beginning of proof of Lemma 3.2]{Katz} it follows from uniqueness that $h_*\scrO_D\cong\scrO_{2\Curve}$.  In particular $\mathrm{H}^0(\scrO_{2\Curve})=\mathbb{C}$, as required.

All the other cases involving partial resolutions of Kleinian singularities with length three curves, respectively length four curves, follow using the same technique.  In each case, we can lift $a\Curve$ (with $a=2$, respectively $a=2,3$) to a root, then construct a sheaf on $Y$ using this root. This sheaf is known to have only constant sections: intersection theory then allows us to push it down to obtain $\scrO_{a\Curve}$, and thus $\scrO_{a\Curve}$ has only constant sections.

\noindent
(3) There is an exact sequence
\[
0\to \scrI_{a-1}/\scrI_{a}\to \scrO_{a\Curve}\to \scrO_{(a-1)\Curve}\to 0.
\]
Since $ \scrI_{a-1}/\scrI_{a}$ is a torsion-free sheaf of rank 1 on $\Curve\cong\mathbb{P}^1$, it is locally free. By (2) we know that $\mathrm{H}^0(\scrI_{a-1}/\scrI_{a})=0=\mathrm{H}^1(\scrI_{a-1}/\scrI_{a})$,  hence it follows that $\scrI_{a-1}/\scrI_{a}\cong\scrO_{\Curve}(-1)$, as required.  The sequence does not split since applying $\Hom_X(-,\scrO_{\Curve}(-1))$ to the surjection $\scrO_X\twoheadrightarrow\scrO_{a\Curve}$ shows that \opt{ams}{$\Hom_X(\scrO_{a\Curve},\scrO_{\Curve}(-1))=0$,}\opt{ip}{\[\Hom_X(\scrO_{a\Curve},\scrO_{\Curve}(-1))=0,\]} and so there can be no splitting morphism.
\end{proof}

The dualizing sheaves on the thickenings $a\Curve$ turn out to be important. As notation, consider the commutative diagram
\[
\begin{tikzpicture}
\node (A) at (0,0) {$a\Curve$};
\node (B) at (1.5,0) {$X$};
\node (C) at (1.5,-1.5) {$\Spec\mathbb{C}$};
\draw[right hook->] (A) to node[above] {$\scriptstyle \upiota$}(B);
\draw[->] (A) to node[below] {$\scriptstyle q$}(C);
\draw[->] (B) to node[right] {$\scriptstyle p$}(C);
\end{tikzpicture}
\]
and recall that $\mathbb{D}_X=\RsHom_X(-,p^!\mathbb{C})$.  The following is very well known.
\def\dualThickn{\mathbb{D}_X(\scrO_{a\Curve})\cong\omega_{a\Curve}[1].}
\begin{lemma}\label{D to kC}
For $1\leq a\leq \ell$, there is an isomorphism \opt{ams}{$\dualThickn$}\opt{ip}{\[\dualThickn\]}
\end{lemma}
\begin{proof}
By definition $\mathbb{D}_X(\scrO_{a\Curve})$ is $\mathbb{D}_X(\upiota_*\scrO_{a\Curve})$.  Since $\upiota$ is a closed embedding, $\RDerived \upiota_*=\upiota_*$. Furthermore, closed immersions are proper, thus sheafified Grothendieck duality gives
\begin{align*}
\mathbb{D}_X(\scrO_{a\Curve})
=\RsHom_X(\upiota_*\scrO_{a\Curve},p^!\mathbb{C})=\upiota_*\RsHom_{a\Curve}(\scrO_{a\Curve},\upiota^!p^!\mathbb{C})=\upiota_*(\upiota^!p^!\mathbb{C}).
\end{align*}
But $\upiota^!p^!\mathbb{C}=q^!\mathbb{C}=\omega_{a\Curve}[1]$ since $a\Curve$ is CM by \ref{Katz prop}, hence this is $\upiota_*\omega_{a\Curve}[1],$ as required.
\end{proof}

\subsection{Perverse Sheaves}\label{section per sheaf}
Recall \cite{Bridgeland} that \emph{zero perverse sheaves} are defined
\[
\Per (X,R)=\left\{ \scrF\in\Db(\coh X)\left| \begin{array}{l}\mathrm{H}^i(\scrF)=0\mbox{ if }i\neq 0,-1\\
f_*\mathrm{H}^{-1}(\scrF)=0\mbox{, }\Rfi{1}_* \mathrm{H}^0(\scrF)=0\\ \Hom(c,\mathrm{H}^{-1}(\scrF))=0\mbox{ for all }c\in\scrC_0 \end{array}\right. \right\},
\]
where $\scrC:=\{ \scrG\in\Db(\coh X)\mid \Rf_*\scrG=0\}$ and $\scrC_0$ denotes the full subcategory of $\scrC$ whose objects have cohomology only in degree $0$.

There is a bundle $\scrO_X(1)$ on $X$ with degree $1$ on~$\Curve$. Writing $\scrL\colonequals\scrO_X(1)$, there is an exact sequence
\[
0\to\scrO_X^{\oplus(\ell-1)}\to\scrM\to\scrL\to 0
\]
associated to a minimal set of $\ell-1$ generators of $\mathrm{H}^1(X,\scrL^{*})$ \cite[3.5.4]{VdB1d}.  Write $\scrN\colonequals \scrM^*$, and $\Lambda\colonequals \End_X(\scrO\oplus\scrN)$. Then $\scrO\oplus\scrN$ gives a derived equivalence $\Uppsi$, which restricts to an equivalence on bounded hearts as illustrated in the following commutative diagram.
\begin{equation}
\begin{array}{c}
\begin{tikzpicture}
\node (A1) at (0,0) {$\Db(\coh X)$};
\node (A2) at (5,0) {$\Db(\mod\Lambda)$};
\node (B1) at (0,-1.5) {$\Per (X,R)$};
\node (B2) at (5,-1.5) {$\mod\Lambda$};
\draw[->] (A1) -- node[above] {$\scriptstyle\Uppsi\colonequals \RHom_X(\scrO\oplus\scrN,-)$} node[below]{$\scriptstyle\sim$}(A2);
\draw[->] (B1) -- node[above] {$\scriptstyle\Uppsi$} node[below]{$\scriptstyle\sim$}(B2);
\draw[right hook->] (B1) -- (A1);
\draw[right hook->] (B2) -- (A2);
\end{tikzpicture}
\end{array}\label{Psi per 0}
\end{equation}

\subsection{Summary of Dynkin Combinatorics}\label{IW9 summary section}

This subsection summarises some results from \cite[\S9, \S10]{IW9} and \cite{HW}, mainly to set notation.  Under our  setup $f\colon X\to\Spec R$, consider $R\oplus f_*\scrM$.  This $R$-module is reflexive, and rigid in the sense that its self-extension group is zero.
\begin{thm}\label{IW9 summary}\cite[10.7]{IW9}
The mutation class containing $R\oplus f_*\scrM$ is in bijection with the chambers of an infinite hyperplane arrangement in $\mathbb{R}^1$. 
\end{thm}

To set notation, we label the walls by rigid reflexive modules $V_i$ with $i\in\mathbb{Z}$, and chambers are then labelled by their direct sum, as follows.
\begin{equation}
\begin{tikzpicture}
\draw[densely dotted,->] (\opt{ams}{-2.5}\opt{ip}{0.4},0)--(9,0);
\opt{ams}{\node (A) at (-2.2,0) [cvertex] {};}
\node (B) at (0.8,0) [cvertex] {};
\node (C) at (3.5,0) [cvertex] {};
\node (D) at (6,0) [cvertex] {};
\node (E) at (8.7,0) [cvertex] {};
\opt{ams}{\node at (-2.2,0.3) {$\scriptstyle V_{-2}$};}
\node at (0.8,0.3) {$\scriptstyle V_{-1}$};
\node at (3.5,0.3) {$\scriptstyle V_{0}$};
\node at (6,0.3) {$\scriptstyle V_{1}$};
\node at (8.7,0.3) {$\scriptstyle V_{2}$};
\opt{ams}{\draw[draw=none] (A) -- node[gap] {$V_{-2}\oplus V_{-1}$}(B);}
\draw[draw=none] (B) -- node[gap] {$V_{-1}\oplus V_0$}(C);
\draw[draw=none] (C) -- node[gap] {$V_{0}\oplus V_1$}(D);
\draw[draw=none] (D) -- node[gap] {$V_{1}\oplus V_2$}(E);
\end{tikzpicture}
\quad
\label{arrangement labelled}
\end{equation}
Crossing the wall labelled $V_i$ from left to right means that we consider $V_{i-1}\oplus V_i$, keep $V_i$, and replace $V_{i-1}$ by $V_{i+1}$.  This mutation process is governed homologically by \emph{exchange sequences}: in this context, for all $i\in\mathbb{Z}$ there is an exact sequence
\begin{equation}
0\to V_{i-1}\to V_i^{\oplus n_i}\to V_{i+1}\label{exchange 1}
\end{equation}
such that applying $\Hom_R(V_{i-1}\oplus V_i,-)$ gives a short exact sequence \cite[(6.Q)]{IW4}.  Similarly, for wall crossing from right to left, there exists an exchange sequence  
\begin{equation}
0\to V_{i+1}\to V_i^{\oplus n_i'}\to V_{i-1}.\label{exchange 2}
\end{equation}
For isolated cDV singularities $n_i=n_i'$ \cite[10.4]{IW9}, and furthermore 
\begin{equation}
\rank_R V_{i+1}+\rank_R V_{i-1}=n_i\cdot\rank_R V_i.\label{count ranks}
\end{equation}
We will use both these facts implicitly throughout.

By convention, since $R\oplus f_*\scrM$ must appear in its mutation class, we set $V_0=R$ and $V_{1}=f_*\scrM$.  Under this convention,  $L=f_*\scrO(1)$ generates a subgroup of the class group $\Cl(R)$ which acts on the above hyperplane arrangement, by translating to the right.
\begin{prop}\label{class action comb}\cite[9.10, 10.7]{IW9}
Set $L\colonequals f_*\scrO(1)$ and consider $\mathbb{Z}\cong\langle L\rangle\leq\Cl(R)$.  Then $L$ acts on \eqref{arrangement labelled} by translation, taking the wall labelled $V_0=R$ to the next wall to the right for which its label $V_i$ has rank one.
\end{prop}

Thus there always exists a number $N$, which below turns out to depend only on $\ell$, such that $V_{i+N}\cong V_i\cdot L\colonequals (V_{i}\otimes L)^{**}$ for all $i\in\mathbb{Z}$.  This number, together with the ranks of the $V_i$, and the $n_i$, can be calculated combinatorially. When $X$ is smooth, this was achieved in \cite[7.11]{HW} using a result of Katz--Morrison \cite{KM} which fails in the more general singular setting here.  The following completes the calculation in all cases.

\begin{prop}\label{ns summary}
Suppose that $X\to\Spec R$ is a length $\ell$ flop, where $X$ has only Gorenstein terminal singularities.  Then in the hyperplane arrangement the walls are numbered by the ranks of \opt{ams}{the }$V_0,\hdots,V_{N-1}$ in the table below.
\[
\begin{tabular}{C{0.6cm}C{0.7cm}ll}
\toprule
$\ell$&$N$&\textnormal{Ranks of $V_0,\hdots,V_{N-1}$}&$n_0,\hdots,n_{N-1}$\\
\midrule
$1$ & $1$ & $1$ & $2$\\
$2$ & $2$ & $1$,$2$ & $4$,$1$\\
$3$ & $4$ & $1$,$3$,$2$,$3$ & $6$,$1$,$3$,$1$\\
$4$ & $6$ & $1$,$4$,$3$,$2$,$3$,$4$ & $8$,$1$,$2$,$3$,$2$,$1$\\
$5$ & $10$ & $1$,$5$,$4$,$3$,$5$,$2$,$5$,$3$,$4$,$5$  & $10$,$1$,$2$,$3$,$1$,$5$,$1$,$3$,$2$,$1$\\
$6$ & $12$ & $1$,$6$,$5$,$4$,$3$,$5$,$2$,$5$,$3$,$4$,$5$,$6$  & $12$,$1$,$2$,$2$,$3$,$1$,$5$,$1$,$3$,$2$,$2$,$1$\\
\bottomrule
\end{tabular}
\]
\end{prop}
\begin{proof}
We do this for $\ell=3$, with all other cases being similar.   There are precisely five places where a vertex is labelled $3$ in an ADE Dynkin diagram:
\[
\begin{array}{c}
\begin{tikzpicture}[scale=0.21]
\node at (0,0) [DB] {};
\node at (1,0) [DB] {};
\node at (2,0) [DW] {};
\node at (2,1) [DB] {};
\node at (3,0) [DB] {};
\node at (4,0) [DB] {};
\end{tikzpicture}
\end{array}
\quad
\begin{array}{c}
\begin{tikzpicture}[scale=0.21]
\node at (5,0) [DB] {};
\node at (0,0) [DB] {};
\node at (1,0) [DW] {};
\node at (2,0) [DB] {};
\node at (2,1) [DB] {};
\node at (3,0) [DB] {};
\node at (4,0) [DB] {};
\end{tikzpicture}
\end{array}
\quad
\begin{array}{c}
\begin{tikzpicture}[scale=0.21]
\node at (5,0) [DB] {};
\node at (0,0) [DB] {};
\node at (1,0) [DB] {};
\node at (2,0) [DB] {};
\node at (2,1) [DB] {};
\node at (3,0) [DW] {};
\node at (4,0) [DB] {};
\end{tikzpicture}
\end{array}
\quad
\begin{array}{c}
\begin{tikzpicture}[scale=0.21]
\node at (0,0) [DB] {};
\node at (1,0) [DB] {};
\node at (2,0) [DB] {};
\node at (2,1) [DB] {};
\node at (3,0) [DB] {};
\node at (4,0) [DB] {};
\node at (5,0) [DW] {};
\node at (6,0) [DB] {};
\end{tikzpicture}
\end{array}
\quad
\begin{array}{c}
\begin{tikzpicture}[scale=0.21]
\node at (0,0) [DB] {};
\node at (1,0) [DB] {};
\node at (2,0) [DB] {};
\node at (2,1) [DW] {};
\node at (3,0) [DB] {};
\node at (4,0) [DB] {};
\node at (5,0) [DB] {};
\node at (6,0) [DB] {};
\end{tikzpicture}
\end{array}
\]
We analyse each individually using the combinatorics of wall crossing described in~\cite[1.16]{IW9}.  The first is calculated in~\cite[Ex.~1.1]{WOVERVIEW} (see also~\cite[7.11]{HW}), and is as \opt{ams}{\mbox{follows}}\opt{ip}{follows}. 
\def\ipDynkinScaleA{0.9}
\def\ipDynkinScaleB{0.85}
\newcommand{\DynkinStripStretch}{1.1}
\newcommand{\DynkinStripShift}{10pt}
\newcommand{\DynkinStripStretchB}{1.18}
\newcommand{\DynkinStripShiftB}{10pt}
\newcommand{\EbaseA}[1][]{%
\begin{tikzpicture}[scale=0.21]
\node at (0,0) [DB] {};
\node at (1,0) [DB] {};
\node at (2,0) [DW] {};
\node at (2,1) [DB] {};
\node at (2,2) [DW] {};
\node at (3,0) [DB] {};
\node at (4,0) [DB] {};
\end{tikzpicture}
}
\newcommand{\EbaseB}[1][]{%
\begin{tikzpicture}[scale=0.21]
\node at (0,0) [DB] {};
\node at (1,0) [DB] {};
\node at (2,0) [DW] {};
\node at (2,1) [DW] {};
\node at (2,2) [DB] {};
\node at (3,0) [DB] {};
\node at (4,0) [DB] {};
\end{tikzpicture}
}
\[
\begin{tikzpicture}[scale=\opt{ip}{\ipDynkinScaleA}\opt{ams}{1},xscale=\DynkinStripStretch]
\filldraw[gray!10!white,transform canvas={xshift=+\DynkinStripShift}] (-2.5,-0.5) -- (-2.5,1) -- (5.5,1)--(5.5,-0.5) --cycle;
\draw[densely dotted,->] (-4.5,0) -- (6.5,0);
\node at (6.75,0) {$\mathbb{R}$};
{\foreach \i in {-3,-2,-1,0,1,2}
\filldraw[fill=white,draw=black] (-2*\i,0) circle (2pt);
}
{\foreach \i in {2,1}
\node at (-2*\i+1,0.5) {\EbaseA};
}
{\foreach \i in {0,-1}
\node at (-2*\i+1,0.5) {\EbaseB};
}
\node at (5,0.5) {\EbaseA};
\node at (-4,-0.25) {$\scriptstyle 3$};
\node at (-2,-0.25) {$\scriptstyle 1$};
\node at (0,-0.25) {$\scriptstyle 3$};
\node at (2,-0.25) {$\scriptstyle 2$};
\node at (4,-0.25) {$\scriptstyle 3$};
\node at (6,-0.25) {$\scriptstyle 1$};
\end{tikzpicture}
\]
The next two cases are covered by the following calculation
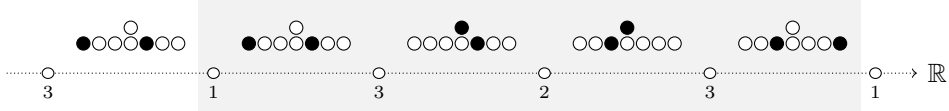
\newcommand{\EsevenbaseA}[1][]{%
\begin{tikzpicture}[scale=0.21]
\node at (-1,0) [DW] {};
\node at (0,0) [DB] {};
\node at (1,0) [DB] {};
\node at (2,0) [DB] {};
\node at (2,1) [DB] {};
\node at (3,0) [DW] {};
\node at (4,0) [DB] {};
\node at (5,0) [DB] {};
\end{tikzpicture}
}
\newcommand{\EsevenbaseB}[1][]{%
\begin{tikzpicture}[scale=0.21]
\node at (-1,0) [DB] {};
\node at (0,0) [DB] {};
\node at (1,0) [DB] {};
\node at (2,0) [DB] {};
\node at (2,1) [DW] {};
\node at (3,0) [DW] {};
\node at (4,0) [DB] {};
\node at (5,0) [DB] {};
\end{tikzpicture}
}
\newcommand{\EsevenbaseC}[1][]{%
\begin{tikzpicture}[scale=0.21]
\node at (-1,0) [DB] {};
\node at (0,0) [DB] {};
\node at (1,0) [DW] {};
\node at (2,0) [DB] {};
\node at (2,1) [DW] {};
\node at (3,0) [DB] {};
\node at (4,0) [DB] {};
\node at (5,0) [DB] {};
\end{tikzpicture}
}
\newcommand{\EsevenbaseD}[1][]{%
\begin{tikzpicture}[scale=0.21]
\node at (-1,0) [DB] {};
\node at (0,0) [DB] {};
\node at (1,0) [DW] {};
\node at (2,0) [DB] {};
\node at (2,1) [DB] {};
\node at (3,0) [DB] {};
\node at (4,0) [DB] {};
\node at (5,0) [DW] {};
\end{tikzpicture}
}
\[
\begin{tikzpicture}[scale=\opt{ip}{\ipDynkinScaleA}\opt{ams}{1},xscale=\DynkinStripStretch]
\filldraw[gray!10!white,transform canvas={xshift=+\DynkinStripShift}] (-2.5,-0.5) -- (-2.5,1) -- (5.5,1)--(5.5,-0.5) --cycle;
\draw[densely dotted,->] (-4.5,0) -- (6.5,0);
\node at (6.75,0) {$\mathbb{R}$};
{\foreach \i in {-3,-2,-1,0,1,2}
\filldraw[fill=white,draw=black] (-2*\i,0) circle (2pt);
}

\node at (-3,0.5) {\EsevenbaseA};
\node at (-1,0.5) {\EsevenbaseA};
\node at (1,0.5) {\EsevenbaseB};
\node at (3,0.5) {\EsevenbaseC};
\node at (5,0.5) {\EsevenbaseD};
\node at (-4,-0.25) {$\scriptstyle 3$};
\node at (-2,-0.25) {$\scriptstyle 1$};
\node at (0,-0.25) {$\scriptstyle 3$};
\node at (2,-0.25) {$\scriptstyle 2$};
\node at (4,-0.25) {$\scriptstyle 3$};
\node at (6,-0.25) {$\scriptstyle 1$};
\end{tikzpicture}
\]
 and the last two cases are given below.
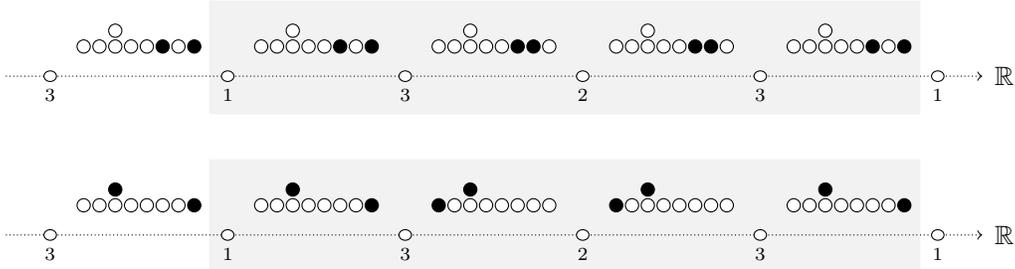
\newcommand{\EeightbaseA}[1][]{%
\begin{tikzpicture}[scale=0.21]
\node at (0,0) [DB] {};
\node at (1,0) [DB] {};
\node at (2,0) [DB] {};
\node at (2,1) [DB] {};
\node at (3,0) [DB] {};
\node at (4,0) [DB] {};
\node at (5,0) [DW] {};
\node at (6,0) [DB] {};
\node at (7,0) [DW] {};
\end{tikzpicture}
}
\newcommand{\EeightbaseB}[1][]{%
\begin{tikzpicture}[scale=0.21]
\node at (0,0) [DB] {};
\node at (1,0) [DB] {};
\node at (2,0) [DB] {};
\node at (2,1) [DB] {};
\node at (3,0) [DB] {};
\node at (4,0) [DB] {};
\node at (5,0) [DW] {};
\node at (6,0) [DW] {};
\node at (7,0) [DB] {};
\end{tikzpicture}
}
\[
\begin{tikzpicture}[scale=\opt{ip}{\ipDynkinScaleB}\opt{ams}{1},xscale=\DynkinStripStretchB]
\filldraw[gray!10!white,transform canvas={xshift=+\DynkinStripShiftB}] (-2.5,-0.5) -- (-2.5,1) -- (5.5,1)--(5.5,-0.5) --cycle;
\draw[densely dotted,->] (-4.5,0) -- (6.5,0);
\node at (6.75,0) {$\mathbb{R}$};
{\foreach \i in {-3,-2,-1,0,1,2}
\filldraw[fill=white,draw=black] (-2*\i,0) circle (2pt);
}
\node at (-3,0.5) {\EeightbaseA};
\node at (-1,0.5) {\EeightbaseA};
\node at (1,0.5) {\EeightbaseB};
\node at (3,0.5) {\EeightbaseB};
\node at (5,0.5) {\EeightbaseA};
\node at (-4,-0.25) {$\scriptstyle 3$};
\node at (-2,-0.25) {$\scriptstyle 1$};
\node at (0,-0.25) {$\scriptstyle 3$};
\node at (2,-0.25) {$\scriptstyle 2$};
\node at (4,-0.25) {$\scriptstyle 3$};
\node at (6,-0.25) {$\scriptstyle 1$};
\end{tikzpicture}
\]
\newcommand{\EeightbaseC}[1][]{%
\begin{tikzpicture}[scale=0.21]
\node at (0,0) [DB] {};
\node at (1,0) [DB] {};
\node at (2,0) [DB] {};
\node at (2,1) [DW] {};
\node at (3,0) [DB] {};
\node at (4,0) [DB] {};
\node at (5,0) [DB] {};
\node at (6,0) [DB] {};
\node at (7,0) [DW] {};
\end{tikzpicture}
}
\newcommand{\EeightbaseD}[1][]{%
\begin{tikzpicture}[scale=0.21]
\node at (0,0) [DW] {};
\node at (1,0) [DB] {};
\node at (2,0) [DB] {};
\node at (2,1) [DW] {};
\node at (3,0) [DB] {};
\node at (4,0) [DB] {};
\node at (5,0) [DB] {};
\node at (6,0) [DB] {};
\node at (7,0) [DB] {};
\end{tikzpicture}
}
\[
\begin{tikzpicture}[scale=\opt{ip}{\ipDynkinScaleB}\opt{ams}{1},xscale=\DynkinStripStretchB]
\filldraw[gray!10!white,transform canvas={xshift=+\DynkinStripShiftB}] (-2.5,-0.5) -- (-2.5,1) -- (5.5,1)--(5.5,-0.5) --cycle;
\draw[densely dotted,->] (-4.5,0) -- (6.5,0);
\node at (6.75,0) {$\mathbb{R}$};
{\foreach \i in {-3,-2,-1,0,1,2}
\filldraw[fill=white,draw=black] (-2*\i,0) circle (2pt);
}
\node at (-3,0.5) {\EeightbaseC};
\node at (-1,0.5) {\EeightbaseC};
\node at (1,0.5) {\EeightbaseD};
\node at (3,0.5) {\EeightbaseD};
\node at (5,0.5) {\EeightbaseC};
\node at (-4,-0.25) {$\scriptstyle 3$};
\node at (-2,-0.25) {$\scriptstyle 1$};
\node at (0,-0.25) {$\scriptstyle 3$};
\node at (2,-0.25) {$\scriptstyle 2$};
\node at (4,-0.25) {$\scriptstyle 3$};
\node at (6,-0.25) {$\scriptstyle 1$};
\end{tikzpicture}
\]
In all cases we see that $N=4$, and the walls are numbered by $1,3,2,3$, then repeat.  That these are the ranks of the $V_i$ is \cite[9.10(2), 10.4]{IW9}.  The numbers $n_i$ follow from these, using \eqref{count ranks}.
\end{proof}

Consider the full subcategory $\scrD\colonequals \{ \scrF\in\Db(\coh X)\mid \Supp\scrF\subseteq \Curve\}$,
and  the manifold $\Stab{n}{\scrD}$ of locally finite  Bridgeland stability conditions on $\scrD$ satisfying the normalisation~$Z([\scrO_x])=\mathrm{i}$.  There is a connected component $\cStab{n}\scrD$ containing those normalised stability conditions with heart $\Per\cap \scrD$.  Fourier--Mukai autoequivalences of $\Db(\coh X)$ restrict to $\scrD$, and we consider those that furthermore preserve $\cStab{n}\scrD$ and are $R$-linear. This group is denoted $\cAut\scrD$.  
\begin{cor}\label{stab cor}
Suppose that $X\to\Spec R$ is a length $\ell$ flop, where $X$ has only Gorenstein terminal singularities.  Then
\def\MSK{\cM_{\scrS\scrK}\colonequals \cStab{n}{\scrD}/\cAut\scrD}
\opt{ams}{$\MSK$}\opt{ip}{\[\MSK\]}
is homeomorphic to the $2$-sphere with $N+2$ points removed, where $N$ is given in \ref{ns summary}.
\end{cor}
\begin{proof}
In the smooth case, this is \cite[7.12]{HW}.  Exactly the same argument works here, using that $N$ is calculated in \ref{ns summary}, and the dependence of $N$ on $\ell$ is the same for the singular setting and for the smooth case. 
\end{proof}

\section{Twisting and Mutating on the Strip}

With notation as in \eqref{arrangement labelled}, set $\Lambda_i\colonequals \End_R(M_i)$, where $M_i\colonequals V_{i-1}\oplus V_i$.  The ring $\Lambda_i$ always has two projective modules, namely $\Hom_R(M_{i},V_{i-1})$ and $\Hom_R(M_i,V_i)$. 

\subsection{Global Ordering on Projectives and Simples}\label{ordering section}
Both for iteration purposes later, and to make theorems easier to state, it is convenient to now fix an ordering on the projectives.  This ordering is induced by mutation.  Concretely,
\def\eqnPA{P_0\opt{ip}{&}=
\left\{ \begin{array}{ll}
\Hom_R(M_i,V_{i})&\mbox{if $i$ is even}\\
\Hom_R(M_i,V_{i-1})&\mbox{if $i$ is odd,}
\end{array}
\right.}
\def\eqnPB{P_1\opt{ip}{&}=
\left\{ \begin{array}{ll}
\Hom_R(M_i,V_{i-1})&\mbox{if $i$ is even}\\
\Hom_R(M_i,V_{i})&\mbox{if $i$ is odd.}
\end{array}
\right.}
\opt{ams}{\[
\eqnPA
\quad
\eqnPB
\]}
\opt{ip}{\begin{align*}
\eqnPA \\
\eqnPB
\end{align*}}
This is a slight abuse of notation, since $P_0$ and $P_1$ depend on which $\Lambda_i$ is being considered.  As a similar abuse, we will always write $S_0$ for the simple corresponding to $P_0$, and $S_1$ for the simple corresponding to $P_1$, regardless of $\Lambda_i$.

Via \eqref{Psi per 0}, this then fixes an ordering on the simples in $\Per$. Indeed, for $\Lambda=\Lambda_0$ across the equivalence \eqref{Psi per 0}, 
\[
\begin{array}{rcl}
\scrO&\longleftrightarrow&P_0\\
\scrN&\longleftrightarrow&P_1
\end{array}
\quad
\mbox{and}
\quad
\begin{array}{rcl}
\omega_{\ell\Curve}[1]&\longleftrightarrow&S_0\\
\scrO_{\Curve}(-1)&\longleftrightarrow&S_1.
\end{array}
\]

\subsection{Mutation and Wall Crossing}\label{mut notation section}
Given $\End_R(A)$ and $\End_R(B)$, consider the functor 
\begin{equation}
\RHom_{\End_R(A)}(T_{AB},-)\colon \Db(\opt{ams}{\mod}\End_R(A))\to\Db(\opt{ams}{\mod}\End_R(B))\label{gen mut functor}
\end{equation}
where $T_{AB}=\Hom_R(A,B)$.  Since $R$ is isolated cDV, if $A$ and $B$ are rigid reflexive $R$-modules that are connected by a finite sequence of mutations, then the above functor is an equivalence \cite[10.1, 10.5]{IW9}.  This then gives a chain of equivalences, which by e.g.\ \cite[4.15(2)]{HomMMP} and our choice of orderings send
\begin{equation}
\begin{array}{c}
\begin{tikzpicture}[xscale=\opt{ams}{1.2}\opt{ip}{1.3}]
\node (A-1) at (-1.5,0) {$\hdots$};
\node (A0) at (0,0) {$\Db(\mod\Lambda_{-1})$};
\node (A1) at (2,0) {$\Db(\mod\Lambda)$};
\node (A2) at (4,0) {$\Db(\mod\Lambda_{1})$};
\opt{ip}{\node (A3) at (5.5,0) {$\hdots$};}
\opt{ams}{\node (A3) at (6,0) {$\Db(\mod\Lambda_{2})$};}
\opt{ams}{\node (A4) at (7.5,0) {$\hdots$};}
\draw[->] (A-1) -- (A0);
\draw[->] (A0) -- (A1);
\draw[->] (A1) -- (A2);
\draw[->] (A2) -- (A3);
\opt{ams}{\draw[->] (A3) -- (A4);}
\node (Sm11) at (-1.5,-1.5) {$\phantom{S_0}$};
\node (Pm11) at (-1.5,-2) {$\phantom{P_1}$};
\node (S00) at (0,-0.5) {$S_0$};
\node (P00) at (0,-1) {$P_1$};
\node (S01) at (0,-1.5) {$S_1[-1]$};
\node (P01) at (0,-2) {$P_0$};
\node (S10) at (2,-0.5) {$S_0[-1]$};
\node (P10) at (2,-1) {$P_1$};
\node (S11) at (2,-1.5) {$S_1$};
\node (P11) at (2,-2) {$P_0$};
\node (S20) at (4,-0.5) {$S_0$};
\node (P20) at (4,-1) {$P_1$};
\node (S21) at (4,-1.5) {$S_1[-1]$};
\node (P21) at (4,-2) {$P_0$};
\opt{ams}{\node (S30) at (6,-0.5) {$S_0[-1]$};}
\opt{ip}{\node (S30) at (5.5,-0.5) {$\phantom{S_0}$};}
\opt{ams}{\node (P30) at (6,-1) {$P_1$};}
\opt{ip}{\node (P30) at (5.5,-1) {$\phantom{S_0}$};}
\opt{ams}{\node (S31) at (6,-1.5) {$S_1$};
\node (P31) at (6,-2) {$P_0$};
\node (S41) at (7.5,-1.5) {$\phantom{S_0}$};
\node (P41) at (7.5,-2) {$\phantom{P_1}$};}
\draw[|->] (Sm11) -- (S01);
\draw[|->] (Pm11) -- (P01);
\draw[|->] (S00) -- (S10);
\draw[|->] (P00) -- (P10);
\draw[|->] (S11) -- (S21);
\draw[|->] (P11) -- (P21);
\draw[|->] (S20) -- (S30);
\draw[|->] (P20) -- (P30);
\opt{ams}{\draw[|->] (S31) -- (S41);
\draw[|->] (P31) -- (P41);}
\end{tikzpicture}
\end{array}\label{S and Ps}
\end{equation}
We index these functors using their domain, and thus write the following. 
\[
\begin{tikzpicture}[scale=\opt{ip}{1.15}\opt{ams}{1},xscale=1.3]
\node (A-1) at (-1.5,0) {$\hdots$};
\node (A0) at (0,0) {$\Db(\mod\Lambda_{-1})$};
\node (A1) at (2,0) {$\Db(\mod\Lambda)$};
\node (A2) at (4,0) {$\Db(\mod\Lambda_{1})$};
\opt{ip}{\node (A3) at (5.5,0) {$\hdots$};}
\opt{ams}{\node (A3) at (6,0) {$\Db(\mod\Lambda_{2})$};}
\opt{ams}{\node (A4) at (7.5,0) {$\hdots$};}
\draw[->] (A-1) -- node[above] {$\scriptstyle \Phi_{-2}$}(A0);
\draw[->] (A0) -- node[above] {$\scriptstyle \Phi_{-1}$}(A1);
\draw[->] (A1) -- node[above] {$\scriptstyle \Phi_0$}(A2);
\draw[->] (A2) -- node[above] {$\scriptstyle \Phi_1$}(A3);
\opt{ams}{\draw[->] (A3) -- node[above] {$\scriptstyle \Phi_2$}(A4);}
\end{tikzpicture}
\]
Considering $T_{BA}$ instead of $T_{AB}$ in \eqref{gen mut functor} there are also equivalences in the reverse direction, which again by our choice of orderings send
\[
\begin{tikzpicture}[xscale=1.3]
\node (A-1) at (-1.5,0) {$\hdots$};
\node (A0) at (0,0) {$\Db(\mod\Lambda_{-1})$};
\node (A1) at (2,0) {$\Db(\mod\Lambda)$};
\node (A2) at (4,0) {$\Db(\mod\Lambda_{1})$};
\opt{ip}{\node (A3) at (5.5,0) {$\hdots$};}
\opt{ams}{\node (A3) at (6,0) {$\Db(\mod\Lambda_{2})$};
\node (A4) at (7.5,0) {$\hdots$};}
\draw[<-] (A-1) -- (A0);
\draw[<-] (A0) -- (A1);
\draw[<-] (A1) -- (A2);
\draw[<-] (A2) -- (A3);
\opt{ams}{\draw[<-] (A3) -- (A4);}
\node (Sm11) at (-1.5,-1.5) {$\phantom{S_0}$};
\node (Pm11) at (-1.5,-2) {$\phantom{P_1}$};
\node (S00) at (0,-0.5) {$S_0[-1]$};
\node (P00) at (0,-1) {$P_1$};
\node (S01) at (0,-1.5) {$S_1$};
\node (P01) at (0,-2) {$P_0$};
\node (S10) at (2,-0.5) {$S_0$};
\node (P10) at (2,-1) {$P_1$};
\node (S11) at (2,-1.5) {$S_1[-1]$};
\node (P11) at (2,-2) {$P_0$};
\node (S20) at (4,-0.5) {$S_0[-1]$};
\node (P20) at (4,-1) {$P_1$};
\node (S21) at (4,-1.5) {$S_1$};
\node (P21) at (4,-2) {$P_0$};
\opt{ip}{\node (S30) at (5.5,-0.5) {$\phantom{S_0}$};
\node (P30) at (5.5,-1) {$\phantom{P_1}$};}
\opt{ams}{\node (S30) at (6,-0.5) {$S_0$};
\node (P30) at (6,-1) {$P_1$};
\node (S31) at (6,-1.5) {$S_1[-1]$};
\node (P31) at (6,-2) {$P_0$};
\node (S41) at (7.5,-1.5) {$\phantom{S_0}$};
\node (P41) at (7.5,-2) {$\phantom{P_1}$};}
\draw[<-|] (Sm11) -- (S01);
\draw[<-|] (Pm11) -- (P01);
\draw[<-|] (S00) -- (S10);
\draw[<-|] (P00) -- (P10);
\draw[<-|] (S11) -- (S21);
\draw[<-|] (P11) -- (P21);
\draw[<-|] (S20) -- (S30);
\draw[<-|] (P20) -- (P30);
\opt{ams}{\draw[<-|] (S31) -- (S41);
\draw[<-|] (P31) -- (P41);}
\end{tikzpicture}
\]
We index these functors using their codomain.  Combining gives the following `strip' of functors.
\begin{equation}
\begin{array}{c}
\begin{tikzpicture}[xscale=\opt{ams}{1.3}\opt{ip}{1.35}]
\node (A0) at (-0.2,0) {$\Db(\mod\Lambda_{-1})$};
\node (A1) at (2,0) {$\Db(\mod\Lambda)$};
\node (A2) at (4,0) {$\Db(\mod\Lambda_{1})$};
\opt{ams}{\node (A3) at (6,0) {$\Db(\mod\Lambda_{2})$};}
\draw[->] (-1.6,0.05) -- node[above] {$\scriptstyle \Phi_{-2}$}(-1.1,0.05);
\draw[<-] (-1.6,-0.05) -- node[below] {$\scriptstyle \Phi_{-2}$} (-1.1,-0.05);
\draw[->] (0.65,0.05) -- node[above] {$\scriptstyle \Phi_{-1}$}(1.2,0.05);
\draw[<-] (0.65,-0.05) -- node[below] {$\scriptstyle \Phi_{-1}$} (1.2,-0.05);
\draw[->] (2.75,0.05) -- node[above] {$\scriptstyle \Phi_0$}(3.2,0.05);
\draw[<-] (2.75,-0.05) -- node[below] {$\scriptstyle \Phi_0$} (3.2,-0.05);
\draw[->] (4.75,0.05) -- node[above] {$\scriptstyle \Phi_1$}(5.2,0.05);
\draw[<-] (4.75,-0.05) -- node[below] {$\scriptstyle \Phi_1$} (5.2,-0.05);
\opt{ams}{\draw[->] (6.75,0.05) -- node[above] {$\scriptstyle \Phi_2$}(7.25,0.05);
\draw[<-] (6.75,-0.05) -- node[below] {$\scriptstyle \Phi_2$} (7.25,-0.05);}
\end{tikzpicture}
\end{array}\label{functors strip 1}
\end{equation}
The notation is deliberate: comparing to \eqref{arrangement labelled} we see that the categories above correspond to the chambers, and the functors correspond to wall-crossing. Indeed, the functor labels above are precisely induced from the label on the corresponding wall.
\[
\begin{tikzpicture}
\draw[densely dotted,->] (-2.5,0)--(\opt{ams}{9}\opt{ip}{6.5},0);
\node (A) at (-2.2,0) [cvertex] {};
\node (B) at (0.8,0) [cvertex] {};
\node (C) at (3.5,0) [cvertex] {};
\node (D) at (6,0) [cvertex] {};
\opt{ams}{\node (E) at (8.7,0) [cvertex] {};}
\node at (-2.2,0.3) {$\scriptstyle {-2}$};
\node at (0.8,0.3) {$\scriptstyle {-1}$};
\node at (3.5,0.3) {$\scriptstyle {0}$};
\node at (6,0.3) {$\scriptstyle {1}$};
\opt{ams}{\node at (8.7,0.3) {$\scriptstyle {2}$};}
\draw[draw=none] (A) -- node[gap] {$V_{-2}\oplus V_{-1}$}(B);
\draw[draw=none] (B) -- node[gap] {$V_{-1}\oplus V_0$}(C);
\draw[draw=none] (C) -- node[gap] {$V_{0}\oplus V_1$}(D);
\opt{ams}{\draw[draw=none] (D) -- node[gap] {$V_{1}\oplus V_2$}(E);}
\end{tikzpicture}
\quad
\]
Later, we will see that monodromy in the complexified complement around wall $i$ will correspond to the composition $\Phi_i\circ\Phi_i$.

For any fixed $i$, both functors labelled $\Phi_i$ are governed by the same numerics.  These mutation functors are governed by the exchange sequences \eqref{exchange 1} and \eqref{exchange 2} respectively, and, as has already been noted, $n_i=n_i'$.  The following will be used heavily later.
\begin{lemma}\label{mut track lem}
For \opt{ams}{a fixed}\opt{ip}{any} $i$, consider either of the functors labelled\opt{ams}{ }\opt{ip}{~}$\Phi_i$.  Under the $(P_0,P_1)$ and  $(S_0,S_1)$ ordering conventions above, the following hold.
\begin{enumerate}
\item If $i$ is even, then 
$\Phi_i(S_1)=S_1[-1]$ and $\Phi_i(S_0)$ is a module of dim vector $(1,n_i)$.
\item If $i$ is odd, then
 $\Phi_i(S_0)=S_0[-1]$ and $\Phi_i(S_1)$ is a module of dim vector $(n_i,1)$.
\end{enumerate}
\end{lemma}
\begin{proof}
Consider $\Phi_i\colon \Db(\mod\Lambda_i)\to\Db(\mod\Lambda_{i+1})$. The proof for $\Phi_i\colon \Db(\mod\Lambda_{i+1})\to\Db(\mod\Lambda_{i})$ is similar.   

Being the mutation functor, $\Phi_i$ is induced by the tilting bimodule $\Hom_R(M_i,M_{i+1})$, which as a $\Lambda_i$-module decomposes as $\Hom_R(M_i, V_{i})\oplus\Hom_R(M_i, V_{i+1})$. This is either $P_0\oplus \Hom_R(M_i, V_{i+1})$ if $i$ is even, or $P_1\oplus\Hom_R(M_i, V_{i+1})$ if $i$ is odd.   The first statement regarding shifting $S_1$, respectively $S_0$, is already in \eqref{S and Ps}. For the second statement, we assume that $i$ is even: the same proof below works in the case $i$ is odd, simply by swapping subscripts $1$ and $0$ throughout. 

\def\eqnRHom{\RHom_{\Lambda_i}(\Hom_R(M_i,V_{i+1}),S_0)=\mathbb{C}^{\oplus n_i}}

The exchange sequence induces an exact sequence
\[
0\to\Hom_R(M_i,V_{i - 1})\to
\Hom_R(M_i,V_{i})^{\oplus n_i}\to
\Hom_R(M_i,V_{i+1})\to
0
\]
which since $i$ is even is
\[
0\to P_1\to P_0^{\oplus n_i}\to\Hom_R(M_i,V_{i+1})\to 0.
\]
Applying $\Hom_{\Lambda_i}(-,S_0)$ we deduce that \opt{ams}{$\eqnRHom$.}\opt{ip}{\[\eqnRHom.\]}  It is clear that $\RHom_{\Lambda_i}(P_0,S_0)=\mathbb{C}$, and so the result follows.
\end{proof}

\subsection{Identifying Line Bundle Twists}\label{line bundle twists section}
The subgroup of the class group generated by $L=f_*\scrO(1)$ acts as in \ref{class action comb}, so for all $i\in\mathbb{Z}$ there is an isomorphism
\[
\Lambda_i=\End_R(V_{i-1}\oplus V_i)\xrightarrow{(-\otimes L)^{**}}\End_R(V_{i+N-1}\oplus V_{i+N})=\Lambda_{i+N}.
\]
This isomorphism relates the ordered projectives for $\Lambda_i$ and $\Lambda_{i+N}$ as follows: when $N$ is odd (which occurs precisely when $\ell=1$, in which case $N=1$) it permutes the ordered projectives, whereas when $N$ is even (which occurs precisely when $\ell\geq 2$) it preserves the ordering. For later reference, we summarize this as follows.

\begin{notation}
For any $i$, by abuse of notation we write $\upbeta$ for the above isomorphism. In particular, $\upbeta$ induces an isomorphism of categories $\Db(\mod\Lambda_i)\to\Db(\mod\Lambda_{i+N})$, which we will also denote by\opt{ams}{ }\opt{ip}{~}$\upbeta$.
\end{notation}

\begin{lemma}\label{beta order}
$\upbeta$ is a ring isomorphism that sends $P_i\mapsto P_i$ and $S_i\mapsto S_i$ when $\ell>1$, and interchanges $P_0\leftrightarrow P_1$ and $S_0\leftrightarrow S_1$ when $\ell=1$.  
\end{lemma}

It will be convenient to write $\upkappa_0=\Id$, and for $i\geq 1$ to set
\[
\upkappa_i\colonequals \Phi_{i-1}\circ\hdots\circ\Phi_0\colon \Db(\mod\Lambda)\xrightarrow{\sim}\Db(\mod\Lambda_i).
\]
For $i<0$, we set
\[
\uplambda_i\colonequals\Phi_{-1}\circ \hdots\circ\Phi_i\colon \Db(\mod\Lambda_i)\xrightarrow{\sim}\Db(\mod\Lambda).
\]
We also define a functor $\Db(\coh X) \to \Db(\mod\Lambda_i)$ as follows.
\[
\Uppsi_i\colonequals
\left\{
\begin{array}{ll}
\upkappa_i\circ\Uppsi&\mbox{if $i\geq 0$}\\
\uplambda_i^{-1}\circ\Uppsi&\mbox{if $i<0$}
\end{array}
\right.
\]
It is obvious that if $i\in\mathbb{Z}$, and $j\geq 1$, then
\begin{equation}
\Phi_{i+j-1}\circ\hdots\circ\Phi_{i}\circ\Uppsi_i
\cong \Uppsi_{i+j}.
\label{obvious comp}
\end{equation}
The following is an easy extension of \cite[7.4]{HW}.
\begin{prop}\label{tensor cor}
For all $i\in\mathbb{Z}$, and $k\geq 1$, the following diagram commutes.
\[
\begin{tikzpicture}
\node (A0) at (0,1.5) {$\Db(\coh X)$};
\node (A5) at (5,1.5) {$\Db(\coh X)$};
\node (B0) at (0,0) {$\Db(\mod\Lambda_i)$};
\node (N5) at (5,0) {$\Db(\mod\Lambda_{i+kN})$};
\draw[->] (A0) -- node[above] {$\scriptstyle -\otimes\scrO(-k)$}(A5);
\draw[->] (B0) -- node[above] {$\scriptstyle \Phi_{i+kN-1}\circ\hdots\circ\Phi_i$}(N5);
\draw[<-] (B0) -- node[left] {$\scriptstyle \Uppsi_i$}(A0);
\draw[<-] (N5) -- node[left] {$\scriptstyle \upbeta\circ \Uppsi_i$}(A5);
\end{tikzpicture}
\]
\end{prop}
\begin{proof}
Consider the diagram 
\[
\begin{tikzpicture}
\node (A0) at (\opt{ams}{-1}\opt{ip}{-0.5},2) {$\Db(\opt{ams}{\coh} X)$};
\node (A1) at (6,2) {$\Db(\opt{ams}{\coh} X)$};
\node (B) at (6,1) {$\Db(\opt{ams}{\mod}\Lambda_0)$};
\node (C0) at (\opt{ams}{-1}\opt{ip}{-0.5},0) {$\Db(\opt{ams}{\mod}\Lambda_0)$};
\node (C1) at (\opt{ams}{1.5}\opt{ip}{2},0) {$\Db(\opt{ams}{\mod}\Lambda_1)$};
\node (C2) at (6,0) {$\Db(\opt{ams}{\mod}\Lambda_{N})$};
\node (B1) at (9,1) {$\Db(\opt{ams}{\mod}\Lambda_1)$};
\node (C3) at (9,0) {$\Db(\opt{ams}{\mod}\Lambda_{N+1})$};
\draw[->] (A0) -- node[above] {$\scriptstyle -\otimes\scrO(-1)$}(A1);
\draw[->] (C0) -- node[above] {$\scriptstyle\Phi_0$}(C1);
\draw[->] (C1) -- node[above] {$\scriptstyle\Phi_{N-1}\circ\hdots\circ\Phi_1$}(C2);
\draw[->] (A0) -- node[left] {$\scriptstyle \Uppsi_0$}(C0);
\draw[->] (A1) -- node[left] {$\scriptstyle \Uppsi_0$}(B);
\draw[->] (B) -- node[left] {$\scriptstyle \upbeta$}(C2);
\draw[->] (B1) -- node[left] {$\scriptstyle \upbeta$}(C3);
\draw[->] (B) -- node[above] {$\scriptstyle\Phi_0$}(B1);
\draw[->] (C2) -- node[above] {$\scriptstyle\Phi_{N}$}(C3);
\draw[densely dotted,->] (A0) -- node[right]{$\scriptstyle\Uppsi_1$}(C1);
\draw[densely dotted,->] (A1) -- node[above]{$\scriptstyle\Uppsi_1$}(B1);
\end{tikzpicture}
\]
The left hand square commutes by \cite[7.4]{HW}. We claim the right hand square also commutes.  On one hand, $\upbeta\circ\Phi_0$ is given by the tilting bimodule $
\Hom_R(M_0,M_1)$, with standard action of $\End_R(M_0)$ and the action of $\End_R(M_{N+1})$ via $\upbeta^{-1}$.  On the other hand,  $\Phi_N\circ\upbeta$ is given by the tilting bimodule $\Hom_R(M_{N},M_{N+1})$, with the action of $\End_R(M_0)$ via $\upbeta$, and the standard action by $\End_R(M_{N+1})$.   These tilting modules are isomorphic via tensor by $L$.  This is clearly an isomorphism of bimodules since both functors~$\upbeta$ are induced by tensoring by $L$.

Since $\Uppsi_1 = \Phi_0\circ\Uppsi$, from the above two commutative squares, the result follows for $i=1$ and $k=1$.   In an identical way, composing suitable squares proves the result for all $i\geq 0$ and $k=1$.  Then, for any $i\geq 0$, all three squares in the following diagram commute
\[
\begin{tikzpicture}[yscale=\opt{ams}{1}\opt{ip}{1.1}]
\node (A0) at (-1,2) {$\Db(\opt{ams}{\coh} X)$};
\node (A1) at (3,2) {$\Db(\opt{ams}{\coh} X)$};
\node (A2) at (7,2) {$\Db(\opt{ams}{\coh} X)$};
\node (B) at (3,0) {$\Db(\opt{ams}{\mod}\Lambda_i)$};
\node (C0) at (-1,-1) {$\Db(\opt{ams}{\mod}\Lambda_i)$};
\node (C2) at (3,-1) {$\Db(\opt{ams}{\mod}\Lambda_{i+N})$};
\node (X) at (7,1) {$\Db(\opt{ams}{\mod}\Lambda_i)$};
\node (B1) at (7,0) {$\Db(\opt{ams}{\mod}\Lambda_{i+N})$};
\node (C3) at (7,-1) {$\Db(\opt{ams}{\mod}\Lambda_{2i+N})$};
\draw[->] (A0) -- node[above] {$\scriptstyle -\otimes\scrO(-1)$}(A1);
\draw[->] (A1) -- node[above] {$\scriptstyle -\otimes\scrO(-1)$}(A2);
\draw[->] (C0) -- node[above] {$\scriptstyle\hdots\circ\Phi_i$}(C2);
\draw[->] (A0) -- node[left] {$\scriptstyle \Uppsi_i$}(C0);
\draw[->] (A1) -- node[left] {$\scriptstyle \Uppsi_i$}(B);
\draw[->] (B) -- node[left] {$\scriptstyle \upbeta$}(C2);
\draw[->] (B1) -- node[left] {$\scriptstyle \upbeta$}(C3);
\draw[->] (A2) -- node[left] {$\scriptstyle \Uppsi_i$}(X);
\draw[->] (X) -- node[left] {$\scriptstyle \upbeta$}(B1);
\draw[->] (B) -- node[above] {$\scriptstyle\hdots\circ\Phi_i$}(B1);
\draw[->] (C2) -- node[above] {$\scriptstyle\hdots\circ\Phi_{i+N}$}(C3);
\end{tikzpicture}
\]
which proves the result for all $i\geq 0$ and all $k=2$.  In a similar way, the result follows for all $i\geq 0$ and all $k\geq 1$.  The result for all $i<0$ and $k\geq 1$ is proved in a very similar way, starting with the following commutative diagram
\[
\begin{tikzpicture}[yscale=\opt{ams}{1}\opt{ip}{1.1}]
\node (A0) at (\opt{ams}{-1}\opt{ip}{0},2) {$\Db(\opt{ams}{\coh} X)$};
\node (A1) at (6,2) {$\Db(\opt{ams}{\coh} X)$};
\node (B) at (3,1) {$\Db(\opt{ams}{\mod}\Lambda_{-1})$};
\node (Cm1) at (\opt{ams}{-4}\opt{ip}{-3},0) {$\Db(\opt{ams}{\mod}\Lambda_{-1})$};
\node (C0) at (\opt{ams}{-1}\opt{ip}{0},0) {$\Db(\opt{ams}{\mod}\Lambda_0)$};
\node (C2) at (3,0) {$\Db(\opt{ams}{\mod}\Lambda_{N-1})$};
\node (B1) at (6,1) {$\Db(\opt{ams}{\mod}\Lambda_0)$};
\node (C3) at (6,0) {$\Db(\opt{ams}{\mod}\Lambda_{N}).$};
\draw[->] (A0) -- node[above] {$\scriptstyle -\otimes\scrO(-1)$}(A1);
\draw[->] (Cm1) -- node[above] {$\scriptstyle\Phi_{-1}$}(C0);
\draw[->] (C0) -- node[above] {$\scriptstyle\upkappa_{N-1}$}(C2);
\draw[->] (A0) -- node[left] {$\scriptstyle \Uppsi_0$}(C0);
\draw[->] (A1) -- node[left] {$\scriptstyle \Uppsi_0$}(B1);
\draw[->] (B) -- node[left] {$\scriptstyle \upbeta$}(C2);
\draw[->] (B1) -- node[left] {$\scriptstyle \upbeta$}(C3);
\draw[->] (B) -- node[above] {$\scriptstyle\Phi_{-1}$}(B1);
\draw[->] (C2) -- node[above] {$\scriptstyle\Phi_{N-1}$}(C3);
\draw[densely dotted,->] (A0) -- node[above left]{$\scriptstyle\Uppsi_{-1}$}(Cm1);
\draw[densely dotted,->] (A1) -- node[above]{$\scriptstyle\Uppsi_{-1}$}(B);
\end{tikzpicture}\qedhere
\]
\end{proof}

\section{t-structures on the Strip, Duality and the Simples Helix}

This section establishes the existence of the simples helix $\{\scrS_i\}_{i\in\mathbb{Z}}$, and shows that these describe the simples in iterated tilts of perverse sheaves.  In later sections, this allows us to give an intrinsic description of monodromy on the SKMS.

\subsection{Generalities on t-structures}
Recall that a t-structure on a triangulated category $\scrD$ is a full subcategory $\scrF\subset\scrD$, satisfying $\scrF[1]\subset\scrF$, such that setting
\[
\scrF^\perp\colonequals \{ d\in\scrD\mid \Hom_{\scrD}(f,d)=0\mbox{ for all }f\in\scrF\}
\] 
then for every $d\in\scrD$ there is a triangle $f\to d\to g\to$ with $f\in\scrF$ and $g\in\scrF^\perp$. The heart is defined to be
\[
\scrA=\scrF\cap\scrF^\perp[1].
\]
A t-structure $\scrF\subset\scrD$ is called bounded if $\scrD=\bigcup_{i,j\in\mathbb{Z}}\scrF[i]\cap\scrF^\perp[j]$.  A\opt{ams}{ }\opt{ip}{~}bounded t-structure $\scrF$ is determined by its heart $\scrA$. Indeed, $\scrF$ is the extension-closed subcategory generated by the subcategories $\scrA[j]$ for integers $j\geq 0$.  

In what follows, recall that the truncation functor $\uptau^\scrA_{\leq 0}$ is defined to be the right adjoint to the inclusion $\scrF\subset\scrD$, and $\uptau^\scrA_{\geq 0}$ as left adjoint to the inclusion $\scrF^\perp[1]\subset \scrD$. 

\begin{lemma}\label{dual aisles}
Suppose that $\scrA,\scrB$ are hearts of bounded t-structures in $\scrD$, such that $\mathbb{D}\scrA=\scrB$ for some exact duality $\mathbb{D}\colon\scrD^{\op} \to\scrD$, namely an exact anti-autoequivalence with $\mathbb{D}^2 = \Id$.  Then  $\mathbb{D}\uptau_{\geq0}^\scrA\cong\uptau^\scrB_{\leq 0}\mathbb{D}$ and $\mathbb{D}\uptau^\scrA_{\leq 0}\cong\uptau_{\geq0}^\scrB\mathbb{D}$.
\end{lemma}
\begin{proof}
Suppose that $\scrF$ and $\scrG$ are the two bounded t-structures that give the hearts $\scrA=\scrF\cap\scrF^\perp[1]$ and $\scrB=\scrG\cap\scrG^\perp[1]$. Since $\scrG$ is the extension-closed subcategory generated by $\scrB[j]$ for integers $j\geq 0$, it follows that $\mathbb{D}(\scrG)$ is the extension-closed subcategory generated by $\scrA[j]$ for $j\leq 0$.  This is $\scrF^\perp[1]$. Thus $\mathbb{D}$ restricts to a duality
\[
\scrF^\perp[1]\longleftrightarrow \scrG.
\]
The chain of functorial isomorphisms
\def\eqnHomG{\Hom_{\scrG}(g,\uptau^\scrB_{\leq 0}\mathbb{D}d) \opt{ip}{&} \cong
\Hom_{\scrD}(d,\mathbb{D}g) \opt{ip}{\\
&} \cong
\Hom_{\scrF^\perp[1]}(\uptau_{\geq0}^\scrA d,\mathbb{D}g) \opt{ip}{\\
&} \cong
\Hom_{\scrG}(g,\mathbb{D}\uptau_{\geq0}^\scrA d)}
\opt{ams}{\[
\eqnHomG
\]}
\opt{ip}{\[
\begin{split}
\eqnHomG
\end{split}
\]}
then shows that $\mathbb{D}\uptau_{\geq0}^\scrA\cong\uptau^\scrB_{\leq 0}\mathbb{D}$.  Swapping $\scrA$ and $\scrB$, we also see that $\uptau^\scrA_{\leq 0}\mathbb{D}\cong\mathbb{D}\uptau_{\geq0}^\scrB$.  Applying $\mathbb{D}$ on both sides gives the the second isomorphism.
\end{proof}
Now suppose that $\scrA$ is the finite length heart of a bounded t-structure in a triangulated category $\scrD$. Each of the simple objects $S\in\scrA$ induces two torsion theories, $(\langle S \rangle, \scrF)$  and $(\scrT,\langle S \rangle)$ on~$\scrA$, where $\langle S \rangle$ is the full subcategory of objects whose simple factors are isomorphic to $S$, and the subcategories $\scrF$ and $\scrT$ are defined by 
\begin{align*}
\scrF&\colonequals \{a\in\scrA\mid\Hom_{\scrA}(S,a)=0\}\\
\scrT&\colonequals \{a\in\scrA\mid\Hom_{\scrA}(a,S)=0\}.
\end{align*}
The corresponding tilted hearts \opt{ip}{in $\scrD$ }are defined by
\begin{align*}
\lefttilt_{S}(\scrA)&\colonequals 
\{ d\opt{ams}{\in\scrD} \mid 
{\rm H}_{\scrA}^i(d)=0 \mbox{ for } i\notin\{0,1\},\, {\rm H}_{\scrA}^0(d)\in\scrF \opt{ams}{\mbox{ and }}\opt{ip}{, \,} {\rm H}_{\scrA}^1(d)\in\langle S\rangle \}\\
\righttilt_{S}(\scrA)&\colonequals 
\{ d\opt{ams}{\in\scrD} \mid {\rm H}_{\scrA}^i(d)=0 \mbox{ for } i\notin\{-1,0\}, \,{\rm H}_{\scrA}^0(d)\in\scrT \opt{ams}{\mbox{ and }}\opt{ip}{, \,} {\rm H}^{-1}_{\scrA}(d)\in\langle S\rangle  \},
\end{align*}
where $\mathrm{H}_{\scrA}(-)$ is the cohomological functor associated to the t-structure defining $\scrA$.  The heart $\righttilt_S(\scrA)$ is called the \emph{right tilt} of $\scrA$ with respect to the simple $S$, and $\lefttilt_S(\scrA)$ is called the \emph{left tilt} at $S$.

\begin{lemma}\label{t structure D duality}
Suppose that $\scrA$ and $\scrB$ are finite length hearts of bounded t-structures. If $\mathbb{D}\scrA=\scrB$, and $S$ is a simple in $\scrA$, then $\lefttilt_{\mathbb{D}S}(\scrB)=\mathbb{D}(\righttilt_{S}\scrA)$, so that $\mathbb{D}$ restricts to a duality 
\[
\lefttilt_{\mathbb{D}S}(\scrB)\longleftrightarrow \righttilt_{S}(\scrA).
\] 
\end{lemma}
\begin{proof}
We claim that  for all $i\in\mathbb{Z}$
\begin{equation}
{\rm H}^{-i}_\scrA(\mathbb{D}d)=\mathbb{D}({\rm H}^i_\scrB(d)).\label{cohom duality}
\end{equation}
The case $i=0$ holds since by repeated use of \ref{dual aisles}
\[
{\rm H}^{0}_\scrA(\mathbb{D}d)
=\uptau^{\scrA}_{\geq 0}\uptau^{\scrA}_{\leq 0}\mathbb{D}d
=\uptau^{\scrA}_{\geq 0}\mathbb{D}\uptau^{\scrB}_{\geq 0}d
=\mathbb{D}\uptau^{\scrB}_{\leq 0}\uptau^{\scrB}_{\geq 0}d
=\mathbb{D}({\rm H}^{0}_\scrB(d)).
\]
The general case of \eqref{cohom duality} then holds by the following chain of equalities
\[
{\rm H}^{-i}_\scrA(\mathbb{D}d)
={\rm H}^{0}_\scrA(\mathbb{D}d[-i])
={\rm H}^{0}_\scrA(\mathbb{D}(d[i]))
=\mathbb{D}({\rm H}^{0}_\scrB(d[i]))
=\mathbb{D}({\rm H}^{i}_\scrB(d))
\]
where the first and last equalities hold by definition, the second is obvious, and the third is the case $i=0$ above.  
Now, by definition
\[
\lefttilt_{\mathbb{D}S}(\scrB)
=
\left\{d\in\scrD 
\left| 
\begin{array}{l}
{\rm H}_{\scrB}^i(d)=0 \mbox{ for } i\notin\{0,1\},\\
\Hom_{\scrB}(\mathbb{D}S,{\rm H}_{\scrB}^0(d))=0,\\ 
{\rm H}_{\scrB}^1(d)\in\langle \mathbb{D}S\rangle
\end{array}
\right. 
\right\}.
\]
Consider the first condition, and note that
\begin{align*}
{\rm H}_{\scrB}^i(d)=0 \mbox{ for } i\notin\{0,1\}
&\iff 
\mathbb{D}({\rm H}_{\scrB}^i(d))=0\mbox{ for } i\notin\{0,1\}\opt{ams}{\tag{by duality}}\\
&\iff 
{\rm H}_{\scrA}^i(\mathbb{D}d)=0\mbox{ for } i\notin\{0,-1\}.\tag{by \eqref{cohom duality}}
\end{align*}
Similarly, for the second condition
\begin{align*}
\Hom_{\scrB}(\mathbb{D}S,{\rm H}_{\scrB}^0(d))=0
&\iff \Hom_{\scrA}(\mathbb{D}({\rm H}_{\scrB}^0(d)),S)=0\opt{ams}{\tag{by duality}}\\
&\iff \Hom_{\scrA}({\rm H}_{\scrA}^0(\mathbb{D}d),S)=0.\tag{by \eqref{cohom duality}}
\end{align*}
Finally, by duality it is clear that the third condition is equivalent to $\mathbb{D}({\rm H}_{\scrB}^1(d))\in\langle S\rangle$, which again by \eqref{cohom duality} is equivalent to ${\rm H}_{\scrA}^{-1}(\mathbb{D}d)\in\langle S\rangle$.  

Summarising the above, it follows that $d\in\lefttilt_{\mathbb{D}S}(\scrB)$ if and only if $\mathbb{D}d$ satisfies the three conditions defining $\righttilt_S(\scrA)$.  Consequently, $\mathbb{D}$ restricts to the desired duality.  
\end{proof}

\subsection{Iterated Algebraic Tilts}
Returning to the flops setting, with ordering as in \S\ref{ordering section}, and notation as in \S\ref{mut notation section}, $\scrA_i=\mod\Lambda_i$ is the heart of a bounded t-structure on $\Db(\mod\Lambda_i)$, with projectives $P_0$ and $P_1$, and simples $S_0$ and $S_1$.  For our applications later, we will need to tilt both $\scrA_i$ and its subcategory of finite length modules, written $\finitelength\scrA_i$.

The case of $\finitelength\scrA_i$ is slightly easier.  By \cite[2.5]{IR} we have $\Db(\finitelength \Lambda_i)=\mathrm{D}^{\mathrm{b}}_{\finitelength \Lambda_i}(\mod\Lambda_i)$, and so $\finitelength\scrA_i=\finitelength\Lambda_i$ is the finite length heart of a bounded t-structure in this subcategory.  Appealing to the previous subsection,  for each of the $S_j$, write the torsion theories as $(\langle S_j \rangle, \scrF_j)$  and $(\scrT_j,\langle S_j \rangle)$, with resulting tilted hearts $\lefttilt_{j}(\finitelength\scrA_i)$ and $\righttilt_{j}(\finitelength\scrA_i)$.

\begin{prop}[{\cite[5.5]{HW}}]\label{finite length ok!}
The $\lefttilt_{j}(\finitelength\scrA_i)$ and $\righttilt_{j}(\finitelength\scrA_i)$ are the image of a standard heart under the mutation functors and their inverses, as follows:
\[
\begin{array}{c}
\begin{tikzpicture}[xscale=\opt{ams}{1.3}\opt{ip}{1.35}]
\node at (-1.5,0) {$\hdots$};
\node at (0,0) {$\Db(\mod\Lambda_{-1})$};
\node at (2,0) {$\Db(\mod\Lambda)$};
\node at (4,0) {$\Db(\mod\Lambda_{1})$};
\node at (5.5,0) {$\hdots$};
\draw[->] (0.85,0.05) -- node[above] {$\scriptstyle \Phi_{-1}$}(1.3,0.05);
\draw[<-] (0.85,-0.05) -- node[below] {$\scriptstyle \Phi_{-1}$} (1.3,-0.05);
\draw[->] (2.75,0.05) -- node[above] {$\scriptstyle \Phi_0$}(3.2,0.05);
\draw[<-] (2.75,-0.05) -- node[below] {$\scriptstyle \Phi_0$} (3.2,-0.05);
\node (L1Am1) at (0,-0.75) {$\lefttilt_1(\finitelength\scrA_{-1})$};
\node (R1Am1) at (0,-1.75) {$\righttilt_0(\finitelength\scrA_{-1})$};
\node (A) at (2,-1.25) {$\finitelength\scrA$};
\node (L1A1) at (4,-0.75) {$\lefttilt_1(\finitelength\scrA_1)$};
\node (R1A1) at (4,-1.75) {$\righttilt_0(\finitelength\scrA_1)$};
\draw[->] (A) -- (L1A1);
\draw[->] (A) -- (L1Am1);
\draw[->] (R1Am1) -- (A);
\draw[->] (R1A1) -- (A);
\end{tikzpicture}
\end{array}
\]
\end{prop}

\begin{defin}\label{def Ti A}
For $i\geq 1$, consider the iterated tilt
\[
\Tilt_i(\finitelength\scrA)\colonequals\underbrace{\hdots \righttilt_{1}\righttilt_{0}\righttilt_{1}}_{i}(\finitelength\scrA),
\]
that is, we first tilt $\finitelength\scrA$ at $S_1$.  The resulting t-structure $\righttilt_1(\finitelength\scrA)$ has two simples, being equivalent to $\finitelength\scrA_1$, one of which is $S_1[1]$.  We then right tilt $\righttilt_1(\finitelength\scrA)$ at the other, new, simple to obtain $\righttilt_0\righttilt_1(\finitelength\scrA)$.  We then right tilt this t-structure at its new simple, and repeat. Similarly, for  $i<0$ consider the iterated tilt
\[
\Tilt_i(\finitelength\scrA)\colonequals\underbrace{\hdots \lefttilt_{0}\lefttilt_{1}\lefttilt_{0}}_{-i}(\finitelength\scrA).
\]
For all $i\in\mathbb{Z}$, we call $\Tilt_i(\finitelength\scrA)$ the $i$th tilt of the heart $\finitelength\scrA$. 
\end{defin}

In contrast, we tilt the full abelian category $\scrA$ using the tilting modules and notation introduced in \S\ref{line bundle twists section}.  This can also be achieved using iterated tilts at torsion theories (see e.g.\ \cite[(5.A)]{HW1}), but for our purposes here, the following definition suffices.

\begin{defin}\label{def Ti}
For $i\geq 1$, consider the iterated tilt
\[
\Tilt_i(\scrA)\colonequals(\Phi_{i-1}\circ\hdots\circ\Phi_0)^{-1}(\scrA_i).
\]
Similarly, for  $i<0$ consider the iterated tilt
\[
\Tilt_i(\scrA)\colonequals(\Phi_{-1}\circ \hdots\circ\Phi_i)(\scrA_i).
\]
For all $i\in\mathbb{Z}$, we call $\Tilt_i(\scrA)$ the $i$th tilt of the heart $\scrA$. 
\end{defin}  
The following is an immediate consequence of the definition above, together with the notation $\Uppsi_i$ from \S\ref{line bundle twists section}. Later it will be used to generalise \eqref{Psi per 0}.

\begin{lemma}\label{Psii abelian equiv}
For all $i\in\mathbb{Z}$, there is a commutative diagram as follows.
\[
\begin{array}{c}
\begin{tikzpicture}
\node (A1) at (0,0) {$\Db(\mod\Lambda)$};
\node (A2) at (5,0) {$\Db(\mod\Lambda_i)$};
\node (B1) at (0,-1.5) {$\Tilt_i(\scrA)$};
\node (B2) at (5,-1.5) {$\scrA_i$};
\draw[->] (A1) -- node[above] {$\scriptstyle \Uppsi_i\circ\Uppsi^{-1}$} node[below]{$\scriptstyle\sim$}(A2);
\draw[->] (B1) -- node[above] {$\scriptstyle$} node[below]{$\scriptstyle\sim$}(B2);
\draw[right hook->] (B1) -- (A1);
\draw[right hook->] (B2) -- (A2);
\end{tikzpicture}
\end{array}
\]
\end{lemma}
%

Definitions~\ref{def Ti A} and \ref{def Ti} are compatible, as follows.
\begin{cor}\label{finite length alg}
$\finitelength\Tilt_i(\scrA)=\Tilt_i(\finitelength\scrA)$ for all $i\in\mathbb{Z}$.
\end{cor}
\begin{proof}
The statement that $\finitelength\Tilt_i(\scrA)=\Uppsi\circ\Uppsi_i^{-1}(\finitelength\scrA_i)$ holds tautologically, since the bottom map in \ref{Psii abelian equiv} is an abelian equivalence.  On the other hand, by  \ref{finite length ok!}, $\Uppsi\circ\Uppsi_i^{-1}(\finitelength\scrA_i)=\Tilt_i(\finitelength\scrA)$.
\end{proof}
In what follows, set $\mathbb{D}_\Lambda\colonequals \Uppsi\circ\mathbb{D}\circ\Uppsi^{-1}$, where $\mathbb{D}$ is Grothendieck duality.
\begin{prop}\label{Tilt duality alg}
For all $i\in\mathbb{Z}$, the following hold.
\begin{enumerate}
\item\label{Tilt duality alg 1} $\mathbb{D}_\Lambda(\Tilt_{i}(\finitelength\scrA))=\Tilt_{1-i}(\finitelength\scrA)$.
\item\label{Tilt duality alg 2} $\mathbb{D}_\Lambda$ preserves the ordering on the simples.
\end{enumerate}
\end{prop}
\begin{proof}
Consider first the case $i=1$. As observed by Bridgeland \cite[(4.8)]{Bridgeland}, with corrected sign in \cite[(27)]{Toda GV}, the image of $\Per X^+$ under the flop functor is $\PerOne X$.  Since the flop functor is the inverse of the mutation functor $\Phi_0$ \cite[4.2]{HomMMP},  we see that $\PerOne X=\Uppsi^{-1}(\Tilt_1(\scrA))$.  In particular  $\finitelength\PerOne X=\Uppsi^{-1}(\finitelength\Tilt_1(\scrA))$, and so
\[
\mathbb{D}_\Lambda(\Tilt_{1}(\finitelength\scrA))
\stackrel{\scriptstyle\ref{finite length alg}}{=}
\mathbb{D}_\Lambda(\finitelength\Tilt_{1}(\scrA))
=\Uppsi\circ\mathbb{D}(\finitelength\PerOne X).
\]
But by \cite[3.5.8]{VdB1d} $\mathbb{D}(\finitelength\PerOne X)=\finitelength\Per X$, and so the right hand side equals $\Uppsi(\finitelength\Per X)$.  In turn, this is just $\finitelength \scrA$.  This establishes the statement~\eqref{Tilt duality alg 1} in the case $i=1$.

Applying $\mathbb{D}_\Lambda$ gives $\mathbb{D}_\Lambda(\finitelength\scrA)=\Tilt_{1}(\finitelength\scrA)$.  Now track $S_1\in\finitelength\scrA$. Note first that $\mathbb{D}\colon \scrO_{\Curve}(-1)\mapsto\scrO_{\Curve}(-1)[1]$, hence $\mathbb{D}_\Lambda\colon S_1\mapsto S_1[1]$, which is the first simple in $\Tilt_{1}(\finitelength\scrA)$.  It follows that the zeroth simple must get sent to the zeroth simple.  This then establishes~\eqref{Tilt duality alg 2} in the case $i=1$.  

When $i=2$, apply \ref{t structure D duality} to hearts $\righttilt_1(\finitelength\scrA)$ and $\finitelength\scrA$, with $S$ the zeroth simple of $\righttilt_1(\finitelength\scrA)=\Tilt_1(\finitelength\scrA)$.  Since by the above $\mathbb{D}_\Lambda(S)$ is the zeroth simple of $\finitelength\scrA$, 
\[
\Tilt_{1-i}(\finitelength\scrA)
=\lefttilt_0(\finitelength\scrA)
\stackrel{{\scriptstyle\ref{t structure D duality}}}{=}
\mathbb{D}(\righttilt_0\righttilt_1(\finitelength\scrA))
=
\mathbb{D}(\Tilt_{2}(\finitelength\scrA)),
\]
establishing the case $i=2$.  The fact that $\mathbb{D}_\Lambda$ preserves the ordering on the simples can be seen this time by tracking the zeroth simple, which tracks to the zeroth simple.  The first simple must thus do likewise.   This establishes both~\eqref{Tilt duality alg 1} and~\eqref{Tilt duality alg 2} in the case $i=2$.  The proof then simply proceeds by induction, alternating the simples: for $i=3$, apply \ref{t structure D duality} to $\righttilt_0\righttilt_1(\finitelength\Per)$ and $\lefttilt_0(\finitelength\Per)$, with $S$ the first simple.
\end{proof}

\subsection{Tensors and Dualities on Perverse Tilts}

\begin{defin}\label{def Ti Per}
The $i$th tilt of perverse sheaves, written $\Tilt_i(\Per)$, is defined in an identical way to $\Tilt_i(\scrA)$ in \ref{def Ti}, using the ordering on its simples from \S\ref{ordering section}.
\end{defin}
Equivalently, we can define $\Tilt_i(\Per)\colonequals\Uppsi^{-1}\Tilt_i(\scrA)$.   Combining with \ref{Psii abelian equiv}, for all $i\in\mathbb{Z}$ there is thus a commutative diagram as follows.
\begin{equation}
\begin{array}{c}
\begin{tikzpicture}
\node (A0) at (-4,0) {$\Db(\coh X)$};
\node (A1) at (0,0) {$\Db(\mod\Lambda)$};
\node (A2) at (4,0) {$\Db(\mod\Lambda_i)$};
\node (B0) at (-4,-1.5) {$\Tilt_i(\Per)$};
\node (B1) at (0,-1.5) {$\Tilt_i(\scrA)$};
\node (B2) at (4,-1.5) {$\scrA_i$};
\draw[->] (A0) -- node[above] {$\scriptstyle \Uppsi$} node[below]{$\scriptstyle\sim$}(A1);
\draw[->] (A1) -- node[above] {$\scriptstyle \Uppsi_i\circ\Uppsi^{-1}$} node[below]{$\scriptstyle\sim$}(A2);
\draw[->] (B0) -- node[above] {$\scriptstyle$} node[below]{$\scriptstyle\sim$}(B1);
\draw[->] (B1) -- node[above] {$\scriptstyle$} node[below]{$\scriptstyle\sim$}(B2);
\draw[right hook->] (B0) -- (A0);
\draw[right hook->] (B1) -- (A1);
\draw[right hook->] (B2) -- (A2);
\end{tikzpicture}
\end{array}\label{comm for tilt}
\end{equation}

In the following two propositions, we prove that the categories $\Tilt_i(\Per)$ admit two key properties: they are closed under tensor, and their finite length subcategories are closed under duality. 
\begin{prop}\label{T via tensor}
For all $i\in\mathbb{Z}$, and $k\geq 0$, 
\[
\Tilt_{i+kN}(\Per)= \Tilt_i(\Per)\otimes \scrO(k).
\] 
Furthermore, tensoring by $\scrO(k)$ relates the ordering on simples as follows.
\begin{enumerate}
\item\label{T via tensor 1} When $\ell=1$, and so $N=1$, the ordering on simples is swapped when $k$ is odd, and preserved when $k$ is even.
\item\label{T via tensor 2} When $\ell>1$, the ordering on simples is preserved.
\end{enumerate}
\end{prop}
\begin{proof}
We chase the heart $\mod\Lambda_{i+kN}$ both ways from the bottom right to top left in \ref{tensor cor}.  On one hand, traversing the inverses of the top and right takes
\[
\mod\Lambda_{i+kN}\xrightarrow{\upbeta^{-1}}\mod\Lambda_i\xrightarrow{\Uppsi_i^{-1}}\Tilt_i(\Per)\xrightarrow{-\otimes\scrO(k)}\Tilt_i(\Per)\otimes\scrO(k),
\]
where the middle equivalence is \eqref{comm for tilt}.  On the other hand, by \eqref{obvious comp} the composition of the left and bottom functors is $\Uppsi_{i+kN}$, and so traversing the inverses of the bottom and left functors gives, again using \eqref{comm for tilt}, $\Tilt_{i+kN}(\Per)$.  Since \ref{tensor cor} is commutative, the first stated result follows.  Both statements~\eqref{T via tensor 1} and~\eqref{T via tensor 2} follow, using \ref{beta order}.
\end{proof}

\begin{prop}[Perverse--Tilt Duality]\label{Per tilt duality}
 For all $i\in\mathbb{Z}$, the following hold.
\begin{enumerate}
\item\label{Per tilt duality 1} $\mathbb{D}(\finitelength\Tilt_{i}(\Per))=\finitelength\Tilt_{1-i}(\Per)$.
\item\label{Per tilt duality 2} $\mathbb{D}$ preserves the ordering on the simples.
\end{enumerate}
\end{prop}
\begin{proof}
Since $\mathbb{D}_\Lambda=\Uppsi\circ\mathbb{D}\circ\Uppsi^{-1}$,  by \ref{Tilt duality alg} it follows that
\[
\Uppsi\circ\mathbb{D}\circ\Uppsi^{-1}(\Tilt_{i}(\finitelength\scrA))=\Tilt_{1-i}(\finitelength\scrA).
\] 
But $\Tilt_n(\finitelength\scrA)=\Uppsi\circ\Uppsi_n^{-1}(\finitelength\scrA_n)$
for all $n$, by \ref{finite length alg}, thus
\[
\Uppsi\circ\mathbb{D}\circ\Uppsi_i^{-1}(\finitelength\scrA_i)=\Uppsi\circ\Uppsi_{1-i}^{-1}(\finitelength\scrA_{1-i}).
\]
Since the bottom line in \eqref{comm for tilt} is an abelian equivalence, $\Uppsi_n^{-1}(\finitelength\scrA_n)= \finitelength\Tilt_n(\Per)$ for all $n$.  Substituting and applying $\Uppsi^{-1}$ to each side establishes \eqref{Per tilt duality 1}.

For \eqref{Per tilt duality 2}, the fact that the ordering is preserved is \ref{Tilt duality alg}\eqref{Tilt duality alg 2}, given that the ordering on the simples in $\Tilt_{i}(\Per)$ is induced from the algebraic ordering in $\Tilt_{i}(\scrA)$.
\end{proof}

Combining with \ref{T via tensor}, the following will later allow us to compute the simples in $\Tilt_i(\Per)$ for all $i\in\mathbb{Z}$, by computing them only in a finite region. 

\begin{cor}\label{simples in Tn cor}
With notation as above, the two simples of \opt{ip}{the heart }$\Tilt_i(\Per)$ are given as follows.
\begin{enumerate}
\item If $i\geq 0$ the simples are $\Uppsi^{-1}_i(S_0)$ and $\Uppsi^{-1}_i(S_1)$.
\item If $i< 0$ the simples are $\mathbb{D}\Uppsi^{-1}_{1-i}(S_0)$ and $\mathbb{D}\Uppsi^{-1}_{1-i}(S_1)$.
\end{enumerate}
\end{cor}
\begin{proof}
The case $i\geq 0$ follows immediately from \eqref{comm for tilt}.  The case $i<0$ is then immediate, since $\finitelength\Tilt_{i}(\Per)=\mathbb{D}(\finitelength\Tilt_{1-i}(\Per))$ by Perverse-Tilt Duality \ref{Per tilt duality}.  
\end{proof}

\subsection{The Simples Helix}\label{section simples helix}
We now define the sequence $\{\scrS_i\}_{i\in\mathbb{Z}}$.  The sheaf $\scrS_0\colonequals \scrO_{\Curve}(-1)$, and the terms $\scrS_1,\hdots,\scrS_{N/2}$ are defined as being either 
\[
\left\{
\begin{array}{ll}
\scrO_{\ell\Curve},\hdots,\scrO_{2\Curve}&\mbox{if }\ell\leq 4\\
\scrO_{\ell\Curve},\hdots,\scrO_{3\Curve},\scrZ,\scrO_{2\Curve} & \mbox{if }\ell=5,6.
\end{array}
\right.
\]   
The sheaf $\scrZ$ will be constructed in the proof of \ref{simples main} below, and it will be shown in \ref{Ext23} that $\scrZ$ occurs as the unique non-split extension 
\[
0\to \scrO_{3\Curve}\to\scrZ\to \scrO_{2\Curve}\to 0.
\]
This extension group vanishes when $\ell\leq 4$ (also shown in \ref{Ext23} below), which explains why $\scrZ$ only appears for high length $E_8$ flops.

The terms $\scrS_{N/2+1},\hdots,\scrS_{N-1}$ are then defined to be either
\[
\left\{
\begin{array}{ll}
\omega_{3\Curve}(1),\hdots,\omega_{\ell\Curve}(1)&\mbox{if }\ell\leq 4\\
\scrZ^{\omega}(1),\omega_{3\Curve}(1),\hdots,\omega_{\ell\Curve}(1) & \mbox{if }\ell=5,6
\end{array}
\right.
\]   
where $\scrZ^{\omega}\colonequals \mathbb{D}(\scrZ)[-1]$, and $\mathbb{D}$ is the Grothendieck dual.  The full simples helix $\{\scrS_i\}_{i\in\mathbb{Z}}$ is then defined by translating the region $\scrS_0,\hdots,\scrS_{N-1}$ via the rule $\scrS_{i+N} = \scrS_i\otimes\scrO(1)$ for all~$i$.  Note that, by construction, the terms $\scrS_{-N/2},\hdots,\scrS_{-1}$ satisfy the rule
\begin{equation}
\scrS_{-i}=\mathbb{D}(\scrS_i)[-1].\label{Other side SKMS simples}
\end{equation}
In most cases these are just $\omega_{k\Curve}$, by \ref{D to kC}.

\begin{thm}\label{simples main}
Consider the simples helix $\{ \scrS_{i}\}_{i\in\mathbb{Z}}$ above.  
\begin{enumerate}
\item For all $i\geq 0$, the category $\Tilt_i(\Per)$ has simples $\scrS_{i-1}[1]$ and $\scrS_i$.  Furthermore, with respect to the order in \S\ref{ordering section}
\[
\Uppsi^{-1}_i S_0=\left\{
\begin{array}{ll}
\scrS_{i-1}[1]&\mbox{if $i$\! even}\\
\scrS_{i}&\mbox{if $i$\! odd}
\end{array}
\right.
\quad
\mbox{and}
\quad
\Uppsi^{-1}_i S_1=\left\{
\begin{array}{ll}
\scrS_{i}&\mbox{if $i$\! even}\\
\scrS_{i-1}[1]&\mbox{if $i$\! odd.}
\end{array}
\right.
\]
\item For all $i<0$, the category $\Tilt_i(\Per)$ has simples $\scrS_{i-1}[1]$ and $\scrS_i$.  
\end{enumerate}
\end{thm}
\begin{proof}
\step{1} When $i=0,1$, both statements in (1) are of course already known.  Indeed, when $i=0$  the projectives are
\[P_0=\Hom_R(R\oplus N,R)\quad\mbox{and}\quad P_1=\Hom_R(R\oplus N,N),\]
and the corresponding simples are $\omega_{\ell\Curve}[1]$ and $\scrO_{\Curve}(-1)$ respectively~\cite[3.5.8]{VdB1d}, which are $\scrS_{-1}[1]$ and $\scrS_0$. Similarly, when $i=1$, the projectives are
\[
P_0=\Hom_R(R\oplus M,R)\quad\mbox{and}\quad P_1=\Hom_R(R\oplus M,M),
\]
and as explained in \ref{finite length alg} it is known that $\Tilt_1(\Per)=\PerOne$.  Again, by \cite[3.5.8]{VdB1d} the corresponding simples are $\scrO_{\ell\Curve}$ and $\scrO_{\Curve}(-1)[1]$ respectively, which are $\scrS_1$ and $\scrS_0[1]$.

We next observe that in Type $A$ (i.e.\ $\ell=1$), by \ref{T via tensor}, we are already done.  All the simples appearing in $\Tilt_i(\Per)$ are line bundle tensors of $\scrS_0$ and $\scrS_1$, and the order is as claimed due to the permutation in \ref{T via tensor}.  This proves all statements  when $\ell=1$.

\step{2} We can now assume that $\ell>1$, in which case $N$ is even.  Observe next that we just need to prove the statements in (1) for $i\in\{0,\hdots,N/2\}$ because, by Perverse-Tilt Duality \ref{Per tilt duality}, if $i\in\{-N/2+1,\hdots,-1\}$ then $\Tilt_{i}(\Per)$ is dual to $\Tilt_{1-i}(\Per)$.  If $|i|$ is even, then $1-i$ is odd, so the ordered simples in $\Tilt_{i}(\Per)$ are thus $\mathbb{D}(\scrS_{1-i})$ and $\mathbb{D}(\scrS_{-i}[1])$. By \eqref{Other side SKMS simples}, these equal $\scrS_{i-1}[1]$ and $\scrS_{i}$ respectively.  A similar analysis holds if $|i|$ is odd, showing in both cases that the (unordered) simples are $\scrS_{i}$ and $\scrS_{i-1}[1]$, and furthermore the ordered version matches the statement in (2). It then follows we have identified the ordered simples in both (1) and (2) within the region
\[
\{-N/2+1,\hdots,N/2\}.
\]
By \ref{T via tensor} this is a fundamental region, and we are done.

\step{3} We now prove the statements in (1) for $i\in\{0,\hdots,N/2\}$.  If $\ell=2$ then $N=2$, so we are done by the first paragraph of this proof.  Thus we can assume that $\ell\geq 3$, in which case $N/2\geq 2$.  Since $i=0,1$ is known by the first paragraph, it suffices to prove all the statements in (1) for the region $i\in\{2,\hdots,N/2\}$, for $\ell\geq 3$.  Since all functors are equivalences, by \ref{simples in Tn cor}  it suffices to show that 
\def\eqnSA{S_0\opt{ip}{&}=\left\{
\begin{array}{ll}
\upkappa_i\Uppsi(\scrS_{i-1}[1])&\mbox{if $i$\! even}\\
\upkappa_i\Uppsi(\scrS_{i})&\mbox{if $i$\! odd}
\end{array}
\right.}
\def\eqnSB{S_1\opt{ip}{&}=\left\{
\begin{array}{ll}
\upkappa_i\Uppsi(\scrS_{i})&\mbox{if $i$\! even}\\
\upkappa_i\Uppsi(\scrS_{i-1}[1])&\mbox{if $i$\! odd}
\end{array}
\right.\label{reduction 1}}
\opt{ams}{\begin{equation}
\eqnSA
\quad
\mbox{and}
\quad
\eqnSB
\end{equation}}
\opt{ip}{\begin{align}
\begin{split}
\eqnSA \\
\eqnSB
\end{split}
\end{align}}
for all $2\leq i\leq N/2$.  

We first show the case $i=2$ in \eqref{reduction 1}, which is even, where the claim is that
\begin{equation}
S_0=\Phi_1\Phi_0\Uppsi(\scrO_{\ell\Curve}[1])\quad\mbox{ and }\quad S_1=\Phi_1\Phi_0\Uppsi(\scrO_{(\ell-1)\Curve})\label{reduct 2}
\end{equation}
in $\mod\Lambda_2$.  The case $i=1$ implies that, in $\mod\Lambda_1$, there are equalities 
\begin{equation*}
S_0=\Phi_0\Uppsi(\scrO_{\ell\Curve})\quad\mbox{ and }\quad S_1=\Phi_0\Uppsi(\scrO_{\Curve}(-1)[1]).  
\end{equation*}
Simply applying $\Phi_1$ to the first of these, using $\Phi_1(S_0)=S_0[-1]$ by \ref{mut track lem}, gives the first required equality in \eqref{reduct 2}.  Now applying $\Phi_1\Phi_0\Uppsi$ to the non-split Katz triangle 
\[
\scrO_{(\ell-1)\Curve}\to\scrO_{\Curve}(-1)[1]\to \scrO_{\ell\Curve}[1]
\]
from \ref{Katz prop}, where the last map is clearly non-zero, gives a triangle
\begin{equation}
\Phi_1\Phi_0\Uppsi(\scrO_{(\ell-1)\Curve})\to \Phi_1(S_1) \to S_0\to\label{reduct 4}
\end{equation}
in $\Db(\mod\Lambda_2)$ where the last map is non-zero.  As the last term is just $S_0=\mathbb{C}$, this non-zero map must be surjective.  But by \ref{mut track lem}, the middle term is a module of dimension vector $(n_1,1)$, so the long exact sequence in cohomology then implies that $\Phi_1\Phi_0\Uppsi(\scrO_{(\ell-1)\Curve})$ is a module.  By inspection of \ref{ns summary} we see that $n_1=1$ since $\ell\geq 3$, thus we conclude that $\Phi_1\Phi_0\Uppsi(\scrO_{(\ell-1)\Curve})$ has dimension vector $(1,0)$, and hence is the simple\opt{ams}{ }\opt{ip}{~}$S_0$.  

\step{4} We have now established the case $i=2$ in \eqref{reduction 1}, so when $N=4$ (equivalently, when $\ell=3$), we are done.   Thus we can assume that $\ell\geq 4$.  The next step is to establish the case $i=3$ in \eqref{reduction 1},  where the claim is that
\begin{equation}
S_0=\Phi_2\Phi_1\Phi_0\Uppsi(\scrO_{(\ell-2)\Curve})\quad\mbox{ and }\quad S_1=\Phi_2\Phi_1\Phi_0\Uppsi(\scrO_{(\ell-1)\Curve}[1]).\label{reduct 5}
\end{equation}
Applying $\Phi_2$ to the previous $i=2$ equality $S_1=\Phi_1\Phi_0\Uppsi(\scrO_{(\ell-1)\Curve})$, using \ref{mut track lem}, immediately gives the second claim in \eqref{reduct 5}.  Now note that \eqref{reduct 4} does not split, else the Katz triangle would split.  Hence in the rotated triangle 
\[
\Phi_1(S_1) \to S_0\to \Phi_1\Phi_0\Uppsi(\scrO_{(\ell-1)\Curve})[1],
\]
the last map is non-zero.  Applying $\Phi_2$ gives a triangle
\[
\Phi_2\Phi_1(S_1) \to \Phi_2(S_0)\to S_1,
\]
where the last map is non-zero, and thus is surjective.  By \ref{mut track lem} the middle term is a module of dimension vector $(1,n_2)$, and by inspection of \ref{ns summary} we see that $n_2=2$ since $\ell\geq 4$.  From the long exact sequence in cohomology, we deduce that $\Phi_2\Phi_1(S_1)$ is a module, of dimension vector $(1,1)$.  Now apply $\Phi_2\Phi_1\Phi_0\Uppsi$ to the non-split Katz triangle
\[
\scrO_{(\ell-2)\Curve}\to\scrO_{\Curve}(-1)[1]\to \scrO_{(\ell-1)\Curve}[1],
\]
where the last map is clearly non-zero, to obtain a triangle
\begin{equation}
\Phi_2\Phi_1\Phi_0\Uppsi(\scrO_{(\ell-2)\Curve})\to \Phi_2\Phi_1(S_1) \to S_1\label{reduct 8}
\end{equation}
where the rightmost two terms are modules, of dimension vectors $(1,1)$ and $(0,1)$ respectively.   The last map is non-zero and hence is surjective.  We deduce that $\Phi_2\Phi_1\Phi_0\Uppsi(\scrO_{(\ell-2)\Curve})$ is a module, of dimension vector $(1,0)$, and so is isomorphic to $S_0$.

\step{5} We have now established the case $i=3$ in \eqref{reduction 1}, so when $N=6$ (equivalently, when $\ell=4$), we are done.   Thus we can assume that $\ell\geq 5$.   From here, a unified proof, although possible, is notationally very heavy, so for notational ease, we now split the proof.

\emph{Case: $\ell=5$}.  Since $N=10$, we just need to verify \eqref{reduction 1} for $i=4,5$.    For the case $i=4$, we know that $\scrO_{3\Curve}$ and $\scrO_{4\Curve}[1]$ are the two previous simples, and by inspection of \ref{ns summary} we know $n_3=3$.  As always, by \ref{mut track lem} $\Phi_3\Phi_2\Phi_1\Phi_0\Uppsi(\scrO_{3\Curve}[1])=S_0$, and so $\scrO_{3\Curve}[1]$ is one of the simples.  For the other, first apply $\Phi_3$ to \eqref{reduct 8} and rotate to obtain 
\[
\Phi_3 \Phi_2\Phi_1(S_1) \to \Phi_3(S_1)\to S_0.
\]
The rightmost two terms are modules, of dimension vectors $(3,1)$ and $(1,0)$ respectively.  The last map is non-zero, hence surjective, and so $\Phi_3 \Phi_2\Phi_1(S_1)$ is a module of dimension vector $(2,1)$.  Applying $\Phi_3\hdots\Phi_0\Uppsi$ to the non-split Katz triangle
\[
\scrO_{2\Curve}\to\scrO_{\Curve}(-1)[1]\to \scrO_{3\Curve}[1],
\]
then gives a triangle
\[
\Phi_3\hdots\Phi_0\Uppsi(\scrO_{2\Curve})\to \Phi_3\Phi_2\Phi_1(S_1)\to S_0.
\]
Applying the same logic as above, we deduce that the above is a non-split short exact sequence of modules, and $\Phi_3\hdots\Phi_0\Uppsi(\scrO_{2\Curve})$ has dimension vector $(1,1)$.  In particular, it is not simple.  It must have a filtration by the two simples, and since
\begin{align*}
\Hom(S_0,\Phi_3\hdots\Phi_0\Uppsi(\scrO_{2\Curve}))
&=
\Hom(\Phi_3\hdots\Phi_0\Uppsi(\scrO_{3\Curve})[1],\Phi_3\hdots\Phi_0\Uppsi(\scrO_{2\Curve}))\\
&=\Hom(\scrO_{3\Curve}[1],\scrO_{2\Curve})\\
&=0,
\end{align*}
we first see that the filtration is necessarily of the form
\begin{equation}
0\to S_1\to\Phi_3\hdots\Phi_0\Uppsi(\scrO_{2\Curve})\to
S_0\to 0,\label{reduct 10}
\end{equation}
and second that this sequence cannot split. Applying the inverse of $\Phi_3\hdots\Phi_0 \Uppsi$ and rotating, we obtain a non-split triangle
\begin{equation}
\scrO_{3\Curve}\to\scrZ\to \scrO_{2\Curve}\label{Z sequence}
\end{equation}
where
\begin{equation}\label{equation definition Z}
\scrZ=\Uppsi^{-1}\Phi_0^{-1}\hdots\Phi_3^{-1}(S_1).
\end{equation}
Necessarily the above triangle~\eqref{Z sequence} arises from a non-split short exact sequence of sheaves.  This proves that $\scrZ$ is a sheaf, corresponding to the simple $S_1$.  

It thus remains to verify the case \eqref{reduction 1} for $i=5$.  By inspection of \ref{ns summary},  $n_4=1$.  As always, by \ref{mut track lem} $\Phi_4\hdots\Phi_0\Uppsi(\scrZ[1])=S_1$, and so $\scrZ[1]$ is one of the simples. For the other, simply apply $\Phi_4$ to the rotation of \eqref{reduct 10} to obtain a triangle
\[
\Phi_4\hdots\Phi_0\Uppsi(\scrO_{2\Curve})\to
\Phi_4(S_0)\to S_1
\]
where the last map is non-zero, and by \ref{mut track lem} the rightmost two terms are modules, of dimension vector $(1,n_4)=(1,1)$ and $(0,1)$ respectively.  The last map must thus be surjective, and so $\Phi_4\hdots\Phi_0\Uppsi(\scrO_{2\Curve})$ is a module of dimension vector $(1,0)$.  In particular, it is isomorphic to $S_0$, as required.

\emph{Case: $\ell=6$}.  Since $N=12$, we just need to verify \eqref{reduction 1} for $i=4,5,6$.    Since $n_4=2$ by \ref{ns summary}, the proof for $i=4$ is very similar to the proof of $(i=3,\ell\geq 4)$ above, using instead the Katz sequence
\[
\scrO_{(\ell-3)\Curve}\to\scrO_{\Curve}(-1)[1]\to \scrO_{(\ell-2)\Curve}[1].
\]
 This settles the case $i=4$.  For the case $i=5$, since the previous simples are $\scrO_{3\Curve}$ and $\scrO_{4\Curve}[1]$, and $n_4=3$, the proof for $i=5$ is very similar to the proof for $(i=4,\ell=5)$ above.  Similarly, since $n_5=1$, the proof for $i=6$ is very similar to the proof for $(i=5,\ell=5)$ above.
\end{proof}

The following is a consequence, which may be of independent interest.
\begin{cor}\label{Ext23}
With notation as above, suppose that $\ell\geq 3$, so that $\scrO_{2\Curve},\scrO_{3\Curve}$ exist.   Then
\[
\Ext^1_X(\scrO_{2\Curve},\scrO_{3\Curve})\neq 0\iff\ell=5,6.
\]
In this case, $\Ext^1_X(\scrO_{2\Curve},\scrO_{3\Curve})=\mathbb{C}$, and so the sheaf $\scrZ$ defined in \eqref{equation definition Z} is the unique non-split extension of $\scrO_{2\Curve}$ by $\scrO_{3\Curve}$.

\end{cor}

\def\extA{\Ext^1_X(\scrO_{2\Curve},\scrO_{3\Curve})\neq 0,}
\def\extB{\Ext^1_X(\scrO_{2\Curve},\scrO_{3\Curve})=\mathbb{C}.}

\begin{proof}
Note first that for consecutive entries $\scrS_{i-1}$ and $\scrS_i$ in the simples helix,
\[
\Ext^1_X(\scrS_{i},\scrS_{i-1})=
\Hom_{\Db(\coh X)}(\scrS_{i},\scrS_{i-1}[1])=0,
\]
since by \ref{simples main} $\scrS_{i-1}[1]$ and $\scrS_{i}$ are the two distinct simples in the heart of some t-structure on $\Db(\coh X)$, and so there can be no homomorphisms between them.

When $\ell=3,4$, by inspection the simples helix contains consecutive entries $\scrO_{3\Curve},\scrO_{2\Curve}$ (in that order).  By the above, $
\Ext^1_X(\scrO_{2\Curve},\scrO_{3\Curve})=0$.

In contrast, when $\ell=5,6$ the sheaf $\scrZ$ separates the sheaves $\scrO_{3\Curve}$ and $\scrO_{2\Curve}$ in the simples helix, and by \eqref{Z sequence} there is a non-split short exact sequence
\[
0\to \scrO_{3\Curve}\to\scrZ\to \scrO_{2\Curve}\to 0.
\]
In particular, $\Ext^1_X(\scrO_{2\Curve},\scrO_{3\Curve})\neq 0$.  To compute the precise dimension, observe that $\Ext^1_X(\scrZ,\scrO_{3\Curve})=0$ since $\scrO_{3\Curve},\scrZ$ are consecutive entries in the simples helix. Hence applying $\Hom_X(-,\scrO_{3\Curve})$ to the above gives a long exact sequence
\def\eqnLES{0\to \Hom_X(\scrO_{2\Curve},\scrO_{3\Curve})\to \Hom_X(\scrZ,\scrO_{3\Curve})\to \opt{ip}{ \phantom{texttext} \\ \phantom{texttext} \to }
\Hom_X(\scrO_{3\Curve},\scrO_{3\Curve})\to
\Ext^1_X(\scrO_{2\Curve},\scrO_{3\Curve})\to
0.}
\opt{ams}{\[
\eqnLES
\]}
\opt{ip}{\[
\begin{split}
\eqnLES
\end{split}
\]}
Note that $\Hom_X(\scrO_{3\Curve},\scrO_{3\Curve})=\mathbb{C}$ since $\scrO_{3\Curve}$ is a simple.  Since \opt{ams}{$\extA$}\opt{ip}{\[\extA\]} and a one-dimensional vector space surjects onto it, necessarily \opt{ams}{$\extB$}\opt{ip}{\[\extB\]}\end{proof}

The following is also a straightforward consequence of the properties of the helix.

\def\eqnTiltA{\Tilt_{-N/2+1}(\Per)}
\def\eqnTiltB{\Tilt_{N/2+1}(\Per)}

\begin{cor}\label{omega is scr shift D4}
If $\ell\geq 2$,  then $\omega_{2\Curve}\cong\scrO_{2\Curve}(-1)$.
\end{cor}
\begin{proof}
As $\ell\geq 2$, $N$ is even. We compare the hearts $\Tilt_{-N/2+1}(\Per)$ and $\Tilt_{N/2+1}(\Per)$. By definition of the helix $\scrS_{N/2} =\scrO_{2\Curve}$ and furthermore 
\[
\scrS_{-N/2}=\mathbb{D}(\scrS_{N/2})[-1] =\mathbb{D}(\scrO_{2\Curve})[-1] \stackrel{{\scriptstyle\ref{D to kC}}}{=} \omega_{2\Curve}.
\]
Then by \ref{simples main} we know that the hearts \opt{ams}{$\eqnTiltA$ and $\eqnTiltB$}\opt{ip}{\[\eqnTiltA \quad\mbox{ and }\quad \eqnTiltB\]} have a simple $\scrS_{-N/2}[1]$ and $\scrS_{N/2}[1]$ respectively. Noting the ordering of simples, by \ref{T via tensor} these are related by tensor by $\scrO(1)$, namely $\scrS_{-N/2}[1] \otimes\scrO(1) = \scrS_{N/2}[1],$ and the result follows.
\end{proof}

\section{Tilting Sheaves and Progenerators}

In this section we establish that all the tilted hearts $\Tilt_i(\Per)$ have progenerators, and these are described by consecutive terms of a helix of vector bundles $\{\scrV_i\}_{i\in\mathbb{Z}}$ constructed in \S\ref{The Vector Bundle Helix}.   When $\ell=1$, the helix will just be $\{\scrV_i=\scrO(i)\}_{i\in\mathbb{Z}}$, but for higher length flops it is more complicated.   We show in \ref{progen main} that the helix gives rise to a $\mathbb{Z}$-indexed set of tilting bundles on $X$, and we prove in \ref{all tilting bundles} that when $X$ is smooth, these are all reflexive tilting sheaves on $X$.

\opt{ip}{\smallskip}
\subsection{The Vector Bundle Helix}\label{The Vector Bundle Helix}

\begin{defin}
For $i\in\mathbb{Z}$, define  
$\scrV_i\colonequals
\Uppsi_i^{-1}\Hom_R(M_i,V_{i})$, and set $\scrP_i\colonequals \scrV_{i-1}\oplus\scrV_i$.
\end{defin}

It is clear that $\scrV_i\in\Db(\coh X)$, but we will prove later that in fact $\scrV_i$ is a vector bundle.  The following two lemmas are elementary.

\begin{lemma}\label{LemmaA}
With notation as above, the following statements hold.
\begin{enumerate}
\item\label{LemmaA 1}  $\scrV_{-1}=\scrN$, $\scrV_0=\scrO$ and $\scrV_1=\scrM$.
\item\label{LemmaA 2} $\scrV_{i+kN}\cong\scrV_i\otimes\scrO(k)$ for all $i\in\mathbb{Z}$ and all $k\geq 0$.
\end{enumerate}
\end{lemma}
\begin{proof}
\def\eqnV{\scrV_{-1} = \Uppsi^{-1}\Phi_{-1}\Hom_R(M_{-1},V_{-1})= \Uppsi^{-1}\Hom_R(M_{0},V_{-1})\cong\scrN}
(1) Note that \opt{ams}{$\eqnV$.}\opt{ip}{\[
\eqnV.\]}  The statements for $\scrV_0$ and $\scrV_1$ are similar.\\
(2) If $k=0$ there is nothing to prove, so suppose that $k>0$.  Consider the commutative diagram from \ref{tensor cor}, namely
\[
\begin{tikzpicture}
\node (A0) at (0,1.5) {$\Db(\coh X)$};
\node (A5) at (8,1.5) {$\Db(\coh X)$};
\node (B0) at (0,0) {$\Db(\mod\Lambda_i)$};
\node (N4) at (5,0) {$\Db(\mod\Lambda_{i+kN})$};
\node (N5) at (8,0) {$\Db(\mod\Lambda_i).$};
\draw[->] (A0) -- node[above] {$\scriptstyle -\otimes\scrO(-k)$}(A5);
\draw[->] (B0) -- node[above] {$\scriptstyle \Phi_{i+kN-1}\circ\hdots\circ\Phi_i$}(N4);
\draw[->] (N4) -- node[above] {$\scriptstyle \upbeta^{-1}$}(N5);
\draw[<-] (B0) -- node[left] {$\scriptstyle \Uppsi_i$}(A0);
\draw[<-] (N5) -- node[left] {$\scriptstyle \Uppsi_i$}(A5);
\end{tikzpicture}
\]
We track $\Hom_R(M_{i+kN},V_{i+kN})\in\Db(\mod\Lambda_{i+kN})$ both ways back to  $\Db(\coh X)$ in the top left.  On one hand, using \eqref{obvious comp}, it tracks back to $\Uppsi_{i+kN}^{-1}\Hom_R(M_{i+kN},V_{i+kN})$, which by definition is $\scrV_{i+kN}$.  On the other hand, under $\upbeta^{-1}$ it tracks to $\Hom_R(M_i,V_i)$, which then under $\Uppsi_i^{-1}$ tracks to $\scrV_i$, which then under $(-\otimes\scrO(-k))^{-1}$ tracks to $\scrV_i\otimes\scrO(k)$.  The isomorphism $\scrV_{i+kN}\cong\scrV_i\otimes\scrO(k)$ follows.
\end{proof}

\begin{lemma}\label{LemmaB}
For all $i\in\mathbb{Z}$, the following hold.
\begin{enumerate}
\item\label{LemmaB 1}  $\scrP_i\cong\Uppsi_i^{-1}\Lambda_i$.
\item\label{LemmaB 2} $\scrP_i$ progenerates $\Tilt_i(\Per)$.
\end{enumerate}
\end{lemma}
\begin{proof}
(1) The case $i=0$ is \ref{LemmaA}\eqref{LemmaA 1}. When $i> 0$, note that
\[
\scrV_{i-1}=\Uppsi^{-1}\upkappa_{i-1}^{-1}\Hom_R(M_{i-1},V_{i-1})
=\Uppsi^{-1}\upkappa_{i}^{-1}\Phi_i\Hom_R(M_{i-1},V_{i-1}),
\]
which by \eqref{S and Ps} equals $\Uppsi^{-1}\upkappa_{i}^{-1}\Hom_R(M_i,V_{i-1})$.  It thus follows that $\scrP_i=(\upkappa_i\Uppsi)^{-1}\Lambda_i$.  Lastly, when $i<0$, again using \eqref{S and Ps} we see that
\[
\scrV_{i-1}=\Uppsi^{-1}\uplambda_{i-1}\Hom_R(M_{i-1},V_{i-1})
=\Uppsi^{-1}\uplambda_{i}\Hom_R(M_{i},V_{i-1}),
\]
from which $\scrP_i=\Uppsi^{-1}\uplambda_i\Lambda_i$ follows. \\
(2) This follows from (1), since  by \eqref{comm for tilt} $\Uppsi_i\colon \Tilt_i(\Per)\to\mod\Lambda_i$
is an equivalence.
\end{proof}

\begin{remark}
We remark that, amongst other things, when $\ell=2$ (and so $N=2$) the above \ref{LemmaA} establishes the non-obvious fact that $\scrM\cong\scrN\otimes\scrO(1)$.
\end{remark}

The content in this section is to establish that for all $i\in\mathbb{Z}$, the $\scrV_i$ are vector bundles, and that there are short exact sequences of sheaves
\[
0\to\scrV_{i-1}\to\scrV_i^{\oplus n_i}\to\scrV_{i+1}\to 0.
\]
To do this requires the following preparatory lemma.

\begin{lemma}\label{only neg}
If $i\geq 0$, then with respect to the standard t-structure on \opt{ip}{the derived category }$\Db(\coh X)$, objects in $\Tilt_i(\Per)$ have cohomology only in non-positive degrees.
\end{lemma}
\begin{proof}
Since $i\geq 0$, by definition $\Tilt_i(\mod\Lambda)=\upkappa_i^{-1}(\mod\Lambda_i)$.  Since the composition $\upkappa_i$ is induced by a single tilting module of projective dimension one, it follows that with respect to the standard t-structure on $\Db(\mod\Lambda)$, the heart $\Tilt_i(\mod\Lambda)$ lives in homological degrees $-1$ and $0$.  Tracking this over the equivalence $\Uppsi^{-1}$, which takes modules to homological degrees $-1$ and $0$, it is easy to see that $\Tilt_i(\Per)=\Uppsi^{-1}(\Tilt_i(\mod\Lambda))$ lives in homological degrees $-2$, $-1$ and $0$.
\end{proof}

The following is also required, and becomes a key concept in \S\ref{def applications}.
\begin{defin}\label{def alg def}
For each $i\in\mathbb{Z}$, the $i$-th helix deformation algebra is defined to be
\[
\Lambda_i^{\deform}\colonequals \Lambda_i/[V_i]\cong \End_R(V_{i-1})/[V_i]
\]  
where $[V_i]$ is the two-sided ideal of $\Lambda_i\cong\End_R(V_{i-1}\oplus V_i)$ consisting of those morphisms which factor through $\add V_i$.
\end{defin}

The connection with deformations will be explained in \ref{all in helix have rep} later.  For now, the following suffices.

\begin{lemma}\label{def alg fd and complete}
For all $i\in\mathbb{Z}$, the algebra $\Lambda_i^{\deform}$ is finite dimensional, and complete with respect to its augmentation ideal.
\end{lemma}
\begin{proof}
Let $\p$ be a height two prime.  Then since localisation is exact
\begin{equation}
(\Lambda_i^{\deform})_\p\cong \End_{R_\p}((V_{i-1})_\p)/[(V_i)_\p].\label{loc of def alg}
\end{equation}
But $(V_j)_\p\in\refl R_\p=\CM R_\p$ for all $j\in\mathbb{Z}$, since $\dim R_\p=2$. Since $R$ has only an isolated singularity, we conclude that all $(V_j)_\p$ are free, and so it follows from \eqref{loc of def alg} that $(\Lambda_i^{\deform})_\p=0$. Thus $\Lambda_{i}^{\deform}$ is only supported on the maximal ideal as an $R$-module, and thus is finite dimensional.  The fact that it is complete follows exactly as in \cite[3.9(4)]{DW2}: $R$~is $\m$-adically complete, hence so too is $\Lambda_i$, and thus by Nakayama both $\Lambda_i$ and its factor $\Lambda_i^{\deform}$ are complete with respect to their augmentation ideals.
\end{proof}

When $\ell=1$  all the $n_i=2$, and the following are all just twists of the Euler sequence.  We thus view the following, which is the main result of this subsection, as giving generalised Euler sequences for higher length flops.

\begin{thm}\label{V pos are bundles}
For all $i\in\mathbb{Z}$, the $\scrV_i$ are vector bundles, and there are short exact sequences 
\[
0\to\scrV_{i-1}\to\scrV_i^{\oplus n_i}\to\scrV_{i+1}\to 0
\]
where the $n_i$ are from \ref{ns summary}.
\end{thm}
\begin{proof}
By \ref{LemmaA}\eqref{LemmaA 2} it clearly suffices to prove the result for $i\geq 0$, as the result for $i<0$ can be obtained by twisting by a line bundle. 

We already know from \ref{LemmaA}\eqref{LemmaA 1} that $\scrV_{-1}=\scrN$, $\scrV_0=\scrO$ and $\scrV_1=\scrM$, and these are vector bundles.  Furthermore, we have a pullback diagram with short exact sequences as follows, where the right hand and bottom rows are as in~\cite[3.5.2]{VdB1d}.
\[
\begin{tikzpicture}
\node (A) at (0,4) {$0$};
\node (B) at (0,3) {$\scrN$};
\node (C) at (0,2) {$\scrO^{\oplus \ell+1}$};
\node (D) at (0,1) {$\scrL$};
\node (E) at (0,0) {$0$};
\node (a) at (-4.5,1) {$0$};
\node (b) at (-3,1) {$\scrO^{\oplus \ell -1}$};
\node (c) at (-1.5,1) {$\scrM$};
\node (e) at (1.3,1) {$0$};
\node (F) at (-1.5,4) {$0$};
\node (G) at (-1.5,3) {$\scrN$};
\node (H) at (-1.5,2) {$\scrF$};
\node (J) at (-1.5,0) {$0$};
\node (f) at (-4.5,2) {$0$};
\node (g) at (-3,2) {$\scrO^{\oplus \ell -1}$};
\node (j) at (1.3,2) {$0$};

\draw[->] (A)--(B);
\draw[->] (B)--(C);
\draw[->] (C)--(D);
\draw[->] (D)--(E);
\draw[->] (a)--(b);
\draw[->] (b)--(c);
\draw[->] (c)--(D);
\draw[->] (D)--(e);
\draw[->] (F)--(G);
\draw[->] (G)--(H);
\draw[->] (H)--(c);
\draw[->] (c)--(J);
\draw[->] (f)--(g);
\draw[->] (g)--(H);
\draw[->] (H)--(C);
\draw[->] (C)--(j);
\draw[-,transform canvas={xshift=+\equalsSep}] (b) to (g);
\draw[-,transform canvas={xshift=-\equalsSep}] (b) to (g);
\draw[-,transform canvas={yshift=+\equalsSep}] (B) to (G);
\draw[-,transform canvas={yshift=-\equalsSep}] (B) to (G);
\end{tikzpicture}
\]
Now $\Ext^1_X(\scrO,\scrO)=0$, so $\scrF \cong \scrO^{\oplus 2 \ell}$, which gives a short exact sequence
\[
0\to\scrN\to\scrO^{\oplus 2\ell}\to\scrM\to 0.
\]
Since $n_0=2\ell$ by \ref{ns summary}, this establishes the short exact sequence in the case $i=0$.

\def\eqnHomR{\Hom_R(M_i,V_{i+1})=\Phi_{i}^{-1}\Hom_R(M_{i+1},V_{i+1}).}

Next we assume that $\scrV_{i}$ and $\scrV_{i-1}$ are bundles, from this establish that so too is $\scrV_{i+1}$, and furthermore the exact sequence between $\scrV_{i-1}$, $\scrV_{i}$ and $\scrV_{i+1}$ holds.   By induction, the result follows.  To do this, consider the exact sequences
\begin{align*}
0\to\Hom_R(M_i,V_{i-1})\to\Hom_R(M_i,V_i)^{\oplus n_i}\to\Hom_R(M_i,V_{i+1})&\to 0\\
0\to\Hom_R(M_i,V_{i+1})\to\Hom_R(M_i,V_i)^{\oplus n_i}\to\Hom_R(M_i,V_{i-1})&\to\opt{ams}{\Lambda_i^{\deform}\to 0}\opt{ip}{\\ \to \Lambda_i^{\deform}\to 0}
\end{align*}
in $\mod\Lambda_i$.  By construction of the mutation functor, the tilting module gets sent to the projectives, and so 
\opt{ams}{$\eqnHomR$}\opt{ip}{\[\eqnHomR\]} Combining with \ref{LemmaB}\eqref{LemmaB 1}, \opt{ams}{this implies that }pulling back the first via $(\upkappa_i\Uppsi)^{-1}$ gives an exact sequence
\begin{align*}
0\to\scrV_{i-1}\to&\scrV_i^{\oplus n_i}\to\scrV_{i+1}\to 0
\end{align*}
in $\Tilt_i(\Per)$.  Pulling back the lower sequence, and splicing, gives two exact sequences
\begin{align*}
0\to\scrV_{i+1}\to&\scrV_i^{\oplus n_i}\to \scrK_i\to 0\\
0\to\scrK_i\to&\scrV_{i-1}\to\scrE_i\to 0 
\end{align*}
in $\Tilt_i(\Per)$, and thus three exact sequences in total.

Since $\Lambda_i^{\deform}$ is finite dimensional by \ref{def alg fd and complete}, and is filtered by the simple $S_0$ (when $i$ is odd) or $S_1$ (when $i$ is even), using \ref{simples main} we see that $\scrE_i$ is a sheaf, filtered by the sheaf $\scrS_i$.   Using the long exact sequence in usual cohomology applied to the lower exact sequence above, which is a triangle, then since  $\scrK_i\in\Tilt_i(\Per)$ can only live in non-positive degrees by \ref{only neg}, it follows that $\scrK_i$ is a sheaf.  In turn, the middle exact sequence then implies that $\scrV_{i+1}$ is a sheaf, and so all of the three exact sequences above are in fact exact sequences of sheaves.  In particular, by the top sequence $\pd_{\scrO_{X,x}}(\scrV_{i+1})_x\leq 1$ for all $x\in X$.  

But by \ref{Katz prop} it is easy to see that all $\scrS_i$ have  depth $1$, so it follows that $\depth_{\scrO_{X,x}}(\scrE_i)_x\geq 1$ for all $x\in X$.  Applying the depth lemma to the bottom two exact sequences, we deduce that $\depth_{\scrO_{X,x}}(\scrV_{i+1})_x=3$.  Since $\pd_{\scrO_{X,x}}(\scrV_{i+1})_x<\infty$ by the above, Auslander--Buchsbaum then implies that 
\[
\pd_{\scrO_{X,x}}(\scrV_{i+1})_x=\dim \scrO_{X,x} - \depth_{\scrO_{X,x}}(\scrV_{i+1})_x = 3-3=0,
\] 
and so $\scrV_{i+1}$ is locally free.  
\end{proof}

\subsection{Functorial Properties}

\begin{thm}\label{progen main}
For all $i\in\mathbb{Z}$, the following statements hold.
\begin{enumerate}
\item\label{progen main 1} $\scrP_i$ is a progenerator of $\Tilt_i(\Per)$, and is a tilting bundle on $X$.
\item\label{progen main 2} $f_*(\scrV_i)\cong V_i$, where the $V_i$ are from \S\ref{IW9 summary section}.
\item\label{progen main 3} There is a functorial isomorphism $\Uppsi_i\cong\RHom_X(\scrP_i,-)$.
\end{enumerate}
\end{thm}
\begin{proof}
(1) The fact $\scrP_i$ is a progenerator is \ref{LemmaB}\eqref{LemmaB 2}, and the fact it is a vector bundle is \ref{V pos are bundles}.  That it is a tilting bundle is clear using $\scrP_i\cong\Uppsi^{-1}\Lambda_i$ in \ref{LemmaB}\eqref{LemmaB 1}.\\
(2) The statement is true for $i=0$. For $i>0$, we have that
\begin{equation}
\Uppsi(\scrV_i)=\upkappa_i^{-1}\Hom_R(M_i,V_i)\cong\Hom_R(M_0,V_i),\label{ind in deg 0}
\end{equation}
where the last isomorphism holds since $\upkappa_i$ is given by the tilting module $\Hom_R(M_0,M_i)$, and thus it sends summands of this tilting module to the projectives. It follows that  $\Uppsi(\scrV_i)$ is in degree zero.  But then 
\begin{align*}
f_*\scrV_i&\cong\Hom_R(f_*\scrV_0,f_*\scrV_i)\tag{since $f_*\scrV_0\cong R$}\\
&\cong\Hom_X(\scrV_0,\scrV_i)\tag{by reflexive equivalence}\\
&\cong\Hom_{\Lambda_0}(\Hom_R(M_0,V_0),\Hom_R(M_0,V_i))\tag{apply $\Uppsi$}\\
&\cong\Hom_R(V_0,V_i),\tag{by reflexive equivalence}
\end{align*}
which, since $V_0=R$, is isomorphic to $V_i$.  This establishes the claim for all $i\geq 0$.  But then for any $i<0$, by \ref{LemmaA} there exists some $k$ such that $\scrV_i\cong\scrV_{i+kN} \otimes \scrO(-k)$ with $i+kN\geq 0$.  Since $f_*$ is a reflexive equivalence, it follows that
\[
f_*(\scrV_i)\cong f_*(\scrV_{i+kN}\otimes \scrO(-k))\cong f_*\scrV_{i+kN}\cdot f_*\scrO(-k)\cong V_{i+kN}\cdot L^{-k},
\]
which by \ref{class action comb} is isomorphic to $V_i$. The claim follows.\\
\def\eqnPsi{\Uppsi(\scrP_i)=\RHom_X(\scrP_0,\scrP_i)\cong\Hom_X(\scrP_0,\scrP_i),}
(3)  Suppose first that $i>0$.  As 
\opt{ams}{$\eqnPsi$}\opt{ip}{\[\eqnPsi\]}  given $\Uppsi(\scrP_i)$ is in degree zero by \eqref{ind in deg 0}, the equivalence $\Uppsi$ implies that  the adjunction map 
\[
\scrP_0\otimes^{\bf L}_{\End_X(\scrP_0)}\Hom_X(\scrP_0,\scrP_i)\to\scrP_i
\]  
is an isomorphism, both as sheaves and as right $\End_X(\scrP_i)$-modules. Thus, setting $\scrT\colonequals\Hom_X(\scrP_0,\scrP_i)$ the top half of the following diagram commutes
\opt{ams}{

}
\[
\begin{tikzpicture}
\node (0) at (0,0) {$\Db(\coh X)$};
\node (0b) at (5,0) {$\Db(\coh X)$};
\node (1a) at (0,-1.5) {$\Db(\mod\End_X(\scrP_0))$};
\node (1b) at (5,-1.5) {$\Db(\mod\End_X(\scrP_i))$};
\node (2a) at (0,-3) {$\Db(\mod\Lambda_0)$};
\node (2b) at (5,-3) {$\Db(\mod\Lambda_i)$};
\draw[->] (0) -- node[left] {$\scriptstyle \Uppsi=\RHom_X(\scrP_0,-)$}(1a);
\draw[->] (0) --  node[above] {$\scriptstyle \Id$}(0b);
\draw[->] (0b) -- node[right]  {$\scriptstyle \RHom_X(\scrP_i,-)$} (1b);
\draw[->] (1a) -- node[above,sloped] {$\scriptstyle\RHom(\scrT,-)$}(1b);
\draw[->] (1a) -- (2a);
\draw[->] (1b) -- (2b);
\draw[->] (2a) -- node[above]{$\scriptstyle \upkappa_i$}(2b);
\end{tikzpicture}
\]
The bottom half clearly commutes, since the outermost maps are isomorphisms induced by global sections,   $f_*\scrP_0\cong M_0$ and $f_*\scrP_i\cong M_i$ by (2), and $\upkappa_i$ is given by the tilting module $\Hom_R(M_0,M_i)$, which is global sections of $\scrT$.  Part (3) follows, for all $i\geq 0$.

\def\eqnRhom{\RHom_X(\scrP_j,\scrP_k)\cong\Hom_R(M_j,M_k),}

We finally assume that $i<0$, and prove that (3) holds for $i$.  This then finishes the proof.  We first claim $\RHom_X(\scrP_j,\scrP_k)$ is in degree zero for all $k> j\geq 0$.  By the above, the left hand side, and the large rectangle in the following diagram are commutative.
\[
\begin{array}{c}
\begin{tikzpicture}[xscale=\opt{ams}{1}\opt{ip}{0.85}]
\node (0) at (0,0) {$\Db(\coh X)$};
\node (0b) at (5,0) {$\Db(\coh X)$};
\node (0c) at (10,0) {$\Db(\coh X)$};
\node (2a) at (0,-1.5) {$\Db(\mod\Lambda_0)$};
\node (2b) at (5,-1.5) {$\Db(\mod\Lambda_j)$};
\node (2c) at (10,-1.5) {$\Db(\mod\Lambda_k)$};
\draw[->] (0) -- node[gap] {$\scriptstyle \Uppsi=\RHom_X(\scrP_0,-)$}(2a);
\draw[->] (0) -- node[above] {$\scriptstyle \Id$}(0b);
\draw[->] (0b) -- node[above] {$\scriptstyle \Id$}(0c);
\draw[->] (0b) -- node[gap]  {$\scriptstyle \RHom_X(\scrP_j,-)$} (2b);
\draw[->] (0c) -- node[gap]  {$\scriptstyle \RHom_X(\scrP_k,-)$} (2c);
\draw[->] (2a) -- node[above]{$\scriptstyle \upkappa_j$}(2b);
\draw[->] (2b) -- node[above]{$\scriptstyle \Phi_{k-1}\circ\hdots\circ\Phi_{j}$}(2c);
\end{tikzpicture}
\end{array}
\]
Hence it follows that the right hand side is commutative.  Tracking $\Lambda_k$ back both ways round the right hand diagram, we see that \opt{ams}{$\eqnRhom$}\opt{ip}{\[\eqnRhom\]} and hence in particular is in degree zero.

\def\eqnLambda{\uplambda_i\circ\RHom_X(\scrP_{i},-)\opt{ip}{&} \cong \RHom_X(\scrP_{i}\otimes^{\bf L}_{\End_X(\scrP_i)}\Hom_X(\scrP_i,\scrP_{0}),-)
\opt{ip}{\\ &} \cong \RHom_X(\scrP_0,-)\opt{ip}{\\ &}=\Uppsi.}

This being the case, after tensoring by a line bundle, it follows that $\RHom_X(\scrP_{i},\scrP_0)$ is only in degree zero.  This, together with the fact that  $\scrP_i$ is a tilting bundle by (1), implies that the adjunction map
\[
\scrP_{i}\otimes^{\bf L}_{\End_X(\scrP_i)}\Hom_X(\scrP_i,\scrP_{0})\to\scrP_0
\]   
is an isomorphism.   As above, abusing notation slightly we see that
\opt{ams}{\[
\eqnLambda
\]}
\opt{ip}{\[
\begin{split}
\eqnLambda
\end{split}
\]}
Hence since $\Uppsi_{i} = \uplambda_i^{-1}\circ\Uppsi$, applying $\uplambda_i^{-1}$ to the above line gives $\Uppsi_i\cong\RHom_X(\scrP_i,-)$.
\end{proof}

\begin{remark}\label{rank scrV eq rank V}
By \ref{progen main}\eqref{progen main 2}, the rank of the vector bundle $\scrV_i$ equals the rank of the \mbox{$R$-module $V_i$}.  This is given explicitly by the table in \ref{ns summary}.
\end{remark}

For reference later, we \opt{ams}{also }record the following, which  generalises\opt{ams}{ }\opt{ip}{~}\eqref{Psi per 0}.

\begin{cor}\label{Psi per 0 general}
For all $i\in\mathbb{Z}$, there is a commutative diagram as follows.
\[
\begin{array}{c}
\begin{tikzpicture}
\node (A1) at (0,0) {$\Db(\coh X)$};
\node (A2) at (5,0) {$\Db(\mod\Lambda_i)$};
\node (B1) at (0,-1.5) {$\Tilt_i(\Per)$};
\node (B2) at (5,-1.5) {$\mod\Lambda_i$};
\draw[->] (A1) -- node[above] {$\scriptstyle \RHom_X(\scrP_i,-)$} node[below]{$\scriptstyle\sim$}(A2);
\draw[->] (B1) -- node[above] {$\scriptstyle$} node[below]{$\scriptstyle\sim$}(B2);
\draw[right hook->] (B1) -- (A1);
\draw[right hook->] (B2) -- (A2);
\end{tikzpicture}
\end{array}
\]
\end{cor}
\begin{proof}
By \eqref{comm for tilt} we already know that the diagram commutes if the top functor is replaced by $\Uppsi_i$. The result is simply then \ref{progen main}\eqref{progen main 3}.
\end{proof}

Theorem~\ref{progen main} also allows us to classify all possible tilting bundles on\opt{ams}{ }\opt{ip}{~}$X$, when $X$ is smooth.  Recall that a sheaf is called basic if there are no repetitions in its Krull--Schmidt decomposition into indecomposables.
\begin{cor}\label{all tilting bundles}
Suppose that $\scrP$ is a basic reflexive tilting sheaf on\opt{ams}{ }\opt{ip}{~}$X$, and that $X$ is smooth.  Then $\scrP\cong\scrP_i$ for some $i\in\mathbb{Z}$.  In particular, 
\begin{enumerate}
\item The set of all basic tilting bundles on $X$ equals  $\{\scrP_i=\scrV_{i-1}\oplus\scrV_i\}_{i\in\mathbb{Z}}$.
\item All reflexive tilting sheaves on $X$ are vector bundles.
\end{enumerate}
\end{cor}
\begin{proof}
Certainly $f_*\scrP$ is a basic reflexive module giving an NCCR.  By \ref{IW9 summary} all such basic reflexive modules are isomorphic to $V_{i-1}\oplus V_i$ for some $i\in\mathbb{Z}$. Hence $f_*\scrP\cong V_{i-1}\oplus V_i$, say.  Since  $f_*\scrP_i\cong V_{i-1}\oplus V_i$ by \ref{progen main}\eqref{progen main 2}, we see that $f_*\scrP\cong f_*\scrP_i$.  Then reflexive equivalence, see e.g.\ \cite[4.2.1]{VdB1d}, implies that $\scrP\cong \scrP_i$.  The final two statements follow immediately, since by \ref{V pos are bundles} all the $\scrV_i$ are vector bundles.
\end{proof}

\section{Monodromy on \texorpdfstring{$\cM_{\scrS\scrK}$}{MSK}}  

This section applies the theory developed to construct actions on the derived categories of the $3$-fold~$X$, and the sequence of algebras~$\Lambda_i$ associated to it.  Section \ref{alg actions} first describes the monodromy action on the algebraic side, namely on the categories $\Db(\mod\Lambda_i)$.  In \S\ref{twist local sect} we construct local twist functors for the simples helix following \cite{DW1}, and in \S\ref{geo actions} we construct the monodromy action on $\Db(\coh X)$ in terms of these twist functors. In particular, this allows us to prove the main results in \S\ref{section main results}.

\def\eqnS{S^2\,\backslash \{ N+2\mbox{ points}\}}

\subsection{Algebraic Actions}\label{alg actions}
Write $\cM_{\scrS\scrK}$ for the  punctured sphere \opt{ams}{$\eqnS$}\opt{ip}{\[\eqnS\]} as before. It is convenient to think of this via an orientation-preserving identification of the grey regions in the punctured rectangle below.  As notation, we write $q_-\in S^2$ for the top pole, and $q_+$ for the bottom pole, and we refer to $q_i\in S^2$ as the \emph{equatorial punctures}.

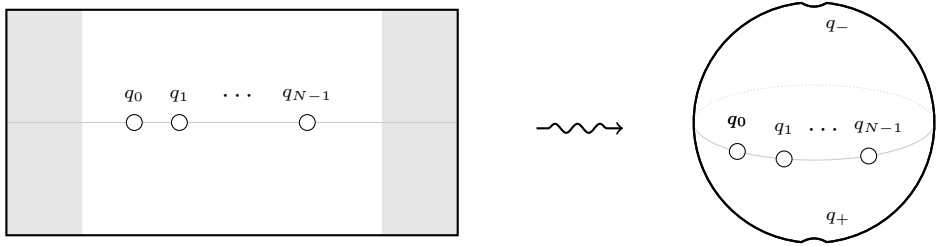
\def\halfwidth{3}
\def\halfheight{1.5}
\def\greywidth{1}
\def\pointshift{0.3}
\def\vertspherescale{0.25}
\def\spherescale{0.8 }
\def\colorlinebundletop{green!60!black}
\def\colorlinebundlebottom{red}
\def\colorflop{blue}
\def\colortwist{black}
\def\colorfront{black!20}
\def\colorback{black!10}
\def\monodromystyle{semithick}
\[
\begin{array}{ccc}
\begin{array}{c}
\begin{tikzpicture}[>=stealth,xscale=1]
\draw[fill,black!10] (0,0) -- (\greywidth,0) -- (\greywidth,2*\halfheight) -- (0,2*\halfheight) -- cycle;
\draw[fill,black!10] (2*\halfwidth-\greywidth,0) -- (2*\halfwidth,0) -- (2*\halfwidth,2*\halfheight) -- (2*\halfwidth-\greywidth,2*\halfheight) -- cycle;
\node (PL) at (-\pointshift,\halfheight) {}; 
\node (P0) at (\greywidth+\pointshift,\halfheight) {}; 
\node (P1) at (\greywidth+1.7+\pointshift,\halfheight) {}; 
\node (P2) at (\greywidth+3.4+\pointshift,\halfheight) {}; 
\node (PN1) at (2*\halfwidth-\greywidth-1.7-\pointshift,\halfheight) {}; 
\node (PN) at (2*\halfwidth-\greywidth-\pointshift,\halfheight) {}; 
\node (PR) at (2*\halfwidth+\pointshift,\halfheight) {}; 
\node (H0) at (\greywidth+0.4+\pointshift,\halfheight) {}; 
\node (H1) at (\greywidth+1+\pointshift,\halfheight) {}; 
\node (HN1) at (2*\halfwidth-\greywidth-0.7-\pointshift,\halfheight) {}; 
\draw[\colorfront] (0,\halfheight) -- (6,\halfheight);
\node (dots) at (3.1,\halfheight+0.35) {$\cdots$};
\filldraw[fill=white,draw=black] (H0) circle (3pt);
\node (H0L) at (H0) [above=0.15] {$\scriptstyle q_0$};
\filldraw[fill=white,draw=black] (H1) circle (3pt);
\node (H1L) at (H1) [above=0.15] {$\scriptstyle q_1$};
\filldraw[fill=white,draw=black] (HN1) circle (3pt);
\node (HN1L) at (HN1) [above=0.15] {$\scriptstyle q_{N-1}$};
\filldraw[white] (-0.5,0) -- (0,0) -- (0,2*\halfheight) -- (-0.5,2*\halfheight) -- cycle;
\filldraw[white] (6.5,0) -- (6,0) -- (6,2*\halfheight) -- (6.5,2*\halfheight) -- cycle;
\draw[thick] (0,0) -- (6,0) -- (6,2*\halfheight) -- (0,2*\halfheight) -- cycle;
\end{tikzpicture}
\end{array}
\opt{ams}{&
\begin{tikzpicture}
\draw [thick,->,decorate, 
decoration={snake,amplitude=.6mm,segment length=3mm,post length=1mm}] 
(3,0.4) -- (4,0.4);
\draw [thick]
(2.85,0.4) -- (3,0.4);
\end{tikzpicture}
&}
\hspace{0.3cm}
\begin{array}{c}
\begin{tikzpicture}[>=stealth,scale=2]
\draw[thick] ([shift=(-84:\spherescale)]0,0) arc (-84:84:\spherescale)  
   [bend left] to (96:\spherescale)
   arc (96:264:\spherescale)
   [bend left] to cycle;
\draw[\colorfront,line cap=round,dash pattern=on 0pt off 3\pgflinewidth] (\spherescale,0) arc (0:180:\spherescale and \vertspherescale);
\draw[\colorfront] (\spherescale,0) arc (0:-180:\spherescale and \vertspherescale)
coordinate[pos=0.72] (A) coordinate[pos=0.58] (B) coordinate[pos=0.47] (dots)   coordinate[pos=0.35] (C);
\filldraw[fill=white,draw=black] (A) circle (1.5pt);
\filldraw[fill=white,draw=black] (B) circle (1.5pt);
\filldraw[fill=white,draw=black] (C) circle (1.5pt);
\node [rotate=0] (l) at (dots) [above=0.2] {$\cdots$};
\node (Alabel) at (A) [above=0.2] {$\scriptstyle q_0$};
\node (Blabel) at (B) [above=0.2] {$\scriptstyle q_1$};
\node (Clabel) at (C) [shift=(70:0.4)] {$\scriptstyle q_{N-1}$};
\node (toplabel) at (A) [above=0.2] {$\scriptstyle q_0$};
\node at (0.2*\spherescale,0.8*\spherescale) {$\scriptstyle q_-$};
\node at (0.2*\spherescale,-0.8*\spherescale) {$\scriptstyle q_+$};
\node (p) at (-0.1*\spherescale,0.6*\spherescale) {};
\node (q) at (-0.1*\spherescale,-0.7*\spherescale) {};
\draw[thick] ([shift=(-84:\spherescale)]0,0) arc (-84:84:\spherescale)  
   [bend left] to (96:\spherescale)
   arc (96:264:\spherescale)
   [bend left] to cycle;
\end{tikzpicture}
\end{array}
\end{array}
\]

\begin{prop}\label{prop algebraic groupoid}
For all $i\in\mathbb{Z}$, the \opt{ams}{fundamental }groupoid $\uppi_1(\cM_{\scrS\scrK}, \{p_j\})$ acts on the categories $\Db(\mod\Lambda_{j})$ as follows: \opt{ams}{assign }the category $\Db(\mod\Lambda_j)$ \opt{ip}{is assigned }to each point $p_j$ shown in the diagram below, and functors to homotopy classes of paths as indicated:
\medskip

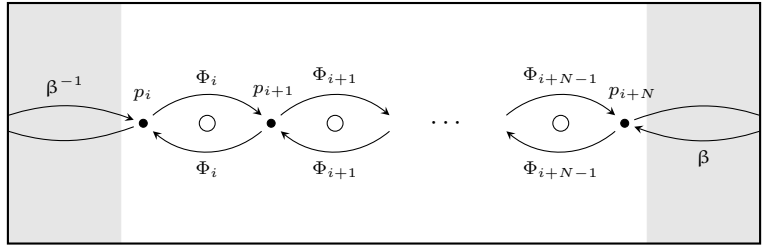
\begin{center}
\def\halfwidth{5}
\def\halfheight{1.6}
\def\greywidth{1.5}
\def\pointshift{0.3}
\def\bend{40}
\begin{tikzpicture}[>=stealth,scale=1]
\draw[fill,black!10] (0,0) -- (\greywidth,0) -- (\greywidth,2*\halfheight) -- (0,2*\halfheight) -- cycle;
\draw[fill,black!10] (2*\halfwidth-\greywidth,0) -- (2*\halfwidth,0) -- (2*\halfwidth,2*\halfheight) -- (2*\halfwidth-\greywidth,2*\halfheight) -- cycle;
\node (PL) at (-\pointshift,\halfheight) {}; 
\node (P0) at (\greywidth+\pointshift,\halfheight) {}; 
\node (P1) at (\greywidth+1.7+\pointshift,\halfheight) {}; 
\node (P2) at (\greywidth+3.4+\pointshift,\halfheight) {}; 
\node (PN1) at (2*\halfwidth-\greywidth-1.7-\pointshift,\halfheight) {}; 
\node (PN) at (2*\halfwidth-\greywidth-\pointshift,\halfheight) {}; 
\node (PR) at (2*\halfwidth+\pointshift,\halfheight) {}; 
\node (H0) at (\greywidth+0.85+\pointshift,\halfheight) {}; 
\node (H1) at (\greywidth+1.7+0.85+\pointshift,\halfheight) {}; 
\node (HN1) at (2*\halfwidth-\greywidth-0.85-\pointshift,\halfheight) {}; 
\node (dots) at (5.85,\halfheight) {$\cdots$};
\filldraw (P0) circle (1.5pt); 
\node (P0L) at (P0) [above=0.18] {$\scriptstyle p_i$};
\filldraw (P1) circle (1.5pt); 
\node (P1L) at (P1) [above=0.18,transform canvas={xshift=1pt,yshift=1pt}] {$\scriptstyle p_{i+1}$};
\filldraw (PN) circle (1.5pt); 
\node (PNL) at (PN) [above=0.18,transform canvas={xshift=3pt}] {$\scriptstyle p_{i+N}$};
\filldraw[fill=white,draw=black] (H0) circle (3pt);
\node (H0L) at (H0) [above=0.08] {$\scriptstyle $};
\filldraw[fill=white,draw=black] (H1) circle (3pt);
\node (H1L) at (H1) [above=0.08] {$\scriptstyle $};
\filldraw[fill=white,draw=black] (HN1) circle (3pt);
\node (HN1L) at (HN1) [above=0.04] {$\scriptstyle $};
\draw[->] (P0) to[bend left=\bend] node[above] {$\scriptstyle \Phi_i$}(P1);
\draw[->] (P1) to[bend left=\bend] node[above] {$\scriptstyle \Phi_{i+1}$}(P2);
\draw[->] (PN1) to[bend left=\bend] node[above] {$\scriptstyle \Phi_{i+N-1}$}(PN);
\draw[<-] (P0) to[bend right=\bend] node[below] {$\scriptstyle \Phi_i$}(P1);
\draw[<-] (P1) to[bend right=\bend] node[below] {$\scriptstyle \Phi_{i+1}$}(P2);
\draw[<-] (PN1) to[bend right=\bend] node[below] {$\scriptstyle \Phi_{i+N-1}$}(PN);
\draw[->] (PL) to[bend left=20] node[above] {$\scriptstyle \upbeta^{-1}$}(P0);
\draw (PN) to[bend left=20] (PR);
\draw[->] (PR) to[bend left=20] node[below] {$\scriptstyle \upbeta$}(PN);
\draw (P0) to[bend left=20] (PL);

\filldraw[white] (-0.5,0) -- (0,0) -- (0,2*\halfheight) -- (-0.5,2*\halfheight) -- cycle;
\filldraw[white] (10.5,0) -- (10,0) -- (10,2*\halfheight) -- (10.5,2*\halfheight) -- cycle;

\draw[thick] (0,0) -- (10,0) -- (10,2*\halfheight) -- (0,2*\halfheight) -- cycle;
\end{tikzpicture}
\end{center}
\medskip

\noindent Recall that $\Phi_j$ are the mutation functors from \S\ref{mut notation section}, and that $\upbeta$ is the isomorphism induced by line bundle twist from \S\ref{line bundle twists section}.
\end{prop}
\begin{proof}
This follows just since the fundamental groupoid is generated by the arrows shown, subject to the relation that arrows marked $\upbeta^{-1}$ and $\upbeta$ are inverse.
\end{proof}

The above action induces a similar action with basepoints a subset of~$\{p_j\}$.  In the following, we just consider the case $i=0$.  Recall from \S\ref{line bundle twists section} that $\upkappa_i$ is the composition of mutation functors.

\begin{cor}\label{prop alg monod} 
There is a group homomorphism
\begin{align*}
\uppi_1(\cM_{\scrS\scrK})&\pad{\to}\Auteq\Db(\mod \Lambda_0) \\
q_i &\pad{\mapsto}\upkappa_{i}^{-1}\circ (\Phi_i\circ\Phi_i)\circ\upkappa_i\\
q_-&\pad{\mapsto}\upbeta^{-1}\circ(\Phi_{N-1}\circ\hdots\circ\Phi_0)\\
q_+ &\pad{\mapsto}(\Phi_0\circ\hdots\circ\Phi_{N-1})\circ\upbeta
\end{align*}
where we reuse the notation $q$ for monodromy around each hole, and take loop orientations as follows.
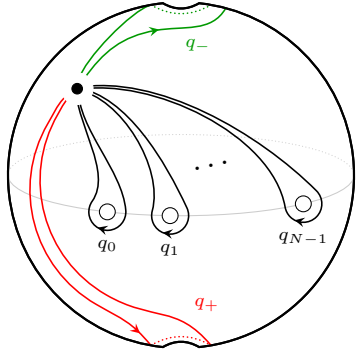
\[
\sphOnePoint{2}
\]
\end{cor}
\begin{proof}
This is immediate from \ref{prop algebraic groupoid}, simply by composing paths there.
\end{proof}

\subsection{Local Twist Functors}\label{twist local sect} 
Let $\scrE_k$  be the object corresponding to $\Lambda_k^{\deform}$ across the equivalence in \ref{Psi per 0 general}. Note that $\Lambda_k^{\deform}$ has finite projective dimension, since by \eqref{exchange 1} and \eqref{exchange 2} mutation twice returns us to our original module, so we can use the analogue of \cite[(5.B)]{DW1} to conclude that $\scrE_k\in\Perf(X)$.  Further, since $\Lambda_k^{\deform}$  is filtered by the corresponding simple $\Lambda_k$-module, $\scrE_k$ is filtered by $\scrS_k$, and thus is itself also a sheaf. 
 
\def\eqntwist{\twistGen_{\scrS_k}\colon\Db(\coh X)\to\Db(\coh X),}

\begin{thm}\label{local twist functors}
For all $k\in\mathbb{Z}$, there is an equivalence \opt{ams}{$\eqntwist$}\opt{ip}{\[\eqntwist\]} which fits into a functorial triangle as follows.
\[
\RHom_X(\scrE_{k},-)\otimes_{\Lambda_k^{\deform}}^{\bf L}\scrE_{k}\to \Id\to\twistGen_{\scrS_k}\to
\]
\end{thm}
\begin{proof}
The natural ring surjection $\Lambda_k\twoheadrightarrow\Lambda_k^{\deform}$  gives rise to a short exact sequence of \mbox{$\Lambda_k$-bimodules} $0\to I_k\to \Lambda_k\to\Lambda_k^{\deform}\to 0$. As in \cite[6.13]{DW1}, but now simply replacing $\scrP_0$ with the new tilting bundle\opt{ams}{ }\opt{ip}{~}$\scrP_k$, we define 
\begin{equation}
\twistGen_{\scrS_k}\colonequals \RHom_X(\scrP_k,-)^{-1}\circ\RHom_{\Lambda_k}(I_k,-)\circ\RHom_X(\scrP_k,-).\label{twist def}
\end{equation}
Using the analogue of \cite[6.16]{DW1}, $\twistGen_{\scrS_k}$ is then a Fourier--Mukai functor that, being a composition of equivalences, is an equivalence.  The functorial triangle follows exactly as in \cite[6.10, 6.11]{DW1} (see also \cite[end of \S5.1]{DW3}), with the bundle $\scrP_k$ replacing $\scrP_0$.
\end{proof}
Combining with  \ref{progen main}\eqref{progen main 3}, the following will be used extensively below.
\begin{cor} \label{alg twist is mono}
For all $k\in\mathbb{Z}$ there is a functorial isomorphism
\[
\Uppsi_k^{-1}\circ (\Phi_k\circ\Phi_k)\circ \Uppsi_k\cong
\twistGen_{\scrS_k}.
\]
\end{cor}
\def\eqnRhomX{\RHom_X(\scrP_k,-)\cong\Uppsi_k}
\begin{proof}
Since $\RHom_{\Lambda_k}(I_k,-)\cong\Phi_k\circ\Phi_k$ by \cite[4.3]{DW3}, and \opt{ams}{$\eqnRhomX$}\opt{ip}{\[\eqnRhomX\]} by \ref{progen main}\eqref{progen main 3}, the result follows directly from the definition \eqref{twist def}.
\end{proof}

\subsection{Geometric Monodromy}\label{geo actions}

The following is one of our main results.  In particular, it shows that the monodromies around equatorial punctures correspond to noncommutative twists around the members of our simples helix.

\begin{thm}\label{prop geom monod} 
There is a group homomorphism
\begin{align*}
\uppi_1(\cM_{\scrS\scrK})&\pad{\to}\Auteq\Db(\coh X) \\
q_i &\pad{\mapsto}\twistGen_{\scrS_i}\\
q_-&\pad{\mapsto}-\otimes\scrO(-1) \\
q_+ &\pad{\mapsto}\flop^{-1}\circ\big(-\otimes\scrO(- 1)\big)\circ\flop
\end{align*}
where $\scrS_i \in \coh X$ is as in~\S\ref{section simples helix}, the functors $\twistGen_{\scrS_i}$ are defined in~\ref{local twist functors}, and with orientations as below. 
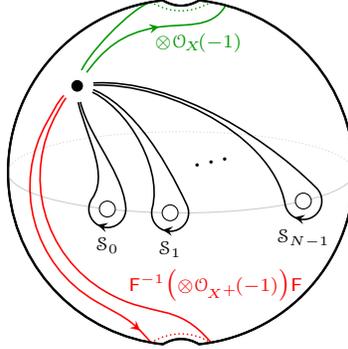
\[
\sphOnePoint{3}
\]
\end{thm}
\begin{proof}
We apply $\Uppsi^{-1} \circ (\placeholder) \circ \Uppsi$ to the representation from \ref{prop alg monod}.  The action of the loops~$q_i$ follows using \ref{alg twist is mono}, and the action of $q_-$ follows from \ref{tensor cor}. For the action of the loop $q_+$, consider the following diagram
\[
\begin{tikzpicture}[xscale=\opt{ams}{1}\opt{ip}{0.9}]
\node (Am1) at (-2.5,1.5) {$\Db(\opt{ams}{\coh} X)$};
\node (A0) at (0,1.5) {$\Db(\opt{ams}{\coh} X^+)$};
\node (A5) at (6.5,1.5) {$\Db(\opt{ams}{\coh} X^+)$};
\node (A6) at (9,1.5) {$\Db(\opt{ams}{\coh} X)$};
\node (Bm1) at (-2.5,0) {$\Db(\opt{ams}{\mod}\Lambda_{0})$};
\node (B0) at (0,0) {$\Db(\opt{ams}{\mod}\Lambda_{1})$};
\node (N4) at (4,0) {$\Db(\opt{ams}{\mod}\Lambda_{1-N})$};
\node (N5) at (6.5,0) {$\Db(\opt{ams}{\mod}\Lambda_{1})$};
\node (N6) at (9,0) {$\Db(\opt{ams}{\mod}\Lambda_{0})$};
\draw[->] (Am1) -- node[above] {$\scriptstyle \flop$}(A0);
\draw[->] (A0) -- node[above] {$\scriptstyle -\otimes\scrO_{X^+}(-1)$}(A5);
\draw[->] (A5) -- node[above] {$\scriptstyle \flop^{-1}$}(A6);
\draw[->] (Bm1) -- node[above] {$\scriptstyle \Phi_{0}^{-1}$}(B0);
\draw[->] (B0) -- node[above] {$\scriptstyle \Phi_{1-N}\circ\hdots\circ\Phi_{0}$}(N4);
\draw[->] (N4) -- node[above] {$\scriptstyle \upbeta$}(N5);
\draw[->] (N5) -- node[above] {$\scriptstyle \Phi_0$}(N6);
\draw[->] (Am1) -- node[left] {$\scriptstyle \Uppsi$}(Bm1);
\draw[->] (A0) -- node[left] {$\scriptstyle \Uppsi^+$}(B0);
\draw[->] (A5) -- node[left] {$\scriptstyle \Uppsi^+$}(N5);
\draw[->] (A6) -- node[left] {$\scriptstyle \Uppsi$}(N6);
\end{tikzpicture}
\]
where the outer squares commute by \cite[4.2]{HomMMP}, and the middle commutes by \cite[7.4]{HW} applied to $X^+$.  Then, exactly as in \ref{tensor cor}, the bottom row is isomorphic to 
\[
\Db(\mod\Lambda_0)\xrightarrow{\upbeta}\Db(\mod\Lambda_{N})\xrightarrow{\Phi_0\circ \hdots\circ\Phi_{N-1}}\Db(\mod\Lambda_0).\qedhere
\]
\end{proof}
\begin{remark} For the case $\ell=1$ there is a single equatorial puncture, and the fundamental group is given by the relation $q_0^{-1} \circ q_+ \circ q_- = \Id$. In this setting Toda showed that $\twistGen_{\scrS_0}^{-1}=\flop\circ\flop$~ \cite[3.1]{TodaFlop}, so the result of~\ref{prop geom monod} is equivalent to the functorial isomorphism
\[
\flop\circ\big(-\otimes \scrO_{X^+}(-1)\big)\circ\flop \circ \big(-\otimes \scrO_X(-1)\big)=\Id.
\]
This was verified in the case of the Atiyah flop by the first author in~\cite[7.12]{Donovan}. It follows for other length one flops from work of Toda: see~\cite[end of \S5.2, Example]{TodaRes} where $\cM_{\scrS\scrK}$ is realised as a quotient of  the normalized Bridgeland stability manifold for $X$.
\end{remark}

\begin{remark}\label{Kawamata Q}
In all Dynkin types, the simples helix always contains the sheaves
\[
\scrO_{\Curve},\scrO_{2\Curve},\hdots,\scrO_{\ell\Curve},
\]
albeit not in that order.  Consequently, by \ref{local twist functors} all of the above give rise to a twist autoequivalence (see also \ref{simples global auto} later).  The more surprising fact is that the relation between them in \ref{prop geom monod}  requires other sheaves, including $\scrZ$ and their Grothendieck duals.  This answers a question of Kawamata \cite[6.8]{Kawamata}.
\end{remark}

\subsection{Symmetric Geometric Monodromy}
It is possible to obtain a more symmetric version of the above.  Write $p_\pm\in\cM_{\scrS\scrK}$ for a choice of basepoints near the poles $q_\pm\in S^2$.  We think of these points $p_\pm$ as large radius limits in $\cM_{\scrS\scrK}$.

\begin{thm}\label{monodromy main}
The fundamental groupoid $\uppi_1(\cM_{\scrS\scrK}, \{p_\pm\})$ acts on the pair of categories $\Db(\coh X)$ and $\Db(\coh X^+)$ as follows: assign categories to points $p_\pm$, and functors to paths as below, where the left hand side is $\ell=1$, and the right hand side $\ell>1$. 
\[
\sphLengthOne{0}
\qquad
\sphTwoPoint{2}
\]
\end{thm}

\begin{proof}
By \cite[4.2]{HomMMP}, the Bridgeland--Chen flop functors are inverse to the mutation functors $\Phi_0$ under the tilting equivalences $\Uppsi$ and $\Uppsi_+=\RHom_{X^+}(\scrO_{X^+}\oplus\scrM_{X^+},-)$, so we can functorially replace \eqref{functors strip 1} with the following strip.
\[
\begin{array}{c}
\begin{tikzpicture}[xscale=\opt{ams}{1.3}\opt{ip}{1.05}]
\node (A0) at (-0.5,0) {$\Db(\opt{ams}{\mod}\Lambda_{-1})$};
\node (A1) at (2,0) {$\Db(\opt{ams}{\coh} X)$};
\node (A2) at (4,0) {$\Db(\opt{ams}{\coh} X^+)$};
\node (A3) at (6.3,0) {$\Db(\opt{ams}{\mod}\Lambda_{2})$};
\draw[->] (-1.9,0.05) -- node[above] {$\scriptstyle \Phi_{-2}$}(-1.4,0.05);
\draw[<-] (-1.9,-0.05) -- node[below] {$\scriptstyle \Phi_{-2}$} (-1.4,-0.05);
\draw[->] (0.35,0.05) -- node[above] {$\scriptstyle \Uppsi^{-1}\Phi_{-1}$}(1.2,0.05);
\draw[<-] (0.35,-0.05) -- node[below] {$\scriptstyle \Phi_{-1}\Uppsi$} (1.2,-0.05);
\draw[->] (2.75,0.05) -- node[above] {$\scriptstyle \flop^{-1}$}(3.2,0.05);
\draw[<-] (2.75,-0.05) -- node[below] {$\scriptstyle \flop^{-1}$} (3.2,-0.05);
\draw[->] (4.75,0.05) -- node[above] {$\scriptstyle \Phi_1\Uppsi_+$}(5.5,0.05);
\draw[<-] (4.75,-0.05) -- node[below] {$\scriptstyle \Uppsi_+^{-1}\Phi_1$} (5.5,-0.05);
\draw[->] (7.05,0.05) -- node[above] {$\scriptstyle \Phi_2$}(7.55,0.05);
\draw[<-] (7.05,-0.05) -- node[below] {$\scriptstyle \Phi_2$} (7.55,-0.05);
\end{tikzpicture}
\end{array}
\]
Bring \opt{ams}{the category }$\Db(\coh X)$ to near the top pole,  \opt{ams}{the category }$\Db(\coh X_+)$ to near the bottom, and  notate monodromy and orientation as in the following diagram.
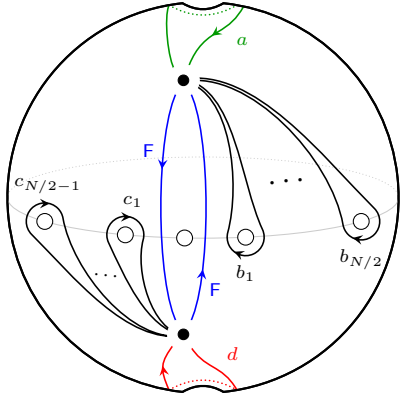
\[
\sphTwoPoint{3}
\]
Then, simply composing arrows and their inverses in the above strip to travel above and below holes as appropriate, we tautologically obtain a representation via
\[
\begin{array}{ccc@{\hspace{2pt}}c@{\hspace{2pt}}c@{\hspace{2pt}}c@{\hspace{2pt}}c@{\hspace{2pt}}c@{\hspace{2pt}}c@{\hspace{2pt}}c@{\hspace{2pt}}c@{\hspace{2pt}}c@{\hspace{2pt}}c@{\hspace{2pt}}c@{\hspace{2pt}}c@{\hspace{2pt}}c@{\hspace{2pt}}c@{\hspace{2pt}}c@{\hspace{2pt}}c@{\hspace{2pt}}c@{\hspace{2pt}}c}
b_i &\mapsto&\flop
&
\circ
&
(\Phi_1\Uppsi_+)^{-1}
&
\circ
&
\hdots
&
\circ
&
\Phi_{i-1}^{-1}
&
\circ
&
\opt{ams}{\Phi_{i}
&
\circ
&
\Phi_{i}}
\opt{ip}{\Phi_{i}^2}
&
\circ
& 
\Phi_{i-1}
&
\circ
&
\hdots
&
\circ
&
(\Phi_1\Uppsi_+)
&
\circ
&
\flop^{-1}\\[3pt]
c_i &\mapsto&
\flop
&
\circ
&
(\Phi_{-1}\Uppsi)^{-1}
&
\circ
&
\hdots
&
\circ
&
\Phi_{-i+1}^{-1}
&
\circ
&
\opt{ams}{\Phi_{-i}
&
\circ
&
\Phi_{-i}}
\opt{ip}{\Phi_{-i}^2}
&
\circ
&
\Phi_{-i+1}
&
\circ
&
\hdots
&
\circ
&
(\Phi_{-1}\Uppsi)
&
\circ
&
\flop^{-1}
\end{array}
\]
and
\[
\begin{array}{ccc@{\hspace{2pt}}c@{\hspace{2pt}}c@{\hspace{2pt}}c@{\hspace{2pt}}c@{\hspace{2pt}}c@{\hspace{2pt}}c@{\hspace{2pt}}c@{\hspace{2pt}}c@{\hspace{2pt}}c@{\hspace{2pt}}c@{\hspace{2pt}}c@{\hspace{2pt}}c@{\hspace{2pt}}c@{\hspace{2pt}}c@{\hspace{2pt}}c}
a&\mapsto&(\Uppsi^{-1}\Phi_{-1})
&
\circ
&
\hdots
&
\circ
&
\Phi_{-N/2}
&
\circ
&
\upbeta^{-1}
&
\circ
&
\Phi_{N/2-1}
&
\circ
&
\hdots
&
\circ
&
(\Phi_1\Uppsi_+)
&
\circ
&
\flop^{-1} \\[3pt]
d&\mapsto&(\Uppsi_+^{-1}\Phi_1)
&
\circ
&
\hdots
&
\circ
&
\Phi_{N/2-1}
&
\circ
&
\upbeta
&
\circ
&
\Phi_{-N/2}
&
\circ
&
\hdots
&
\circ
&
(\Phi_{-1}\Uppsi)
&
\circ
&
\flop^{-1}.
\end{array}
\]
By \cite[4.2]{HomMMP} we have $\Uppsi_+\flop^{-1}\cong\Phi_0\Uppsi$, so using \ref{alg twist is mono} we see that the image of $b_i$ is functorially isomorphic to $\twistGen_{\scrS_i}$.  Similarly, observe  that on $X^+$  the perverse zero tilted algebra is $\Lambda_0'\cong\Lambda_1$, and that the hyperplane arrangement swaps direction for $X^+$.  Hence, by the symmetry of the situation, using \cite[4.2]{HomMMP} and the $X^+$ version of \ref{alg twist is mono}, it follows that the image of $c_i$ is isomorphic to $\twistGen_{\scrS_i'}$, where $\{\scrS_i'\}_{i\in\mathbb{Z}}$ is the simples helix on $X^+$.  That the image of $a$ is $-\otimes\scrO_X(1)$ and the image of $d$ is $-\otimes\scrO_{X^+}(1)$ follows exactly as in \ref{prop geom monod}.
\end{proof}

\section{Applications to Deformation Theory and Curve Invariants}

In this section, we show that the noncommutative deformation functor associated to every member of the simples helix is representable.  We then control the representing object, describe precisely when it is not commutative, and give lower bounds on its dimension.  As a consequence, we construct global autoequivalences and produce tight lower bounds on Gopakumar--Vafa (GV) invariants for higher length flops.

\subsection{Deformation Theory}\label{def applications}
For background on noncommutative deformation theory, we refer the reader to \cite{DW1} or \cite{DW2}.  For our purposes here, for any $E\in\coh X$ there is a deformation functor
\[
\Def_E^X\colon\art_1\to\Sets
\]
from the category $\art_1$ of augmented $\mathbb{C}$-algebras to the category $\Sets$ of sets, controlled by the DGA $\sEnd_X(\scrI)$, where $0\to E\to\scrI$ is some injective resolution.  For any $M\in\mod \Lambda$ there is a similar functor $\Def_M^\Lambda$, again controlled by the endomorphism  DGA of $M$.

Recall from \ref{def alg def} that $\Lambda_i^{\deform}\colonequals\Lambda_i/[\scrV_i]$.  Since $\Lambda_i^{\deform}$ is clearly a factor of $\Lambda_i$, the techniques in \cite{DW1,DW2,DW3} still hold.

\begin{remark}\label{calibration remark}
As calibration, $\Lambda_0^{\deform}=\End_R(N)/[R]$. This is the contraction algebra $\CA$ in \cite[2.12, 3.1]{DW1}, which represents noncommutative deformations of $\scrO_{\Curve}(-1)$,  the zeroth member $\scrS_0$ of the simples helix. Further, $\Lambda_1^{\deform}=\End_R(R)/[M]$. This is the fibre algebra $\mathrm{B}_{\mathrm{fib}}$ in \cite[5.1, (5.B)]{DW2}, which represents noncommutative deformations of $\scrO_{\ell\Curve}$, the first member $\scrS_1$ of the simples helix.
\end{remark}

The following extends \ref{calibration remark} to all $i\in\mathbb{Z}$, and is the main result of this subsection.

\begin{thm}\label{all in helix have rep}
For all $i\in\mathbb{Z}$, write $S$ for the simple $\Lambda_i$-module corresponding to the projective $\Hom_X(\scrP_i,\scrV_{i-1})$. Then
\begin{enumerate}
\item There is a functorial isomorphism $\Def^X_{\scrS_i}\cong\Def^{\Lambda_i}_{S}$.
\item $\Lambda_i^{\deform}$ represents the noncommutative deformation functor of \opt{ip}{the objects }$\scrS_i$ in\opt{ams}{ }\opt{ip}{~}$X$. 
\end{enumerate}
\end{thm}
\begin{proof}
This follows immediately from our previous papers \cite{DW1,DW2}, with the key point being that \ref{Psi per 0 general} just directly replaces \eqref{Psi per 0} in all proofs. Indeed, the simple $S$ corresponding to the projective summand $\Hom_X(\scrP_i,\scrV_{i-1})$
always corresponds to a sheaf (namely $\scrS_i$) across the equivalence \ref{Psi per 0 general}.

Using the fact that $S$ corresponds to a sheaf, just repeating word-for-word \cite[3.9, 5.3]{DW2}, or \cite[\S3]{DW1}, there is a functorial isomorphism
\[
\Def^{X}_{\scrS_i}\cong\Def^{\Lambda_i}_{S}.
\] 
That the right hand side is prorepresented by $\Lambda_i^{\deform}$ is just \cite[3.1]{DW1}.  That $\Lambda_i^{\deform}\in\art_1$, and so the functor is representable, is \ref{def alg fd and complete}.\end{proof}

When $X$ is smooth, GV invariants $\GV{1},\hdots,\GV{\ell}$ of the flopping curve were defined in \cite{BKL,Katz}.  The following shows how to obtain \emph{all} of these from the NCCRs of $R$, via mutation and taking factors.  We remark that the following is mildly awkward to state, due in part to the ordering in the simples helix, but also due to the existence of $\scrZ$ in high length.
\begin{cor}\label{comm defs rep}
The commutative deformations of $\scrS_i$ are represented by the abelianisation of $\Lambda_i^{\deform}$, written  $(\Lambda_i^{\deform})_{\ab}$. In particular, if $X$ is smooth, consider the GV invariants $\GV{1},\hdots,\GV{\ell}$.  Then 
\[
\GV{i}=\dim_{\mathbb{C}}(\Lambda_{j}^{\deform})_{\ab}
\]
for appropriate $j$, depending on  both $i$ and the length $\ell$.
\end{cor}
\begin{proof}
By \ref{all in helix have rep} noncommutative deformations of $\scrS_i$ are represented by $\Lambda_{i}^{\deform}$ and so, as is standard, its abelianisation represents the commutative deformations.

It is well-known \cite{BKL, Katz} that $\GV{i}$ is the multiplicity of $\scrO_{i\Curve}$ in the Hilbert scheme, which is precisely the dimension of the representing object of the commutative deformations of the sheaf $\scrO_{i\Curve}$.  Then simply observe that the simples helix begins either $\scrO_{\Curve}(-1),\scrO_{\ell\Curve},\hdots,\scrO_{2\Curve}$ when $\ell\leq 4$, or $\scrO_{\Curve}(-1),\scrO_{\ell\Curve},\hdots,\scrO_{3\Curve},\scrZ,\scrO_{2\Curve}$ when $\ell=5,6$.  In all cases, the sheaves $\scrO_{i\Curve}$ appear as some $\scrS_j$.
\end{proof}

\subsection{Commutativity of Deformation Algebras}
We can go further, and determine when $\Lambda_i^{\deform}$ is commutative, and thus establish when commutative deformations suffice.  From this, we also determine the only sheaves in the simples helix that can possibly be genuinely spherical, and furthermore later give lower bounds on GV invariants. 

The techniques in this subsection are applicable when $X$ is Gorenstein terminal, however, as our main goal is controlling GV invariants, we restrict here to smooth $X$.  To determine $\Lambda_i$ and thus $\Lambda_i^{\deform}$ requires first an understanding of its quiver.  By the Perverse--Tilt duality, we just need to know the quivers  for $i=1,\hdots,N/2$.  To do this, we first slice by a generic element $g$. 

\begin{thm}\label{quivers and dims}
Suppose that $X\to\Spec R$ where $X$ is smooth. Writing $\mathbb{F}=-\otimes_RR/g$, then the dimensions of $\mathbb{F}\Lambda_i^{\deform}$ and $(\mathbb{F}\Lambda_i^{\deform})_{\ab}$, and whether $\Lambda_i$ is commutative, is summarised in the following table\opt{ip}{s}
\def\table{$1$ & $0$ & $1$ &$1$ \opt{ams}{&\cmark}\\
$2$ & $2,0$  & $4$,$1$& $3$,$1$\opt{ams}{& \xmark,\,\cmark}\\
$3$ & $2$,$0$,$1$ & $12$,$1$,$3$& $5$,$1$,$3$\opt{ams}{&  \xmark,\,\cmark,\,\cmark}\\
$4$ & $2$,$0$,$1$,$2$ & $24$,$1$,$2$,$6$&$6$,$1$,$2$,$4$\opt{ams}{&\xmark,\,\cmark,\,\cmark,\,\xmark}\\
$5$ & $2$,$0$,$1$,$1$,$0$,$2$ & $40$,$1$,$2$,$4$,$1$,$10$ &$7$,$1$,$2$,$4$,$1$,$6$\opt{ams}{&\xmark,\,\cmark,\,\cmark,\,\cmark,\,\cmark,\,\xmark}\\
$6$ & $2$,$0$,$1$,$1$,$2$,$1$,$2$ & $60$,$1$,$2$,$3$,$6$,$2$,$15$ &$6$,$1$,$2$,$3$,$4$,$2$,$6$\opt{ams}{&\xmark,\,\cmark,\,\cmark,\,\cmark,\,\xmark,\,\cmark,\,\xmark}\\}
\def\tableB{$1$ & $0$ &\cmark\\
$2$ & $2,0$  & \xmark,\,\cmark\\
$3$ & $2$,$0$,$1$ &  \xmark,\,\cmark,\,\cmark\\
$4$ & $2$,$0$,$1$,$2$ &\xmark,\,\cmark,\,\cmark,\,\xmark\\
$5$ & $2$,$0$,$1$,$1$,$0$,$2$ &\xmark,\,\cmark,\,\cmark,\,\cmark,\,\cmark,\,\xmark\\
$6$ & $2$,$0$,$1$,$1$,$2$,$1$,$2$ &\xmark,\,\cmark,\,\cmark,\,\cmark,\,\xmark,\,\cmark,\,\xmark\\}
\opt{ams}{\[
\begin{tabular}{C{0.6cm}llll}
\toprule
$\ell$&
\textnormal{\# loops in $\mathbb{F}\Lambda_i^{\deform}$}&
\textnormal{$\dim \mathbb{F}\Lambda_i^{\deform}$}&\textnormal{$\dim (\mathbb{F}\Lambda_i^{\deform})_{\ab}$}&\textnormal{Is $\Lambda_i^{\deform}$ commutative?}\\
\midrule
\table
\bottomrule
\end{tabular}
\]}
\opt{ip}{\[\begin{tabular}{C{0.6cm}lll}
\toprule
$\ell$&
\textnormal{\# loops in $\mathbb{F}\Lambda_i^{\deform}$}&
\textnormal{$\dim \mathbb{F}\Lambda_i^{\deform}$}&\textnormal{$\dim (\mathbb{F}\Lambda_i^{\deform})_{\ab}$}\opt{ams}{&\textnormal{Is $\Lambda_i^{\deform}$ commutative?}}\\
\midrule
\table
\bottomrule
\end{tabular}\]
\[\begin{tabular}{C{0.6cm}ll}
\toprule
$\ell$&
\textnormal{\# loops in $\mathbb{F}\Lambda_i^{\deform}$}
&\textnormal{Is $\Lambda_i^{\deform}$ commutative?}\\
\midrule
\tableB
\bottomrule
\end{tabular}\]
\smallskip
}
where in each row the sequence is over the range $i=0,\hdots,N/2.$ 
\end{thm}
\begin{proof}
We prove the $\ell=3$ row\opt{ip}{ of the first table}, with all other rows being similar. To ease notation, set $\mathsf{V}_i\colonequals \mathbb{F}V_i$ and $S\colonequals \mathsf{V}_0$.   It\opt{ams}{ }\opt{ip}{~}is well-known that $\mathbb{F}\Lambda_i\cong\End_S(\mathsf{V}_{i-1}\oplus \mathsf{V}_i)$. By Katz--Morrison \cite{KM}, the CM modules $\mathsf{V}_{-1}$ and $\mathsf{V}_1$ correspond to the middle vertex of the extended $E_6$ Dynkin diagram via McKay correspondence, and by the combinatorics in the proof of \ref{ns summary}, $\mathsf{V}_2$ corresponds to the rank 2 vertex between the middle and extending vertices.

The quivers of the $\mathbb{F}\Lambda_i$, and the dimension and a presentation of the $\mathbb{F}\Lambda_i^{\deform}$ can be obtained via knitting, which we very briefly outline here, referring the reader to \cite[\S5.4]{HomMMP} and \cite[\S4]{GL2} for more details.   Indeed, the number of arrows between the vertices in the quiver of $\mathbb{F}\Lambda_i$ is identical to \cite[\S4]{GL2}, and we see that the quivers for $\mathbb{F}\Lambda_i\cong\End_S(\mathsf{V}_{i-1}\oplus\mathsf{V}_i)$ for $i=0,1,2$  are
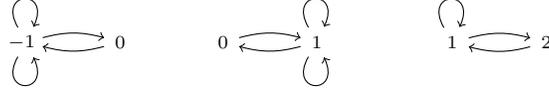
\[
\begin{array}{c}
\begin{tikzpicture}[scale=1.25,bend angle=15, looseness=1]
\node at (-0.05,0) {$\scriptstyle -1$};
\node (a) at (0,0) [pvertex] {$\phantom 0$};
\node (b) at (1,0) [pvertex] {$0$};
\draw[->,bend left] (b) to (a);
\draw[->,bend left] (a) to (b);
\draw[<-]  (a) edge [in=-120,out=-65,loop,looseness=7]  (a);
\draw[<-]  (a) edge [in=120,out=65,loop,looseness=7]  (a);
\end{tikzpicture}
\end{array}
\qquad
\begin{array}{c}
\begin{tikzpicture}[scale=1.25,bend angle=15, looseness=1]
\node (a) at (0,0) [pvertex] {$1$};
\node (b) at (-1,0) [pvertex] {$0$};
\draw[->,bend left] (b) to (a);
\draw[->,bend left] (a) to (b);
\draw[<-]  (a) edge [in=-120,out=-65,loop,looseness=7]  (a);
\draw[<-]  (a) edge [in=120,out=65,loop,looseness=7]  (a);
\end{tikzpicture}
\end{array}
\qquad
\begin{array}{c}
\begin{tikzpicture}[scale=1.25,bend angle=15, looseness=1]
\node (a) at (1,0) [pvertex] {$2$};
\node (b) at (0,0) [pvertex] {$1$};
\draw[->,bend left] (b) to (a);
\draw[->,bend left] (a) to (b);
\draw[draw=none]  (b) to [in=-120,out=-65,loop,looseness=7]  (b);
\draw[<-]  (b) to [in=120,out=65,loop,looseness=7]  (b);
\end{tikzpicture}
\end{array}
\]
where a vertex labelled $i$ corresponds to the module $\mathsf{V}_i$.   As for the dimension of say $\mathbb{F}\Lambda_0^{\deform}$, exactly as in \cite[3.16]{DW1} we can calculate it via knitting as follows:
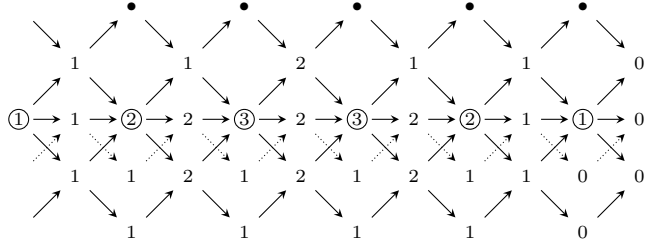
\[
\begin{tikzpicture}[scale=0.75,>=stealth]
\node (s1) at (0,0) {$\scriptstyle {\phantom 0}$};
\node (s2) at (2,0) {$\scriptstyle \bullet$};
\node (s3) at (4,0) {$\scriptstyle \bullet$};
\node (s4) at (6,0) {$\scriptstyle \bullet$};
\node (s5) at (8,0) {$\scriptstyle \bullet$};
\node (s6) at (10,0) {$\scriptstyle \bullet$};
\node (t1) at (1,-1) {$\scriptstyle 1$};
\node (t2) at (3,-1) {$\scriptstyle 1$};
\node (t3) at (5,-1) {$\scriptstyle 2$};
\node (t4) at (7,-1) {$\scriptstyle 1$};
\node (t5) at (9,-1) {$\scriptstyle 1$};
\node (t6) at (11,-1) {$\scriptstyle 0$};
\node (a) at (0,-2) {$\scriptstyle 1$};
\node (b) at (1,-2) {$\scriptstyle 1$};
\node (c) at (2,-2) {$\scriptstyle 2$};
\node (d) at (3,-2) {$\scriptstyle 2$};
\node (e) at (4,-2) {$\scriptstyle 3$};
\node (f) at (5,-2) {$\scriptstyle 2$};
\node (g) at (6,-2) {$\scriptstyle 3$};
\node (h) at (7,-2) {$\scriptstyle 2$};
\node (i) at (8,-2) {$\scriptstyle 2$};
\node (j) at (9,-2) {$\scriptstyle 1$};
\node (k) at (10,-2) {$\scriptstyle 1$};
\node (l) at (11,-2) {$\scriptstyle 0$};
\node (A) at (0,-3) {$\scriptstyle {\phantom 0}$};
\node (B) at (1,-3) {$\scriptstyle 1$};
\node (C) at (2,-3) {$\scriptstyle 1$};
\node (D) at (3,-3) {$\scriptstyle 2$};
\node (E) at (4,-3) {$\scriptstyle 1$};
\node (F) at (5,-3) {$\scriptstyle 2$};
\node (G) at (6,-3) {$\scriptstyle 1$};
\node (H) at (7,-3) {$\scriptstyle 2$};
\node (I) at (8,-3) {$\scriptstyle 1$};
\node (J) at (9,-3) {$\scriptstyle 1$};
\node (K) at (10,-3) {$\scriptstyle 0$};
\node (L) at (11,-3) {$\scriptstyle 0$};
\node (x1) at (0,-4) {$\scriptstyle {\phantom 0}$};
\node (x2) at (2,-4) {$\scriptstyle 1$};
\node (x3) at (4,-4) {$\scriptstyle 1$};
\node (x4) at (6,-4) {$\scriptstyle 1$};
\node (x5) at (8,-4) {$\scriptstyle 1$};
\node (x6) at (10,-4) {$\scriptstyle 0$};
\draw[->] (s1) -- (t1);
\draw[->] (s2) -- (t2);
\draw[->] (s3) -- (t3);
\draw[->] (s4) -- (t4);
\draw[->] (s5) -- (t5);
\draw[->] (s6) -- (t6);

\draw[->] (t1) -- (s2);
\draw[->] (t2) -- (s3);
\draw[->] (t3) -- (s4);
\draw[->] (t4) -- (s5);
\draw[->] (t5) -- (s6);

\draw[->] (a) -- (t1);
\draw[->] (c) -- (t2);
\draw[->] (e) -- (t3);
\draw[->] (g) -- (t4);
\draw[->] (i) -- (t5);
\draw[->] (k) -- (t6);

\draw[->] (t1) -- (c);
\draw[->] (t2) -- (e);
\draw[->] (t3) -- (g);
\draw[->] (t4) -- (i);
\draw[->] (t5) -- (k);

\draw[->] (a) -- (b);
\draw[->] (b) -- (c);
\draw[->] (c) -- (d);
\draw[->] (d) -- (e);
\draw[->] (e) -- (f);
\draw[->] (f) -- (g);
\draw[->] (g) -- (h);
\draw[->] (h) -- (i);
\draw[->] (i) -- (j);
\draw[->] (j) -- (k);
\draw[->] (k) -- (l);

\draw[->] (a) -- (B);
\draw[->] (c) -- (D);
\draw[->] (e) -- (F);
\draw[->] (g) -- (H);
\draw[->] (i) -- (J);
\draw[->] (k) -- (L);

\draw[->] (B) -- (c);
\draw[->] (D) -- (e);
\draw[->] (F) -- (g);
\draw[->] (H) -- (i);
\draw[->] (J) -- (k);

\draw[->] (x1) -- (B);
\draw[->] (x2) -- (D);
\draw[->] (x3) -- (F);
\draw[->] (x4) -- (H);
\draw[->] (x5) -- (J);
\draw[->] (x6) -- (L);

\draw[->] (B) -- (x2);
\draw[->] (D) -- (x3);
\draw[->] (F) -- (x4);
\draw[->] (H) -- (x5);
\draw[->] (J) -- (x6);

\draw[->,densely dotted] (A) -- (b);
\draw[->,densely dotted] (C) -- (d);
\draw[->,densely dotted] (E) -- (f);
\draw[->,densely dotted] (G) -- (h);
\draw[->,densely dotted] (I) -- (j);
\draw[->,densely dotted] (K) -- (l);

\draw[->,densely dotted] (b) -- (C);
\draw[->,densely dotted] (d) -- (E);
\draw[->,densely dotted] (f) -- (G);
\draw[->,densely dotted] (h) -- (I);
\draw[->,densely dotted] (j) -- (K);

\draw (a) circle (5pt);
\draw (c) circle (5pt);
\draw (e) circle (5pt);
\draw (g) circle (5pt);
\draw (i) circle (5pt);
\draw (k) circle (5pt);

\end{tikzpicture}
\]
We deduce that $\dim_\mathbb{C}\mathbb{F}\Lambda_0^{\deform}=1+2+3+3+2+1=12$. A very similar calculation gives $\mathbb{F}\Lambda_1^{\deform}=1$, and $\mathbb{F}\Lambda_2^{\deform}=3$.

As is standard, from the dimension calculation above, it is possible to pick algebra generators of $\mathbb{F}\Lambda_i^{\deform}$.  Then, directly verifying that certain relations hold, and comparing dimensions on both sides, we can furthermore deduce by direct calculation that $\mathbb{F}\Lambda_0^{\deform}\cong\mathbb{C}\langle x,y\rangle/(x^2,y^3,(x+y)^3)$, $\mathbb{F}\Lambda_1^{\deform}\cong\mathbb{C}$, and $\mathbb{F}\Lambda_2^{\deform}\cong\mathbb{C}[x]/x^3$.  Taking the abelianisation of these, we see that 
\def\eqnDimA{\dim_\mathbb{C}(\mathbb{F}\Lambda_0^{\deform})_{\ab}\opt{ip}{&}=5\opt{ip}{,}}
\def\eqnDimB{\dim_\mathbb{C}(\mathbb{F}\Lambda_1^{\deform})_{\ab}\opt{ip}{&}=1\opt{ip}{,}}
\def\eqnDimC{\dim_\mathbb{C}(\mathbb{F}\Lambda_2^{\deform})_{\ab}\opt{ip}{&}=3.}
\opt{ams}{$\eqnDimA$, $\eqnDimB$, and $\eqnDimC$}\opt{ip}{\begin{align*}\eqnDimA \\ \eqnDimB \\ \eqnDimC\end{align*}}    

Repeating this calculation over all lengths $\ell$ gives all columns in the \opt{ams}{table except the last}\opt{ip}{first table}, and furthermore a presentation for each algebra $\mathbb{F}\Lambda_i^{\deform}$.  We claim that  $\Lambda_i^{\deform}$ is not commutative if and only if the quiver of $\mathbb{F}\Lambda_i^{\deform}$ has two loops.  The \opt{ams}{last column in the table}\opt{ip}{second table} then follows\opt{ams}{ from the first column}.

The ($\Leftarrow$) direction is a calculation.  Indeed, by the above knitting calculation (repeated for all lengths),  if the quiver of  $\mathbb{F}\Lambda_i^{\deform}$ has two loops then, by its explicit presentation in each case, we observe that $\mathbb{F}\Lambda_i^{\deform}$ is not commutative.  Hence since $\mathbb{F}\Lambda_i^{\deform}$ is a factor of $\Lambda_i^{\deform}$, certainly $\Lambda_i^{\deform}$ is not commutative.

For the ($\Rightarrow$) direction, we show the contrapositive.  Suppose that the quiver of $\mathbb{F}\Lambda_i^{\deform}$ does not have two loops.  By the \opt{ams}{above}\opt{ip}{first} table, it must have $\leq 1$ loops, and so there is an isomorphism $\mathbb{F}\Lambda_i^{\deform}\cong\mathbb{C}[y]/y^n$ for some $y$ and some $n$.   But since $\Lambda_i^{\deform}$ is local, and $g$ is not a unit, $g$ belongs to the radical $J$ of $\Lambda_i^{\deform}$. Further, every element of $A$ can be written in the form $1+ya_1+ga_2$ for some $a_1,a_2\in A$, and $J=yA+gA$.  But then $y$ and $g$ generate the finite dimensional algebra $\Lambda_i^{\deform}$, and so in particular $\Lambda_i^{\deform}$ is commutative.
\end{proof}

The following asserts the conditions under which noncommutative deformation theory is necessary, and yields additional information.
\begin{cor}\label{when NC needed}
Suppose that $X$ is smooth.  Then the representing object of $\scrS_i$ is not commutative if and only if $\scrS_i$ is, up to line bundle twist, one of the following: 
\begin{enumerate}
\item\label{sph1} $\scrO_{\Curve}$ when $\ell>1$.
\item\label{sph2} $\scrO_{2\Curve}$ when $\ell=4,5,6$.
\item\label{sph3} $\scrO_{3\Curve}$ and $\omega_{3\Curve}$ when $\ell=6$.
\end{enumerate}
Furthermore, $\scrS_i$ can be a spherical object only if $\scrS_i$ is,  up to line bundle twist, one of the following:
\begin{enumerate}[resume]
\item\label{sph4} $\scrO_{\ell\Curve}$ and $\omega_{\ell\Curve}$ in all types.
\item\label{sph5} $\scrZ$ and $\scrZ^\omega$  when $\ell=5$.
\end{enumerate}
The sheaves in \eqref{sph4} and \eqref{sph5} may or may not be spherical, depending on the example.  Indeed, the sheaves in \eqref{sph4} are spherical if and only if the top GV invariant is one.
\end{cor}
\begin{proof}
Accounting for line bundle twists and Perverse--Tilt duality, parts \eqref{sph1}--\eqref{sph3} are then a direct translation of the last column of the \opt{ip}{second }table in \ref{quivers and dims}.   The sheaf $\scrS_i$ is spherical if and only if $\Lambda_i^{\deform}\cong\mathbb{C}$.  Hence, if it is spherical, then certainly the factor $\mathbb{F}\Lambda_i^{\deform}$ also needs to be one-dimensional.  Thus, referring to the \opt{ip}{first }table in \ref{quivers and dims}, parts~\eqref{sph4} and \eqref{sph5} follow.  Finally, since the deformation algebra of $\scrS_1=\scrO_{\ell\Curve}$ is always commutative by \ref{quivers and dims}, its dimension equals the dimension of its abelianisation.  Hence the last statement is just \ref{comm defs rep}.
\end{proof}

\begin{remark}\label{higher mult def}
The above gives more evidence that noncommutative deformations of $\scrO_{a\Curve}$ controls how both $a\Curve$ and all its higher multiples $n(a\Curve) \colonequals (na)\Curve$ move. Indeed, by \ref{when NC needed} above, strictly noncommutative deformations of $\scrO_{a\Curve}$ exist if and only if  $2a\leq \ell$, which is if and only if higher multiples $n(a\Curve)$ of $a\Curve$ exist.
\end{remark}

\subsection{Gopakumar--Vafa corollaries}
The table below gives the first non-trivial lower bounds on GV invariants. The final column also greatly improves on results from both \cite[3.17]{DW1} and \cite[1.4]{Toda GV}; see also \ref{current bounds known} below.

\begin{cor}\label{GV bounds}
Lower bounds for GV invariants and the dimension of the contraction algebra are as follows.
\[
\begin{tabular}{C{0.6cm}lc}
\toprule
$\ell$&\textnormal{GV lower bound}&$\dim\CA$ \textnormal{lower bound}\\
\midrule
$1$ & $(1)$& $1$\\
$2$ & $(4,1)$& $8$\\
$3$ & $(5,3,1)$& $26$\\
$4$ & $(6,4,2,1)$& $56$\\
$5$ & $(7,6,4,2,1)$& $124$\\
$6$ &$(6,6,4,3,2,1)$& $200$\\
\bottomrule
\end{tabular}
\]
\end{cor}
\begin{proof}
First recall from \ref{comm defs rep} that the GV invariants $\GV{i}$ can be obtained as $
\dim_{\mathbb{C}}(\Lambda_j^{\deform})_{\ab}$ for appropriate~$j$.

The case $\ell=1$ is clear.  When $\ell=2$, since the normal bundle of $\Curve$ is $(-3,1)$, by \cite[2.15]{HomMMP} there are two loops in the quiver of $\Lambda_0^{\deform}$.  Hence the abelianisation is at least $4$-dimensional, being a factor of $\mathbb{C}[[x,y]]$ modulo two relations, both quadratic or higher.  This proves that $\GV{1}\geq 4$. The fact that $\GV{2}\geq 1$ is obvious, given that $\scrS_1=\scrO_{2\Curve}$ is a non-zero object.  For all other cases, the dimension of a factor of the abelianisation is the column labelled $\dim \mathbb{F}(\Lambda_i^{\deform})_{\ab}$ in \ref{quivers and dims}. Obviously this gives  a lower bound for the dimension of the abelianisation.  The result then follows, just by matching the first $N/2$ members of the simples helix with the data in \ref{quivers and dims}, where recall that the simples helix begins either $\scrO_{\Curve}(-1),\scrO_{\ell\Curve},\hdots,\scrO_{2\Curve}$ when $\ell\leq 4$, or $\scrO_{\Curve}(-1),\scrO_{\ell\Curve},\hdots,\scrO_{3\Curve},\scrZ,\scrO_{2\Curve}$ when $\ell=5,6$.
\end{proof}

\begin{remark}\label{current bounds known}
It is known that the lower bounds in \ref{GV bounds} are obtained when $\ell=1$ or $\ell=2$.  From \cite{BW} it is known that the lower bound is also achieved for $\ell=6$.  The current lowest known GV invariants from \cite{BW} when $\ell=3,4,5$ are $(6,3,1)$, $(6,5,2,1)$, and $(8,6,4,2,1)$, which have contraction algebras of dimension $27$, $60$, and $125$ respectively.  Thus, the bounds in \ref{GV bounds} are very close to the current minimum. 
\end{remark}

\subsection{Global Equivalences} 
Deformation theory is analytic, but our results in the previous sections have global consequences.  Suppose that $f\colon Y\to Y_{\con}$ is a flopping contraction of quasi-projective $3$-folds, where $Y$ has at worst Gorenstein terminal singularities in a neighbourhood of the flopping curves.  Consider the exceptional fibre, which with its reduced scheme structure consists of a finite collection of curves, each isomorphic to $\mathbb{P}^1$.   Choose such a curve $\Curve$, and as notation, consider an affine open subset $U_{\con}=\Spec R$ around the point $f(\Curve)$, and set $U\colonequals f^{-1}(U_{\con})$.  Taking the formal fibre then gives the following commutative diagram.
\[
\begin{array}{c}
\begin{tikzpicture}
\node (U) at (1,0) {$U$};
\node (Uhat) at (-1,0) {$\mathfrak{U}$};
\node (X) at (3,0) {$Y$};
\node (Uc) at (1,-1.5) {$\Spec R$};
\node (Rhat) at (-1,-1.5) {$\Spec \mathfrak{R}$};
\node (Xc) at (3,-1.5) {$Y_{\con}$};
\draw[right hook->] (U) to (X);
\draw[->] (Uhat) to (U);
\draw[->] (Uhat) to (Rhat);
\draw[->] (Rhat) to (Uc);
\draw[right hook->] (Uc) to (Xc); 
\draw[->] (X) --  node[right] {$\scriptstyle f$} (Xc);
\draw[->] (U) --  node[right] {$\scriptstyle f|_U$}  (Uc);
\end{tikzpicture}
\end{array}
\]
The curve $\Curve$ can be contracted individually in $\mathfrak{U}$, and as such individually it has an associated length $\ell$, and thus has an associated simples helix $\{\scrS_i\}_{i\in\mathbb{Z}}$ in $\mathfrak{U}$.  We then view this simples helix on both $U$ and $Y$ via pushforward, but suppress this from the notation. 

\begin{lemma}\label{global rep}
With notation as above, there is a chain of isomorphisms
\[
\Def_Y^{\scrS_i}
\cong
\Def_U^{\scrS_i}
\cong
\Def_{\mathfrak{U}}^{\scrS_i}
\cong
\Hom_{\art_1}(\Lambda_i^{\deform},-).
\]
\end{lemma}
\begin{proof}
This is standard; see e.g.\ \cite[(2.G)]{DW3}.
\end{proof}
Thus by \ref{global rep} the deformation functor associated to each sheaf $\scrS_i\in\coh Y$ is representable, and exactly as in \cite[3.3]{DW1} the resulting universal sheaf is given as the pushforward of the universal sheaf on $\mathfrak{U}$.  Again, we suppress this from the notation, and simply write $\scrE_i\in\coh Y$ for the universal sheaf on $Y$.   Exactly as in \cite[3.9]{DW1},  $\End_Y(\scrE_i)\cong \Lambda_{i}^{\deform}$.

\begin{thm}\label{simples global auto}
Suppose that $Y\to Y_{\con}$ is a flopping contraction of quasi-projective \mbox{$3$-folds}, where $Y$ has at worst Gorenstein terminal singularities.  For every contracted curve~$\Curve$ with reduced scheme structure, there is an associated length $\ell$, and simples helix $\{\scrS_i\}_{i\in\mathbb{Z}}$ viewed on $Y$. Then for all $k\in\mathbb{Z}$, there is a global autoequivalence
\[
\twistGen_{\scrS_k}\colon\Db(\coh Y)\to\Db(\coh Y),
\]
and the following statements hold.
\begin{enumerate}
\item\label{simples global auto 1} $\scrE_k\in\Perf(Y)$.
\item\label{simples global auto 2} There is a functorial triangle
\[
\RHom_Y(\scrE_{k},-)\otimes_{\Lambda_k^{\deform}}^{\bf L}\scrE_{k}\to \Id\to\twistGen_{\scrS_k}\to.
\]
\end{enumerate}
\end{thm}
\begin{proof}
Part \eqref{simples global auto 1} follows exactly as in \cite[7.1]{DW1}.  Indeed, complete locally $\Lambda_k^{\deform}$ has finite projective dimension, since by \eqref{exchange 1} and \eqref{exchange 2} mutation twice returns us to our original module, so we can simply use the analogue of \cite[(5.B)]{DW1}.  Hence $\Lambda_k^{\deform}$ also has finite projective dimension Zariski locally, since projective dimension can be calculated complete locally, and $\Lambda_k^{\deform}$ is supported only at one point.  This translates into $\scrE_k$ being perfect on~$U$, and thus its pushforward is perfect on~$Y$.

There is a functor $\twistGen_{\scrS_k}^*$ which is constructed in a word-for-word identical manner to \cite[\S5]{DW3}, except that the tilting bundle $\scrP_k$ at all stages replaces the tilting bundle used in \cite{DW3}, which is $\scrP_0$ in the notation here. The key point is that $\scrE_k$ is still perfect, by~\eqref{simples global auto 1}.  Exactly the same spanning class argument in \cite[5.22]{DW3} shows that a functor $\twistGen_{\scrS_k}^*$ is an equivalence, and $\twistGen_{\scrS_k}$ is then defined to be its inverse.  The functorial triangle in part~\eqref{simples global auto 2} is again very similar: the bundle $\scrP_k$ just replaces $\scrP_0$, and the analogue of  \cite[5.23]{DW3} holds.
\end{proof}

\begin{remark}
The above establishes that all members $\scrS_i$ of the simples helix, although they need not be perfect since $X$ is singular, do deform to give a universal sheaf which is a perfect complex.
\end{remark}

\end{document}